\documentclass[12pt,a4paper]{article}
\usepackage{amsmath}
\usepackage{amsfonts}
\usepackage{amssymb}
\usepackage{amsthm}
\usepackage{color}

\newcommand{\ga}{\alpha}\newcommand{\gb}{\beta}
\newcommand{\gga}{\gamma}\newcommand{\gG}{\Gamma}
\newcommand{\gd}{\delta}\newcommand{\gD}{\Delta}
\newcommand{\gep}{\varepsilon}\newcommand{\gz}{\zeta}

\newcommand{\gl}{\lambda}\newcommand{\gL}{\Lambda}
\newcommand{\gr}{\rho}
\newcommand{\gs}{\sigma}
\newcommand{\gp}{\varphi}\newcommand{\gP}{\Phi}
\newcommand{\gO}{\Omega}
\newcommand{\gh}{\eta}
\newcommand{\cB}{\mathcal{B}}\newcommand{\cF}{\mathcal{F}}
\newcommand{\cP}{\mathcal{P}}\newcommand{\cQ}{\mathcal{Q}}
\newcommand{\cR}{\mathcal{R}}\newcommand{\cE}{\mathcal{E}}
\newcommand{\cI}{\mathcal{I}}
\newcommand{\cK}{\mathcal{K}}\newcommand{\cC}{\mathcal{C}}
\newcommand{\cM}{\mathcal{M}}\newcommand{\cN}{\mathcal{N}}
\newcommand{\cJ}{\mathcal{J}}\newcommand{\cL}{\mathcal{L}}
\newcommand{\cS}{\mathcal{S}}\newcommand{\cD}{\mathcal{D}}
\newcommand{\cU}{\mathcal{U}}\newcommand{\cA}{\mathcal{A}}
\newcommand{\cV}{\mathcal{V}}
\newcommand{\cT}{\mathcal{T}}\newcommand{\cG}{\mathcal{G}}

\newcommand{\bN}{\mathbb{N}}
\newcommand{\bR}{\mathbb{R}}
\newcommand{\bT}{\mathbb{T}}
\newcommand{\bZ}{\mathbb{Z}}
\newcommand{\bE}{\mathbb{E}}
\newcommand{\bP}{\mathbb{P}}

\theoremstyle{definition}
\newtheorem{definition}{Definition}[section]

\theoremstyle{plain}
\newtheorem{theorem}[definition]{Theorem}
\newtheorem{proposition}[definition]{Proposition}
\newtheorem{lemma}[definition]{Lemma}
\newtheorem{corollary}[definition]{Corollary}
\theoremstyle{remark}
\newtheorem{remark}[definition]{Remark}

\usepackage{fullpage}
\usepackage{hyperref}
\usepackage[title]{appendix}
\usepackage{booktabs}
\numberwithin{equation}{section}
\usepackage[english]{babel}
\usepackage{tikz}
\usetikzlibrary{trees}
\tikzset{every node/.style={circle, draw=black, fill, inner sep=0.1ex, grow=up}, level/.style={level distance=1.2ex, sibling distance=0.7em)}, grow'=up, noise/.style={edge from parent/.style={draw, densely dotted}}, derivative/.style={edge from parent/.style={draw,double}}, neutral/.style={edge from parent/.style={draw}}}
\newcommand{\abs}[1]{\vert #1 \vert}
\newcommand{\norm}[1]{\Vert #1 \Vert}
\newcommand{\f}[1]{\mathfrak{#1}}

\newcommand{\1}{\mathbf{1}}
\newcommand{\X}{\mathbf{X}}

\newcommand{\z}{\bar{z}}

\newcommand{\PI}{\mathbf{\Pi}}
\newcommand{\hPI}{\hat{\mathbf{\Pi}}}
\newcommand{\hPi}{\hat{\Pi}}
\newcommand{\hgG}{\hat{\Gamma}}

\title{An It\^o type formula for \\ the additive stochastic heat equation}\author{Carlo Bellingeri \thanks{Laboratoire de Probabilités Statistique et Modélisation}}
\date{}
\begin{document}
\maketitle
\begin{abstract}
We use the theory of regularity structures to develop an It\^o formula for $u$, the solution of the one-dimensional stochastic heat equation driven by space-time white noise with periodic boundary conditions. In particular, for any smooth enough function $\gp$ we can express the random distribution $(\partial_t-\partial_{xx})\gp(u)$ and the random field $\gp(u)$ in terms of the reconstruction of some modelled distributions. The resulting objects are then identified with some classical constructions of stochastic calculus.
\end{abstract}
\tableofcontents
\section{Introduction}
We consider $\{u(t,x)\colon t\in [0,T],\; x\in \bT=\bR/\bZ\}$ the solution of the additive stochastic heat equation with periodic boundary conditions and zero initial value:
\begin{equation}\label{eqSHE}
\begin{cases}
\partial_t u= \partial_{xx} u +\xi \,,\\
u(0,x)=0  \\
u(t,0)=u(t,1)\\
\partial_xu(t,0)=\partial_x u(t,1)
\end{cases}
\end{equation}
where $\xi$ is a space-time white noise over $\bR\times \bT$. This equation was originally formulated to model a one-dimensional string exposed to a stochastic force (see \cite{funaki1983}). From a theoretical point of view, the equation \eqref{eqSHE} represents one of the simplest examples of a stochastic PDE whose solution can be written explicitly, the so-called \emph{stochastic convolution} (see e.g. \cite{walsh1984,daprato2014}). Writing $\xi= \partial W/\partial t\partial x$, where $W$ is the Brownian sheet associated to $\xi$, one has
\begin{equation}\label{sto_conv}
u(t,x)= \int_0^t\int_{\bT}P_{t-s}(x-y)dW_{s,y}\,,
\end{equation} 
where the integral $dW_{s,y}$ is a Walsh integral taken with respect to the martingale measure associated to $W$ and $P\colon (0,+\infty)\times \bT\to \bR$ is the fundamental solution of the heat equation with periodic boundary conditions:
\[P_t(x)= \sum_{m\in \bZ}G_t(x+m)\,,\quad G_t(x)= \frac{1}{\sqrt{4\pi t}}\exp\left(-\frac{x^2}{4t}\right)\,.\]

It is well known in the literature (see e.g. \cite{daprato2014}) that $u$ admits a continuous modification in both variables $t$ and $x$ and it satisfies the equation \eqref{eqSHE} in a weak sense, that is for any smooth function $l\colon \bT\to \bR$ one has
\begin{equation}\label{weak_form}
\int_{\bT} u(t,y) l(y) dy= \int_0^t\int_{\bT} u(s,y) l''(y)dy\, ds+ \int_0^t\int_{\bT} l(y)dW_{s,y}\,.
\end{equation}
Looking at $u$ as a process with values in an infinite-dimensional space, the process $u_t=u(t,\cdot) $ is also a Feller Diffusion taking values in $C(\bT)$, the space of periodic continuous functions, and $L^2(\bT)$. Its hitting properties were intensively studied in \cite{Mueller2002} using the potential theory for Markov processes and some general features of Gaussian random fields. Nevertheless, the general extension of the It\^o formula in the infinite-dimensional stochastic calculus (see e.g. \cite{daprato2014}) cannot be applied directly to $u_t$, because the quadratic variation of this process in this setting must coincide with the trace of the identity operator. Moreover, it has been shown in \cite{swanson2007} that the process $t\to u(t,x)$ for any fixed $x\in \bT$ has an a.s. infinite quadratic variation as a real-valued process. Therefore any attempt to apply classically the powerful theory of It\^o calculus seems pointless.

Introduced in 2014 and explained through the famous ``quartet" of articles \cite{Hairer2014,Bruned2019, chandra_analytic16,Bruned_final17}, the theory of regularity structures has provided a very general framework to prove local pathwise existence and uniqueness of a wide family of stochastic PDEs driven by space-time white noise. In this article, we will show how these new techniques allow formulating an It\^o formula for $u$. The formula itself will be expressed under a new form, reflecting the new perspective under which the stochastic PDEs are analysed. Indeed, for any fixed smooth function $\gp\colon \bR\to \bR$, we will study the quantity $(\partial_t - \partial_{xx})\gp(u)$, interpreted as a space-time random distribution. This choice is heuristically motivated by the parabolic form of the equation \eqref{eqSHE} defining $u$ and it is manageable by the regularity structures, where it is possible to manipulate random distributions. Thus we are searching for a random distribution $g_{\gp}$, depending on higher derivatives of $\gp$, such that, denoting by $\langle\cdot,\cdot \rangle$ the duality bracket, one has a.s. the identity
\begin{equation}\label{example_differential}
\langle(\partial_t - \partial_{xx})\gp(u),\psi \rangle =\langle g_{\gp}, \psi\rangle\,,
\end{equation}
for any test function $\psi$.  We will refer to this formula as a \emph{differential It\^o formula}, because of the presence of a differential operator on the left-hand side of \eqref{example_differential}. By uniqueness of the heat equation with a distributional source $g_{\gp}$ (see section \ref{elements}), for every $(t,x)\in [0,T]\times \bT\to \bR$, we can write formally
\begin{equation}\label{example_integral}
\gp(u(t,x))= \gp(0)+ \int_0^t \int_{\bT} P_{t-s}(x-y)g_{\gp}(s,y)\,ds \,dy
\end{equation}
where for any fixed $ (t,x)$ the equality \eqref{example_integral} hold  a.s. We call a similar identity an \emph{integral It\^o formula} because of the double integral on the right-hand side of \eqref{example_integral}. This resulting formula may be one possible tool to improve our comprehension of the trajectories of $u$, even if it is still not clear whether it will be as effective as it is for finite-dimensional diffusions (see e.g. \cite{revuz2004continuous}).

To obtain these identities, we will follow the general philosophy of the regularity structure theory. Instead of working directly with the process $u$, we will consider $\{u_{\gep}\}_{\gep>0}$ an approximating sequence of $u$, solving a so-called ``Wong-Zakai" formulation of \eqref{eqSHE}
\begin{equation}\label{eq3}
\begin{cases}
\partial_t u_{\gep}= \partial_{xx} u_{\gep} +\xi_{\gep}\,,\\
u_{\gep}(0,x)=0  \\
u_{\gep}(t,0)=u_{\gep}(t,1)\\
\partial_xu_{\gep}(t,0)=\partial_x u_{\gep}(t,1)
\end{cases}
\end{equation}
where the random fields $\{\xi_{\gep}\}_{\gep>0}$ are defined by extending $\xi$ periodically on $\bR^2$ and convolving it with a  fixed smooth, compactly supported function $\gr\colon \bR^2\to \bR$ such that $\int_{\bR^2} \gr=1$ and $ \gr(t,x)=\gr(-t,-x)$. That is  denoting by $*$ the convolution on $\bR^2$ for any $\gep>0$ we set 
\begin{equation}\label{rho_eps}
\gr_{\gep}= \gep^{-3}\gr(\gep^{-2}t, \gep^{-1}x)\,, \quad \xi_{\gep}(t,x)= (\gr_{\gep}*\xi)(t,x)\,. 
\end{equation}
The inhomogeneous scaling in the mollification procedure is chosen accordingly to the parabolic nature of the equation \eqref{eqSHE}. This regularisation makes $\xi_{\gep}$ an a.s. smooth function $\xi_{\gep} \colon [0,T]\times \bT\to \bR$ and the equation \eqref{eq3} admits for any $\gep>0$ an a.s. periodic strong solution (in the analytical sense) $u_{\gep}\colon [0,T]\times \bT\to \bR$ which is smooth in space and time. Therefore in this case $(\partial_t - \partial_{xx})\gp(u_{\gep})$ is calculated by applying the classical chain rule between $u_{\gep}$ and $\gp$, obtaining
\begin{equation}\label{chain1}
\partial_t(\gp(u_{\gep}))=\gp'(u_{\gep})\partial_t u_{\gep}\,,\quad \partial_x (\gp(u_{\gep}))=\gp'(u_{\gep})\partial_x u_{\gep}\,,
\end{equation}
\begin{equation}\label{chain2}
\partial_{xx} (\gp(u_{\gep}))=\gp''(u_{\gep})(\partial_x u_{\gep})^2+\gp'(u_{\gep})\partial_{xx}u_{\gep} \,.
\end{equation}
Thereby yielding
\begin{equation}\label{parabolic_phi}
(\partial_t - \partial_{xx})\gp(u_{\gep})=\gp'(u_{\gep})\xi_{\gep}-\gp''(u_{\gep})(\partial_x u_{\gep})^2\,.
\end{equation}
Let us understand heuristically what happens when $\gep\to 0^+$. Since $\gr_{\gep}$ is an approximation of the delta function, $u$ is a.s. continuous and the derivative is a continuous operation between distributions, we can reasonably infer that the left-hand side of \eqref{parabolic_phi} converges in some sense to $ (\partial_t- \partial_{xx} )\gp(u)$. Thus the right-hand side of \eqref{parabolic_phi} should converge too to some limit distribution. However, written under this form, it is very hard to study this right-hand side because it is possible to show 
\[
\norm{\gp'(u_{\gep})\xi_{\gep}}\overset{\bP}{\to}  +\infty\,,\quad\norm{\gp''(u_{\gep})(\partial_x u_{\gep})^2}\overset{\bP}{\to}  +\infty\]
with respect to some norm in an infinite-dimensional space (see the remark \ref{div_proba}). These two results suggest a cancellation phenomenon between two objects whose divergences compensate each other.
This simple cancellation phenomenon between two diverging random quantities, which lies at the heart of the recent study of singular SPDEs, has already been noticed in the pioneering article \cite{Zambotti2006} (see also \cite{gradinaru05,Lanconelli2007}) and  now we are able to reinterpret that result in the general context of the renormalization theory, as explained in the theory of regularity structures. Using the notion of modelled distribution and the reconstruction theorem, we can also explain the limit as the difference of two explicit random distributions. However, these limits are only characterised by some analytical properties which cannot allow to understand immediately their probabilistic representation. Therefore the convergence is also linked with some specific \emph{identification theorems} which describe their law. Summing up both these results we can state the main theorem of the article:
\begin{theorem}[Integral and differential It\^o formula]\label{Integral_Ito}
Let $\gp$ be a function of class $C^{4}_b(\bR)$, the space of $C^4$ functions with all its derivatives bounded. Then for any test function $\psi\colon \bR\times \bT \to \bR$ with supp $(\psi) \subset (0,T) \times \bT$, one has the a.s. equality
\[\begin{split}
&\langle(\partial_t - \partial_{xx})\gp(u),\psi \rangle =\\&\int_{0}^{T}\int_{\bT} \gp'(u_s(y))\psi(s,y)dW_{s,y}+ \frac{1}{2}\int_{0}^{T}\int_{\bT}\psi(s,y)\gp''(u_s(y)) C(s)dy ds\\&-\int_{[0,T]^2\times \bT^2}\left[\int_{\mathbf{s}_2\vee \mathbf{s}_1}^T \int_{\bT}\psi(s,y)\gp''(u_s(y))P'_{s-\mathbf{s}_1}( y- \mathbf{y}_1)P'_{s-\mathbf{s}_2}( y- \mathbf{y}_2) dyds \right]dW^{2}_{\mathbf{s},\mathbf{y}}\,.
\end{split}\]
Moreover for any $(t,x)\in [0,T]\times \bT$ we have a.s.
\[
\begin{split}
&\gp(u_t(x))=\\&\gp(0)+\int_{0}^{t}\int_{\bT}  P_{t-s}(x-y)\gp'(u_s(y)) dW_{s,y} +\frac{1}{2}\int_{0}^{t}\int_{\bT}P_{t-s}( x- y)\gp''(u_s(y)) C(s) dy ds\\&-\int_{[0,t]^2\times \bT^2}\left[\int_{\mathbf{s}_2\vee \mathbf{s}_1}^t \int_{\bT}P_{t-s}(x-y)\gp''(u_s(y))P'_{s-\mathbf{s}_1}( y- \mathbf{y}_1)P'_{s-\mathbf{s}_2}( y- \mathbf{y}_2) dyds \right]dW^{2}_{\mathbf{s},\mathbf{y}},
\end{split}
\]
where in both formulae $P'_s(y)= \partial_x P_s(x)$, the integral $dW^{2}_{\mathbf{s},\mathbf{y}} $ is the multiple Skorohod integral of order two integrating the variables $\mathbf{s}=(\mathbf{s}_1,\mathbf{s}_2)$ and $\mathbf{y}=(\mathbf{y}_1\,,\mathbf{y}_2)$,  $u_s(y)=u(s,y)$ and $C\colon(0,T)\to \bR$ is the integrable function $C(s):=\norm{ P_s( \cdot)}^2_{L^2(\bT)}$.
\end{theorem}
\begin{remark}
It is natural to ask whether the same techniques could be applied to a generic convex function $\gp$, to establish a Tanaka formula for $u$. In case $\gp$ is not a regular function, the formalism of regularity structures does not work anymore (see the section \ref{calculus_reg_struct}). However even if we try to generalise the Theorem \ref{Integral_Ito} using only the Malliavin calculus, the presence of a multiple Skorohod integral of order $2$ in both formulae would require apriori that the random variable $\gp''(u(s,y))$ ought to be twice differentiable (in the Malliavin sense). Hence the condition $\gp\in C^4_b(\bR)$ in the statement appears to be optimal. Finally, any Tanaka formula would require a robust theory of local times associated to $u$, yet this notion is very ambiguous in the literature. On the one hand, using some general results on Gaussian variables (such as \cite{geman1980})  for any $x\in \bT$  we can prove the existence of a local time with respect to the occupation measure of the process $t\to u(t,x)$. On the other hand, an alternative notion of a local time for $u$ has been developed employing distributions on the Wiener space in \cite{gradinaru05}.
\end{remark}
We discuss the organization of the paper: in Section \ref{abstract_reg_struct} and \ref{calculus_reg_struct} we will apply the general theorems of the regularity structures theory to build the analytical and algebraic tools to study the problem: all the constructions are mostly self-contained. In some cases we will also recall some previous results obtained in \cite{MartinHairer2015, Bruned2019,chandra_analytic16}. Then in the section \ref{ito_form} we will combine all these tools to obtain firstly two formulae involving only objects built in the previous sections (we will refer to them as \emph{pathwise It\^o formulae}) and later the identifications theorems.

We finally remark that some of the techniques presented here could potentially be used to establish an It\^o formula for a non-linear perturbation of \eqref{eqSHE}, the so-called generalised KPZ equation:
\begin{equation}\label{gKPZeq}
\begin{cases}
\partial_t u= \partial_{xx} u + g(u)(\partial_x u)^2 +  h(u)(\partial_x u)+k(u) +f(u)\xi\,,\\
u(0,x)=u_0(x)\\
u(t,0)=u(t,1)\\
 \partial_xu(t,0)=\partial_x u(t,1)\,
\end{cases}
\end{equation}
where $g$, $f$, $h$, $k$ are smooth functions and $u_0\in C(\bT)$ is a generic initial condition. (We refer the reader to \cite{Bruned2015,hairer_string16,brunedSHE2019}). Establishing such a formula in this generalized setting shall be subject to further investigations. Other possible directions of research may also take into account the Itô formula for the solutions of other stochastic PDEs with Dirichlet boundary conditions (see \cite{Gerencser2019}) and, using the reformulation in the regularity structures context of differential equations driven by a fractional Brownian motion (see \cite{fritz_bayer17}), we could recover some classical results for fractional processes (see e.g. \cite{flandoli2002,Russo1993}).
\subsection*{Acknowledgements}
The author is especially grateful to Lorenzo Zambotti, Cyril Labbé, Nikolas Tapia, Henri Elad Altman and Yvain Bruned for many enlightening discussions, comments and suggestions concerning the content and the organization of the present article. A special thank goes also to the organisers of the Second Haifa Probability School and the weekly seminar in SPDEs and rough path based in the greater Berlin area, where a preliminary version of these results was presented. Finally, the author is also very thankful for the very careful referees' reports, which led to many improvements and corrections during the revision of the article.

\section{H\"older spaces and Malliavin calculus}\label{elements}
We recall here some preliminary notions and notations we will use throughout the chapter. For any space-time variable $z=(t,x)\in \bR^2$, in order to preserve the different role of time and space in the parabolic equation \eqref{eqSHE} we define, with an abused notation, its parabolic norm as
\[\norm{z}:= \sqrt{\abs{t}}+ \abs{x}\,.\]
Moreover for any multi-index $k=(k_1,k_2)$  the parabolic degree of $k$ is given by $\abs{k}:=2k_1+k_2$ and we adopt the multinomial notation for monomials $z^k= t^{k_1}x^{k_2}$ and derivatives $\partial^k= \partial_t^{k_1}\partial_{x}^{k_2}$\footnote{the derivative $\partial_x^2 f$ will be denoted in some cases  by  $ \partial_{xx} f$ to shorten the notation}.  According to the definition of $\gr_{\gep}$ in \eqref{rho_eps}, the parabolic rescaling of any function $\gh\colon \bR^2 \to \bR$ of parameter $\gl>0$ and centred at $z=(t,x)$ is given by
\[\gh_z^{\gl}(\z):= \gl^{-3}\gh(\frac{\bar{t}-t}{\gl^2},\frac{\bar{x}-x}{\gl})\,,\quad \z=(\bar{t},\bar{x})\,.\]
For any non integer $\ga\in \bR$, a function $f\colon \bR^2 \to \bR $ belongs to the $\ga$ H\"older space $\cC^{\ga}$ when one of these conditions is verified:
\begin{itemize}
\item If $0<\ga<1$, $f$ is continuous and for any compact set $\cK\subset \bR^2$ 
\[ \norm{f}_{\cC^{\ga}(\cK)}:= \sup_{z\in \cK} \abs{f(z)}+ \sup_{\substack{z,w\in \cK\\z\neq w}}\frac{\vert f(z)-f(w)\vert}{\Vert z-w\Vert^{\ga} }<\infty\,.\]
\item If $\ga>1$, $f$ has $ \lfloor \ga\rfloor$ continuous derivative in space and $\lfloor \ga/2\rfloor $ continuous derivative in time, where $ \lfloor \cdot\rfloor$ is the integer part of a real number. Moreover for any compact set $\cK\subset \bR^2$
\[ \norm{f}_{\cC^{\ga}(\cK)}:= \sup_{z\in \cK} \sup_{\abs{k}\leq \lfloor \ga\rfloor}\abs{\partial^kf(z)}+ \sup_{\substack{z,w\in \cK\\z\neq w}}\sup_{\abs{k}= \lfloor \ga\rfloor}\frac{\vert \partial^k f(z)-\partial^k f(w)\vert}{\Vert z-w\Vert^{\ga-\lfloor \ga\rfloor} }<\infty\,.\]
\item If $\ga<0$, $f$ is an element of $\cS'(\bR^2)$, the set of tempered distributions on $\bR^2$ and at the same time $f$ belongs to the dual of $\cC^{r}_0(\bR^2)$, the set of compactly supported functions belonging to $\cC^{r}(\bR^2) $, where $r= - \lfloor \ga\rfloor+1$. Moreover for every compact set $\cK\subset\bR^2$
\[ \norm{f}_{\cC^{\ga}(\cK)}:=\sup_{z\in \cK}\sup_{\gh\in \cB_r}\sup_{\gl\in (0,1]}\frac{\vert\langle f,\eta^{\gl}_z\rangle\vert}{ \gl^{\ga}}<\infty \,,\]
where $\cB_r$ is the set of all test functions $\gh$ supported on $\{z\in \bR^{2}\colon \norm{z}\leq 1\}$, such that all the directional derivatives up to order $r$ are bounded in the sup norm.
\end{itemize}
The spaces $\cC^{\ga}$ and the respective localised version $\cC^{\ga}(D)$, defined on every open set $D\subset\bR^2$ are naturally a family of Fréchet spaces. Moreover for any $\ga>0$ and any compact set $\cK\subset \bR^2$, defining $\cC^{\ga}(\cK)$ by restriction of $f$ on $\cK$, we obtain a Banach space using the quantity $\norm{f}_{\cC^{\ga}(\cK)}$. The elements $f\in \cC^{\ga}(\bR\times \bT)$ are interpreted as elements of $\cC^{\ga}$ whose space variable lives in $\bT$. Most of the classical analytical operations apply to the $\cC^{\ga}$ spaces as follows:
\begin{itemize}
\item \emph{Differentiation} if $f\in \cC^{\ga}$ and $k$ is a multi-index then the map $f\to \partial^kf$ is a continuous map from $\cC^{\ga}$ to $ \cC^{\gb}$ where $\gb= \ga-\abs{k}$.
\item \emph{Schauder estimates} (see \cite{Simon1997}) if $P$ is the Heat kernel on some domain, then the space-time convolution with $P$, $f\to P*f$ is a well defined map for every f supported on positive times and it sends continuously $\cC^{\ga}$ in $\cC^{\ga+2}$ for every real non integer $\ga$.
\item \emph{Product} (see \cite[Prop. 4.14]{Hairer2014}) for any real non integer $\gb$ the map $(f,g)\to f\cdot g$ defined over smooth functions extends \emph{continuously} to a bilinear map $B\colon \cC^{\ga} \times \cC^{\gb}\to  \cC^{\ga \wedge \gb}$ if and only if $\ga+\gb>0$.
\end{itemize}
The H\"older spaces and the operations defined on them provide a natural setting to formulate the deterministic PDE
\begin{equation}\label{deterministic_equation}
\begin{cases}
\partial_t v- \partial_{xx}v = g\,,\\
v(0, x)= v_0(x)\\
v(t,0)=v(t,1)\\
\partial_xv(t,0)=\partial_x v(t,1)\,,
\end{cases}
\end{equation}
where $g\in \cC^{\gb}(\bR\times \bT)$ and $v_0\in  \cC(\bT)$. For any $\gb>0$, classical results on PDE theory (see e.g. \cite{krylov1996lectures}) imply that there exists a unique strong solution $v\in\cC^{\gb+2}([0,T]\times \bT)$ of \eqref{deterministic_equation} which is given explicitly by the so-called \emph{variation of the constant formula}
\begin{equation}\label{explicit_heat_equation}
v(t,x)= \int_{\bT}P_t(x-y)v_0(y) dy+ (P*\1_{[0,t]}\,g) (t,x)\,,
\end{equation}
where for any $t>0$, $\1_{[0,t]}$ is the indicator function of the interval $[0,t]\times \bT$. Furthermore if we consider $\gb\in (-2,0)$ non-integer, the equation \eqref{deterministic_equation} admits again a unique solution $v\in\cC^{\gb+2}([0,T]\times \bT)$ satisfying \eqref{deterministic_equation} but only in a distributional sense. This solution can be expressed again by the formula \eqref{explicit_heat_equation} by interpreting $\1_{[0,t]}$ as a continuous linear map $\1_{[0,t]}\colon  \cC^{\gb}(\bR\times \bT)\to \cC^{\gb}(\bR\times \bT)$ such that $(\1_{[0,t]}g)(\psi)= g(\psi)$ for any smooth test function $\psi$ satisfying supp$(\psi)\subset [0,t]\times \bT$ and $(\1_{[0,t]}g)(\psi)=0$ if supp$(\psi)\cap [0,t]\times \bT= \emptyset$ (see \cite[Lem. 6.1]{MartinHairer2015}). In particular, it is possible to show (see  \cite[Prop. 6.9]{Hairer2014} and \cite[Prop. 2.15]{Gerencser2019}) that for any test function $\psi$  
\begin{equation}\label{indicator}
(\1_{[0,t]}g)(\psi) =\lim_{N\to +\infty} g(\gp_N \psi)\,, 
\end{equation}
where $\gp_N $ is a fixed sequence of smooth functions converging a.e. to $\1_{(0,t)\times \bT}$. Thus the solution of the equation \eqref{deterministic_equation} is given by the same formula \eqref{explicit_heat_equation} if $g\in \cC^{\gb}(\bR\times \bT)$. The following procedure can be adapted straightforwardly to define a linear map $\1_{[s,t]}\colon  \cC^{\gb}(\bR\times \bT)\to \cC^{\gb}(\bR\times \bT)$ for any interval $[s,t]\subset \bR$ eventually unbounded.
 
The equation \eqref{eqSHE} can be expressed in the context of the spaces $\cC^{\ga}$. Indeed for every $\kappa>0$ it is possible to show that there exists a modification of $\xi$ belonging to $\cC^{-3/2 - \kappa}(\bR\times \bT)$ such that the sequence $\xi_{\gep}$ defined in \eqref{rho_eps} converges in probability to $\xi$ with respect to the topology of $ \cC^{-3/2 - \kappa}(\bR\times \bT)$ (see \cite[Lem. 10.2]{Hairer2014}). Choosing $\kappa<1/2$ and $v_0=0$, we can apply then the deterministic results of \eqref{deterministic_equation} with every a.s. realisation of $\xi$ and by uniqueness of the solution \eqref{eqSHE} we obtain the pathwise representation
\begin{equation}\label{analytic_representation}
u(t,x)= (P*\1_{[0,t]}\xi)(t,x)\,.
\end{equation}
Using this identity we deduce immediately from the Schauder estimates that every a.s. realisation of $u$ must belong to $ \cC^{1/2-\kappa}([0,T]\times \bT )$. This property excludes immediately the possibility to define an object like ``$u\xi$" because the sum of the H\"older regularity of each factor will be $-1-2\kappa$ and we cannot apply the property of the product stated before. The same reasoning applies also for the formal object ``$(\partial_x u)^2$" and it tells us there is no classical theory to understand these products.

Since we will need to compare distributions  defined on $\bR\times\bT$ with distribution on $\bR^2$, for every function $v\in \cS'(\bR\times \bT)$ we denote by $\widetilde{v}\in  \cS'(\bR^2)$ the  \emph{periodic extension} of $v$, defined for every test function $\psi\colon \bR^2\to \bR$ as
\begin{equation}\label{equation_periodic}
\widetilde{u}(\psi)=u(\sum_{m\in \bZ}\psi(\cdot, \cdot+m))\,.
\end{equation}
This operation extends to the level of distribution  the usual periodic extension of functions defined on $\bR\times\bT$ to $\mathbb{R}^2$ (which we denote by the same notation). Thanks to this definition we have the identities 
\begin{equation}\label{analytic_representation2}
\widetilde{u}(t,x)= (G*\widetilde{\1_{[0,t]}\xi})(t,x)\,, \quad \widetilde{u_{\gep}}(t,x)= (G*\widetilde{\1_{[0,t]}\xi_{\gep}})(t,x)\,.
\end{equation}
From a probabilistic point of view, $\xi$ is an isonormal Gaussian process on $H=L^2(\bR\times \bT)$, defined on a complete probability space $(\Omega, \cF, \bP)$. That is we can associate to any $f\in H$ a real Gaussian random variable $\xi(f)$ such that for any couple $f,g\in H$ one has
\[\bE[\xi(f)\xi(g)]= \int_{\bR}\int_{\bT}f(s,x)g(s,x)ds\,dx\,.\]
We denote by $I_n\colon H^{\otimes n}\to L^{2}(\gO)$, $n\geq 1$ the multiple stochastic Wiener integral with respect to $\xi$ (see \cite{nualart1995malliavin}). $I_n$ is an isometry between the symmetric elements of $H^{\otimes n}$ equipped with the norm $\sqrt{n!}\norm{\cdot}_{H^{\otimes n}}$ and the Wiener chaos of order $n$, the closed linear subspace of $L^2(\gO)$ generated by $\{H_n(\xi(h))\colon \norm{h}_H=1\}$ where $H_n$ is the $n$-th Hermite polynomial. Thus we have the natural identifications $\xi(f)= I_1(f)=\int_{\bR}\int_{\bT}f(s,y)dW_{s,y}$.
Let us introduce some elements of Malliavin calculus for $\xi$ (see \cite{nualart1995malliavin} for a thorough introduction on this subject). We consider $\cS\subset L^2(\gO)$ the set of random variables $F$ of the form 
\[F= h(\xi(f_1),\cdots, \xi(f_n))\,,\]
where $h\colon \bR^n\to \bR$ is a Schwartz test function and $f_1\,,\cdots\,, f_n\in H$. The Malliavin derivative with respect to $\xi$ (see \cite[Def. 1.2.1]{nualart1995malliavin}) is the family of $H$-valued random variables $\nabla F= \{\nabla_z F\colon z\in \bR\times  \bT\}$  defined by
\[\nabla_z F= \sum_{i=1}^n \frac{\partial h}{\partial x_i}(\xi(f_1),\cdots, \xi(f_n))f_i(z)\,.\]
Iterating the procedure and adopting the usual convention $\nabla^0=\text{id}$, for any $k\geq 0$ one can define the $k$-th Malliavin derivative $\nabla^k F= \{\nabla^k_{z_1\cdots z_k} F\colon z_1\,,\cdots \,, z_k \in \bR\times  \bT\} $, which is a family of $H^{\otimes k}$-valued random variables. Moreover starting from a  separable Hilbert space $V$ and considering the random variables $G\in \cS_V$ of the form
\[G= \sum_{i=1}^n F_i v_i\quad F_i\in \cS\,, \quad v_i\in V\,,\]
we can also define the $H^{\otimes k}\otimes V$-random variable $\nabla^k G$. In all of these cases the operator $\nabla^k$ is closable and its domain contains the space $\mathbb{D}^{k,p}(V)$, the closure of $ \cS_V$ with respect to the norm $\norm{\cdot}_{k,p,V}$ defined by 
\begin{equation}\label{malliavin_norm}
\norm{F}_{k,p,V}^p:= \bE[\norm{F}_V^p]+ \sum_{l=1}^k \bE[ \norm{\nabla^l F}_{H^{\otimes l}\otimes V}^p]\,.
\end{equation}
The space $\mathbb{D}^{k,p}(\bR)$ is also denoted by $\mathbb{D}^{k,p} $. Trivially all variables belonging to some finite Wiener chaos are infinitely Malliavin differentiable. We denote by $\gd\colon \text{Dom}(\gd)\subset L^2(\gO; H)\to L^{2}(\gO)$ the adjoint operator of $\nabla$ defined by duality as
\[\bE[\gd (u)F]= \bE[\langle u, \nabla F\rangle_{H}]\] 
for any $u\in \text{Dom}(\gd) $, $ F\in \mathbb{D}^{1,2}$. The operator $\gd$ is known in the literature as the Skorohod integral and for any $u\in \text{Dom}(\gd) $ we will write again $\gd (u)$ with the symbol $\int_{\bR}\int_{\bT}u(s,y)dW_{s,y}$ because $\gd$ is a proper extension of the stochastic integration over a class of non adapted integrands. Using the same procedure we define $\gd^k\colon \text{Dom}(\gd^k)\subset L^2(\gO; H^{\otimes k})\to L^{2}(\gO)$ as the adjoint of $\nabla^k $. Similarly to before we call the operator $\gd^k$ the multiple Skorohod integral of order $k$ and we denote it by  
\[\int_{(\bR\times\bT)^k}u((t_1,x_1)\,,\cdots\,,(t_k,x_k) )dW^k_{t,x}\,.\]
Let us recall the main properties of $\gd^k$:
\begin{itemize}
\item \emph{Extension of the Wiener integral} For any $h\in H^{\otimes k}$, we have  $\gd^k(h)=I_k(h)$.
\item \emph{Product Formula} (see \cite[Lem. 2.1]{Nourdin2010}) Let  $ u\in \text{Dom}(\gd^k)$ be a symmetric function in the variables  $t_1\,,\cdots \,,t_k$ and $F\in \mathbb{D}^{k,2}$. If for any couple of positive integers $j\,,r$ such that $0\leq j+r\leq k$ one has $\langle\nabla^r F, \gd^{j}u\rangle_{H^{\otimes r}}\in L^{2}(\gO; H^{\otimes(k-r- j)})$ then $\langle\nabla^r F, u\rangle_{H^{\otimes r}}\in \text{Dom}(\gd^{q-r})$ and we have
\begin{equation}\label{product_formula}
F\gd^k(u)= \sum_{r=0}^k \binom{k}{r}\gd^{k-r}(\langle\nabla^r F, u\rangle_{H^{\otimes r}})\,.
\end{equation}
\item \emph{Continuity property} (see \cite[Pag. 8]{Nourdin2010}) We have the inclusion $\mathbb{D}^{k,2}(H^{\otimes k})\subset \text{Dom}(\gd^k)$ and the map $\gd^2\colon \mathbb{D}^{k,2}(H^{\otimes k})\to L^2(\gO)$ is continuous. In other words, there exists a constant $C>0$ such that for any $u\in  \mathbb{D}^{k,2}(H^{\otimes k})$ one has
\begin{equation}\label{continuity_property}
\norm{\gd^k(u)}_{L^2(\gO)}\leq C \norm{u}_{\mathbb{D}^{k,2}(H^{\otimes k})}\,.
\end{equation}
\end{itemize}
Extending periodically $\xi$ and the Brownian sheet $W$ to $\bR^2$, we can transfer the Walsh integral and the Skorohod integral to stochastic processes $H\colon \gO\times \bR^2\to \bR$ through the definition:
\begin{equation}\label{stochastic_periodic}
\int_{\bR^2}H(s,y)d\widetilde{W}_{s,y}:= \int_{\bR\times \bT}\sum_{m\in \bZ}H(s,y+m)dW_{s,y}\,,
\end{equation}
as long as the right-hand side above is well defined. Similar definitions hold for the multiple Skorohod integral of order $k$, \emph{mutatis mutandis}.
Using this notation one has
\[\xi_{\gep}(t,x)= \int_{\bR^2}\gr_{\gep}(t-s, x-y)d\widetilde{W}_{s,y}\,,\quad u(t,x)=\int_{0}^t\int_{\bR}G(t-s, x-y)d\widetilde{W}_{s,y}\,. \]


%

\section{Regularity structures}\label{abstract_reg_struct}
In this part we will recall some general concepts of the theory of regularity structures to show the existence of an explicit \emph{regularity structure} and a \emph{model}. These objects will permit to define some analytical operations on $u$. For a quick introduction to the whole theory, we refer the reader to  \cite{Hairer2015}.
\subsection{Algebraic construction}
The starting point of the theory is the notion  of a \emph{regularity structure} $(A,T,G)$, a triple of the following elements:
\begin{itemize}
\item A discrete lower bounded subset $A$ of $\bR$ containing $0$.
\item A graded vector space $T= \bigoplus_{\ga\in A} T_{\ga}$ such that each space $T_{\ga}$ is a Banach space with norm $\norm{\cdot}_{\ga}$ and $T_0$ is generated by a single element $\1$.
\item A group $G$ of linear operators on $T$ such that for each $\ga \in A$, $a$ in $T_{\ga}$ and $\gG$ in $G$, one has $\gG\1=\1$ and 
\begin{equation}\label{reg_tri}
\gG a- a \in \bigoplus_{\gb<\ga}T_{\gb}\,.
\end{equation}
\end{itemize}
Recalling the equations \eqref{parabolic_phi} and \eqref{eq3}, our aim is then to build a regularity structure $\cT$ whose elements are capable to describe for any $\gep>0$ the systems of equations
\begin{equation}\label{referring_equation}
\begin{cases}
\partial_t u_{\gep}= \partial_{xx} u_{\gep} +\xi_{\gep} \\ \partial_t v_{\gep}= \partial_{xx}v_{\gep}+ \gp'(u_{\gep})\xi_{\gep}-\gp''(u_{\gep})(\partial_x u_{\gep})^2 \,.
\end{cases}
\end{equation}
Let us give a first description of $\cT$ in terms of abstract symbols. We start by considering the real polynomials on two indeterminates. For any multi-index $k\in \bN^2$, $k=(k_1,k_2)$ we will write $\X^k$ as a shorthand for the monomial $X_1^{k_1}X_2^{k_2}$ while the unit will be denoted by $\1$. In this way, we will be able to describe smooth functions. At the same time, we introduce an additional abstract symbol $\Xi$ to represent the space-time white noise $\xi$ and for any symbol $\gs$ and $k\in \bN^2$ we introduce a family of symbols $\cI_k(\gs)$ ($\cI_{(0,0)}(\gs)$ is denoted by $\cI(\gs)$) to represent formally the convolution of the $k$-th derivative of the heat kernel with the function associated to the symbol $\gs$. Since $\cI_k(\X^m)$ should be identified with a smooth function, we simply put it to $0$ to avoid repetitions. Finally for any two symbols $\tau_1$ and $\tau_2$ we consider also the juxtaposition of symbols $\tau_1\tau_2$ as a new symbol. To include the product between polynomials as a juxtaposition of symbols we impose that the juxtaposition with $\1$ does not change the symbol and for every multi-index $l,m$ $\X^l\X^m= \X^{l+m}$. We denote also the iterated juxtaposition of the same symbol by an integer power. Adding all these formal rules, we denote by $F$ the set of all possible formal expressions satisfying
\begin{itemize}
\item $ \{\X^k\}_{k\in \bN^2} \cup \{\Xi\}\subset F$.
\item For any $\tau_1,\tau_2\in F$, $\tau_1\tau_2\in F$.
\item For any $\gs\in F$ and $m\in \bN^2$, $\cI_m(\gs)\in F$.
\end{itemize}
We write $\cF$ for the free vector space generated by $F$. Similarly to polynomials, we define a \emph{homogeneity} map $\vert \cdot\vert \colon F \to \bR$ which has approximately the same properties of the degree of polynomials but in the context of the H\"older spaces, as described in Section \ref{elements}. In particular we set recursively:
\begin{itemize}
\item $\abs{\X^k}:=2k_1+ k_2$ the parabolic degree\footnote{By identifying $X_1=t$  and $X_2=x$, $ \abs{\X^k}$ coincides with the parabolic degree $\abs{k}$ defined in Section \ref{elements}.};
\item $\abs{\Xi}:= - 3/2- \kappa\,$ for some fixed parameter $\kappa>0$ ;
\item $\abs{ \cI_k(\tau)}:= \abs{\tau} +2-2k_1 -k_2\;,\;\; \vert \tau\tau'\vert:= \abs{\tau}+\abs{\tau'} $ for any $\tau,\tau'\in F$.
\end{itemize}
Starting from the linear space $\cF$, we introduce a subset of $F$ where we choose all reasonable products that we will need in our calculations. We write $\cI_1(\Xi) $ as shorthand of $\cI_{(0,1)}(\Xi)$.
\begin{definition} \label{defn_reg_symb}
We define the sets of symbols $T, U,  U'\subset F$ as the smallest triple of sets satisfying:
\begin{itemize}
\item $\{\Xi\}\subset T\,,$  $ \{\cI(\Xi)\}\cup \{\X^k\}_{k\in \bN^2}\subset U\,,  \{\cI_1(\Xi)\}\cup \{\X^k\}_{k\in \bN^2}\subset U'$;
\item for every $k\geq 0$ and any finite family of elements $ \tau_1,\ldots, \tau_k\in U$ and any couple of elements $\gs_1,\gs_2\in U'$, then $\{\tau, \tau \Xi, \tau\gs_1,\tau \gs_1 \gs_2\} \subset T $ and $ \tau_1 \ldots \tau_k \in U$.
\end{itemize}
We denote also by $\cT$ and $\cU$ respectively the free vector spaces upon $T$ and $U$.
\end{definition}
The definition of $T$ has an equivalent description in terms of symbols. Defining $V= \{\cI(\Xi)^m\X^l\colon m\in \bN\,,\; l \in \bN^{2} \}$ and for any $\gs\in \{\Xi, \cI_1(\Xi), \cI_1(\Xi)^2 \}$ $V_{\gs}:= \gs V $ the set of all symbols of the form $\gs$ times an element of $V$, it is straightforward to show the identities
\begin{equation}\label{decomposition_T}
U=V\,,\quad T=V_{\Xi}\sqcup V_{\cI_1(\Xi)^2}\sqcup V_{\cI_1(\Xi)}\sqcup V\,.
\end{equation}
Therefore, denoting by $\cV_{\gs}$ the free vector space generated upon $V_{\gs}$, we have the decomposition $
\cT=\cV_{\Xi}  \oplus \cV_{\cI_1(\Xi)^2} \oplus  \cV_{\cI_1(\Xi)}\oplus\cU$. Let us give the construction of the structure group associated to $\cT$. For any $h\in \bR^{3}$, $h=(h_1,h_2,h_3)$ we define the function $\gG_h\colon \cT\to\cT$ as the unique linear map such that
\begin{equation}\label{explicit_gamma}
\gG_h(\gs \cI(\Xi)^m\X^l):= \gs [(X_1+ h_1\1)^{l_1}(X_2+ h_2\1)^{l_2}(\cI(\Xi)+ h_3\1)^m ]\,,
\end{equation}
for any $\gs\in \{\Xi, \cI_1(\Xi), \cI_1(\Xi)^2 , \1\}$, $m\in \bN,\; l\in \bN^2$. Using this explicit definition it is straightforward to show
\begin{equation}\label{algebraic_Gamma}
\gG_h\gG_k= \gG_{h+k}
\end{equation}
for any $h,k\in \bR^3$ and the map $h\to \gG_h$ is injective. We will denote by  $\cG$ the group $\{\gG_h\colon h\in \bR^{3}\}$.
\begin{proposition}
For any $\kappa<1/2$, the triple $(\cA,\cT, \cG)$ where $\cA=\{\abs{\tau}\colon \tau\in T\}$ is a regularity structure.
\end{proposition}
\begin{proof}
To prove that $\cA$ is a discrete lower bounded  set, we show that for any  $\gb\in \bR $ the set $I:=\{\tau\in T\colon \vert \tau\vert \leq \gb\}$ is finite. For any $\tau\in I$ by means of the identity \eqref{decomposition_T} there exist two indices $m\in \bN $, $n\in \bN^2$ and $\gs\in \{\Xi, \cI_1(\Xi), \cI_1(\Xi)^2 \}$ such that $\tau= \gs \cI(\Xi)^m\X^n$. From $\vert \tau\vert \leq \gb$ we deduce
\begin{equation}\label{ineq_T}
n_1+ 2n_2 + (1/2 -\kappa)m\leq \gb - \abs{\gs}\,.
\end{equation}
Imposing $\kappa<1/2$, the left-hand side of the inequality \eqref{ineq_T} is bigger or equal than $0$ and the set $I$ is bounded. This finiteness result implies also the identity $\cT= \bigoplus_{\gga\in A}\cT_{\gga}$ where $\cT_{\gga}=\langle\tau\in T\colon \vert \tau\vert = \gga\rangle$. Moreover there is no need to specify a norm on $\cT_{\gga}$, since it is finite-dimensional. Finally the property \eqref{reg_tri} comes directly from Newton's binomial formula and the positive homogeneity of the symbol $\cI(\Xi)$. 
\end{proof}
\begin{remark}\label{positive_renom}
The Definition \ref{defn_reg_symb} is a simplification of the vector space introduced in \cite[Pag. 7]{hairer_string16} with fewer symbols. The triple $(\cA,\cT, \cG)$ is also intimately linked with $(\cA^{HP},\cT^{HP}, \cG^{HP})$, the regularity structure defined in \cite[Pag. 13-14]{MartinHairer2015}. More precisely we consider $U^{HP}$ the smallest set of symbols of $F$ such that $ \{\X^k\}_{k\in \bN^2}\subset U$ and  satisfying the properties
\[\tau\in U^{HP}\Rightarrow \cI(\tau)\,, \cI(\Xi\tau) \in U^{HP}\,;\quad \tau\,, \tau'\in U^{HP}\Rightarrow \tau\tau'\in U^{HP}\,.\]
Introducing the set $T_{\Xi}^{HP}=\{\Xi v\in F\colon v\in U^{HP}\} $ and $\cT^{HP}_{\Xi}$, $\cU^{HP} $ the corresponding free vector spaces defined on these sets, the space  $\cT^{HP}$ is defined by $\cT^{HP}= \cT^{HP}_{\Xi}\oplus \cU^{HP}$. Looking  at $ \cT \cap \cT^{HP}=\cV_{\Xi}\oplus\cU $, it is also possible to show that the action of the group $\cG^{HP}$ coincides with that of  $\cG$ on these subspaces and from the explicit definition of $\gG$ one has $\gG(\cV_{\Xi})\subset \cV_{\Xi} $ and $\gG(\cU)\subset \cU $ for any $\gG\in \cG$. Hence the subspaces $\cV_{\Xi}$ and $\cU$ are respectively a sector of regularity $-3/2-\kappa $ and a function-like sector of both $\cT^{HP}$ and $\cT$ (see  \cite[Def. 2.5]{Hairer2014}). Due to these identifications, we can transfer some results of \cite{MartinHairer2015} to our context.
\end{remark}
\begin{remark}\label{finite_dim}
As a matter of fact, we can restrict our considerations  once and for all to a subspace of $\cT$ generated by all symbols with homogeneity less than some parameter $\gz>0$. By convention we denote by $\abs{\cdot}_{\gb}$ the euclidean norm on $\cT_{\gb}$ (the euclidean norm is coherent with \cite{MartinHairer2015} but there is no ``canonical" choice because $\cT_{\gb}$ is finite-dimensional). For any $\gb\in A$, we will denote by $\cQ_{\gb}$ and $\cQ_{<\gb}$ the projection operator respectively on $\cT_{\gb}$ and $ \bigotimes_{\ga<\gb}\cT_{\ga}$.
\end{remark}
\begin{remark}\label{multiplicative_gamma}
Following the definition \eqref{explicit_gamma}, we can also easily prove that  $\gG_h \tau\tau'= \gG_h\tau\gG_h\tau'$ for every symbol $\tau,\tau'\in T$ such that also their product $\tau\tau'$ belongs to $T$ and $h\in \bR^3$. We remark that the explicit expression of $\gG_h$ in \eqref{explicit_gamma} can be easily rewritten as
\begin{equation}\label{delta_gamma}
\gG_h= ( id\otimes h' )\gD\,,
\end{equation}
where $h'\colon \cU\to \bR$ is the unique real character over $\cU$ such that
\[h'(X_1)=h_1\,, \quad h'(X_2)=h_2\,, \quad h'(\cI(\Xi))=h_3\,\]
and $\gD\colon \cT\to \cT\otimes\cU$ is the unique linear map such that
\begin{equation}\label{delta_gamma_defn}
\begin{gathered}
\gD X_i = X_i \otimes \1 + \1 \otimes X_i \;,\quad 
\gD \1  = \1 \otimes \1\;, \quad \gD \gs   = \gs\otimes\1\\
\gD \cI(\Xi) = \cI(\Xi) \otimes \1 + \1 \otimes \cI(\Xi)\;, \quad \gD\tau\tau'= \gD\tau\gD\tau'\,.
\end{gathered}
\end{equation}
for every $i=1,2$, $\gs\in \{\Xi, \cI_1(\Xi), \cI_1(\Xi)^2 \}$ and all $\tau,\tau'\in T$ such that $\tau\tau'\in T$. Comparing the relations \eqref{delta_gamma_defn} with the explicit definitions given in \cite[Sec. 8.1]{Hairer2014} and  \cite[Pag. 14]{MartinHairer2015}, the group $\cG$ can be obtained also with the general construction presented in \cite{Hairer2014}.
\end{remark}
To apply some general results obtained in \cite{Bruned2019} and \cite{chandra_analytic16}, we show how to express the regularity structure $\cT$ using the formalism of \emph{trees}. Let us recall some basic notations. We start by considering labelled, rooted trees $\tau$. That is $\tau$ is a combinatorial tree (a finite connected simple graph) with a non-empty set of nodes $N_{\tau}$ and a set of edges $E_{\tau}$ without cycles, where we fixed a specific node $\gr_{\tau}\in N_{\tau}$ called the root of $\tau$. The trees we consider are also \emph{labelled} i.e. there exists a finite set of labels $\cL$ and a function $\mathfrak{t}\colon E_{\tau}\to \cL$. These trees are the building blocks of a more general family of trees. We define a \emph{decorated tree} as a triple $\tau^{\f{n}}_{\f{e}}= (\tau,\f{n},\f{e})$, where $\tau $ is a LR rooted tree and $\f{n}\colon N_{\tau}\to\bN^2$, $\f{e}\colon E_{\tau}\to \bN^{2}$ are two fixed functions. The set of decorated tree is denoted by $\mathfrak{T}$.

Let us define two main operations on $\mathfrak{T}$. For any two elements two elements $\tau^{\f{n}}_{\f{e}}$, $\gs^{\f{n}'}_{\f{e}'}\in \mathfrak{T}$ we introduce the \emph{product tree} $\tau^{\f{n}}_{\f{e}}\gs^{\f{n}'}_{\f{e}'}$ by simple considering $\tau\gs$, the tree obtained by joining the roots of $\tau^{\f{n}}_{\f{e}}$ and $\gs^{\f{n}'}_{\f{e}'}$. Moreover we  impose $\mathfrak{n}(\gr_{\tau\gs})= \mathfrak{n}(\gr_{\tau})+\mathfrak{n'}(\gr_{\gs})$ and we keep $\mathfrak{e}$ unaltered. Then for any $m\in \bN^2$, $\mathfrak{l}\in \cL$ we define the \emph{grafting application} $\cE^{\mathfrak{l}}_m\colon \mathfrak{T}\to \mathfrak{T}$ as follows: for any $\gs^{\f{n}}_{\f{e}}\in \mathfrak{T}$, $\cE^{\mathfrak{l}}_m(\gs^{\f{n}}_{\f{e}})\in \mathfrak{T}$ is the tree with zero decoration on the root obtained by adding one more edge decorated by $(\mathfrak{l},m)$ to the root of $\gs$. The set $\mathfrak{T}$ can be constructed recursively starting from the root trees $\{\bullet_{k}\}_{k\in \bN^2}$ and applying iteratively the grafting operations and the multiplication. Similarly to what we did for the set of symbols $F$, for any fixed $\mathfrak{s}\colon \cL\,\sqcup\, \bN^2\to \bR$  we define  a \emph{homogeneity map} $\vert \cdot\vert_{\mathfrak{s}} \colon \mathfrak{T}\to \bR $ as follows
\begin{equation}
\abs{\tau^{\f{n}}_{\f{e}}}_{\mathfrak{s}}:= \sum_{e\in E_{\tau}}\mathfrak{s}(\mathfrak{t}(e))-\mathfrak{s}(\f{e}(e))  + \sum_{x\in N_{\tau}} \mathfrak{s}(\f{n}(x))\,,
\end{equation}
for any $\tau^{\f{n}}_{\f{e}}\in \mathfrak{T}$. The name homogeneity is used because the function $\vert \cdot\vert_{\mathfrak{s}} $ has the following properties 
\[\abs{\tau^{\f{n}}_{\f{e}}\gs^{\f{n}'}_{\f{e}'}}_{\mathfrak{s}}= \abs{\tau^{\f{n}}_{\f{e}}}_{\mathfrak{s}}+\abs{\gs^{\f{n}'}_{\f{e}'}}_{\mathfrak{s}}\,,\quad \abs{\cE^{\mathfrak{l}}_m(\gs^{\f{n}}_{\f{e}})}_{\mathfrak{s}}=\abs{\gs^{\f{n}}_{\f{e}}}_{\mathfrak{s}}+ \mathfrak{s}(\mathfrak{l})- \mathfrak{s}(m)\,, \] 
which are similar to the properties of the homogeneity on symbols $\vert \cdot\vert$ introduced above.

To describe the symbols defining $\cT$ using trees, we choose in this case a set of labels with three elements $\cL=\{\Xi, I, J\}$, associated to the symbol $\Xi$ and $\cI$. The presence of two different labels $I, J$ to denote $\cI$ is done to isolate all the trees we need for our calculations. Once we introduce the labels, the function $\mathfrak{s}$ is defined by
\begin{equation}\label{scaling_s}
\mathfrak{s}(k_1,k_2)=2k_1+ k_2\;,\quad \mathfrak{s}(\Xi)= - \frac{3}{2}- \kappa\;,\quad \mathfrak{s}(I)=\mathfrak{s}(J)=2\,,
\end{equation}
where $\kappa>0$ is a fixed parameter. This choice of $\mathfrak{s} $ is done to be coherent with the homogeneity  $\vert \cdot\vert$, as explained in the Remark \ref{iota_map}. We can easily draw a labelled decorated tree $\tau^{\f{n}}_{\f{e}}\in \mathfrak{T}$ by simply putting its root at the bottom and decorating the nodes and the edges with the non zero values of $\mathfrak{n} $, $\mathfrak{e}$ and the labels of $\cL$. For example, when we write the tree
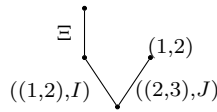
\[
\tikz[scale=3]{\node[label={[label distance=0.65cm]50:$\scriptstyle (1,2)$}]{} child{ node[label={[label distance=0.02cm]225:$\scriptstyle{((1,2), I)}$}]{} child{node[label={[label distance=1cm]-40:$\scriptstyle((2,3), J)$}]{} }}child{node[label={[label distance=1cm]168:$\scriptstyle \Xi$}]{}}}
\]

\vspace*{-0.4cm}

\noindent we suggest that $\mathfrak{n}$ is zero over three nodes and $\mathfrak{e}$ is zero on the edge labelled by $\Xi$.

To conclude the construction of a regularity structure, we need to choose a suitable subset of trees contained in $\mathfrak{T}$. This operation is formalised in the context of decorated trees by the notion of a ``rule" (see \cite[Def. 5.7]{Bruned2019}). This object takes in account the branching behaviour of the trees and consequently what type of edges are allowed next to every label. More precisely, denoting by $\mathfrak{E}$ the set of all finite multi-sets of $ \cL \times \bN^2$, a rule is a function $R\colon \cL \to P(\mathfrak{E})\setminus \{\emptyset\}$ where $P(\mathfrak{E})$ is the power set of $\mathfrak{E}$. Let us explain what rule we choose in this context. 
\begin{definition}\label{defn_rule}
Writing $\Xi$, $I$, $I_1$ and $[I]_k$ as a shorthand for $((0,0),\Xi) $, $( (0,0), I)$ ,  $((0,1),I)$ and the multi-set $(I,\dots, I)$ repeated $k$ times, we define
\begin{gather} 
R(\Xi)=\{()\}\,,\quad R(I)= \{()\,, \Xi\}\,, \nonumber\\
 R(J) = \lbrace ()\,,[I]_k, \;   ([I]_k,I_1), \;    ([I]_k,I_1,I_1), \;  ([I]_k,\Xi), \; k \in \bN \rbrace\,,
\end{gather}
where the brackets $ \{\}$ describe a subset of $\mathfrak{E}$ and the brackets $()$ describe a multi-set of $\cL \times \bN^2$ (the symbol $()$ denotes the empty multi-set).
\end{definition}
Once we established a rule, we can consider the set of all decorated trees which \textit{strongly conforms to the rule $R$} (see \cite[Def. 5.8]{Bruned2019}), denoted by $\mathfrak{T}(R)$, that is $\tau^{\f{n}}_{\f{e}}\in \mathfrak{T}(R)$ if the following properties are satisfied
\begin{itemize}
\item  Looking at the edges attached at the root $ \gr_{\tau}$, they can be expressed as $R(\mathfrak{l})$ for some $\mathfrak{l}\in \cL$; 
\item for any node $x\in N_{\tau}\setminus\{ \gr_{\tau}\}$, all the edges attached at $ x$ can be written as $R(\mathfrak{t}(e))$, where $e$ is the unique edge linking $x$ to its parent.
\end{itemize}
For example let us consider the two trees
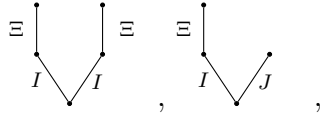
\[
\tikz[scale=3]{\node{} child{ node[label={[label distance=0.7cm]-20:$\scriptstyle I$}]{} child{node[label={[label distance=0.8cm]270:$\scriptstyle I$}]{} }}child{node[label={[label distance=1cm]168:$\scriptstyle \Xi$}]{} child{node[label={[label distance=0.25cm]-45:$\scriptstyle \Xi$}]{}}}}\;\;,\tikz[scale=3]{\node{} child{ node[label={[label distance=0.7cm]-20:$\scriptstyle J$}]{} child{node[label={[label distance=0.8cm]270:$\scriptstyle I$}]{} }}child{node[label={[label distance=1cm]168:$\scriptstyle \Xi$}]{}}}\quad \,,\]
where we used the shorthand notation $J= ((0,0), J)$. The tree on the left-hand side strongly conforms to the rule $R$, but the tree on the right one does not because the multi-set $\{I, J\}$ is not in the image of $R$. From the Definition \ref{defn_rule} it is straightforward to see that all possible decorations of the trees in $\mathfrak{T}(R) $ are of three types. We will abbreviate them with the shorthand notations\footnote{to improve the readability the same trees will be henceforth drawn on a smaller}
\[
\tikz[scale=3]{\node{} child{ node[label={[label distance=0.3cm]225:$\scriptstyle \Xi$}]{} }}=\tikz[scale=3]{\node{} child[noise]{ node{} }}\;\;,\qquad \tikz[scale=3]{\node{} child{ node[label={[label distance=0.3cm]225: $\scriptstyle I$}]{} }}= \tikz[scale=3]{\node{} child{ node{} }}\;\;,\qquad \tikz[scale=3]{\node{} child{ node[label={[label distance=0.3cm]225:$\scriptstyle I_1$}]{} }}=\tikz[scale=3]{\node{} child[derivative]{ node{} }}\,\,.
\]

Thanks to the definition of $\mathfrak{T}(R)$, we can extract a specific subset of trees from $\mathfrak{T}$. To conclude the existence of a regularity structure from $\mathfrak{T}(R)$, it is necessary to check  two last fundamental properties on $R$:
\begin{itemize}
\item $R$ is \emph{subcritical} (see \cite[Def. 5.14]{Bruned2019}), that is there  exists a function $\text{reg}\colon \cL\to \bR$ such that when we extend it to $\mathfrak{E}$ for any $(\mathfrak{l},k)\in \cL\times \bN^2 $ and $ M\in \mathfrak{E}$ 
\[
\text{reg}(l,k):= \text{reg}(l)-\mathfrak{s}(k)\,, \qquad
\text{reg}(M) := \sum_{(l,k) \in M} \text{reg}(l,k)\,,
\]
then we have for any $\mathfrak{l}\in \cL$
\[\text{reg}(l)< \mathfrak{s}(l)+ \inf_{M\in R(l)}\text{reg}(M)\,.\]
\item $R$ is \emph{normal} (see \cite[Def. 5.8]{Bruned2019}), that is $ R(\mathfrak{l})= \{()\}$ for every $ \mathfrak{l}\in \cL$ such that $ \mathfrak{s}(\mathfrak{l})<0$ and for any couple of  multi-sets $M, N\in \mathfrak{E}$ such that $N\in R(\mathfrak{l})$ for some $\mathfrak{l}\in \cL$ and $M\subset N$, then  $M \in R(\mathfrak{l}) $.
\end{itemize}
Both properties are relatively easy to check in this specific case. Indeed the rule $R$ is normal by construction and we can verify the subcriticality hypothesis using the function $\text{reg}\colon \cL \to \bR$ defined by
\[\text{reg}(\Xi)= -\frac{3}{2}-2\kappa\,,\quad \text{reg}(I)=\text{reg}(J)= \frac{1}{2}-3\kappa\,,\]
as long as $\kappa$ is sufficiently small. These two properties allow us to apply the results \cite[Prop. 5.21]{Bruned2019},  \cite[Prop. 5.39]{Bruned2019} and  the definition \cite[Def. 6.22]{Bruned2019} to prove the following result
\begin{proposition}\label{reg_struct_tree}
There exists a rule $R'$ such that $ R(\mathfrak{l})\subset R'(\mathfrak{l})$ for every $\mathfrak{l}\in \cL$ and a group $\cG'$ such that the triple $(\cA',\cT', \cG')$ is a regularity structure, where $\cT'$ is the free vector space generated from $ \mathfrak{T}(R') $ and $\cA=\{\abs{\tau}_{\mathfrak{s}}\colon \tau\in \mathfrak{T}(R')\}$. 
\end{proposition} 
Even if we could give an explicit description of the group $\cG'$ and $\cT'$ given in the Proposition \ref{reg_struct_tree}, for our purposes it is sufficient to establish a relation between the regularity structure $\cT$ and $\cT'$. From the explicit definition of $F$ and $\mathfrak{T}$, it is possible to define recursively an injective map $\iota\colon F \to \mathfrak{T}$ as follows:
\begin{itemize}
\item for any $ m\in \bN^2$ we set
\[ \iota(\Xi):= \tikz[scale=1.3]{\node{} child[noise]{node{}}}\;,\quad \quad \iota(\X^m):=\bullet_m \]
\item For any symbol $\gs$ such that $\iota(\gs)$ is defined, then $\iota(\cI_k(\gs)):= \cE^I_k(\gs)$ .
\item For any couple of symbols $\gs,\gs'$ such that $\iota(\gs)$ and $\iota(\gs')$ are well defined we set $\iota(\gs\gs')=\iota(\gs)\iota(\gs') $.
\end{itemize}
\noindent We present two examples of the action of $\iota$:
\[\iota(\Xi^2\mathcal{I}(\Xi))= \tikz[scale=1.3]{\node{} node{} child[noise]{node{} } child[noise]{node{} } child{node{} edge from parent[noise] child{node{}}}}\,,\quad \iota(\cI_1(\Xi)^2\cI(\Xi\X^{(3,4)}))=\tikz[scale=1.3]{\node{} child[derivative]{node[label={[label distance=0.6cm]0: \tiny (3,4)}]{} edge from parent[noise] child {node{}}} child[derivative]{node{} edge from parent[noise] child {node{}}} child{node{} child[noise]{node{}}} }\,.\]
Restricting the map $\iota$ on $T$ and extending it by linearity we have the following inclusion
\begin{proposition}\label{inclusion}
The regularity structure $(\cA, \cT, \cG)$ is contained in $(\cA',\cT', \cG')$ in the sense of the inclusion of regularity structure explained in \cite[Sec. 2.1]{Hairer2014}.
\end{proposition}
\begin{proof}
The theorem is a strict consequence of the choices done to define $\cT'$. Firstly by definition of $R$, every decorated tree $\iota (\tau)$ for some $\tau\in T$ strongly conforms to the rule $R$, therefore it will strongly conform to the rule $R'$. Moreover by construction of $\mathfrak{s}$ in \eqref{scaling_s} we have $\abs{\iota(\gs)}_{\mathfrak{s}}= \abs{\gs}$ for any $\gs\in F$. Thus $\cA\subset \cA'$. Finally, when we consider the groups $\cG$, $\cG'$, it has been showed in \cite[Equation (6.32)]{Bruned2019} that the group $\cG'$ acts on $\iota(\cT)$ in the same way as the operator $\gD$ explained in the Remark \ref{multiplicative_gamma}. Therefore we obtain the inclusion of the regularity structures.
\end{proof}
\begin{remark}\label{iota_map}
The function $\iota$ is an injective map function from $F$ to $\mathfrak{T}$ but there are many trees of $\mathfrak{T}$ which do not belong to $\iota(F)$. In particular, If a tree contains only the label $I$ and no $\Xi$, then it is not contained, because we identified all the symbols $\cI_k(\X^m)$ to zero. Moreover, none of the trees labelled with $J$ belong to $\iota(F)$. In what follows, we will identify both symbols and decorated trees, without writing explicitly the map $\iota$.
\end{remark}

\subsection{Models on a regularity structure}
The algebraic structure comes also with a \emph{model} associated to it. In order to recall this notion and to simplify the whole exposition, we fix a parameter $\gz\geq 2$ and with an abuse of notation we will identify all along the chapter $\cT$ (respectively its canonical basis $T$) with the finite-dimensional vector space $\cQ_{<\gz}\cT$ (resp. the finite set $\{\tau\in T\colon \abs{\tau}<\zeta\}$). The same applies also for the sets $V_{\Xi}\,,V_{\cI_1(\Xi)^2}\,, V_{\cI_1(\Xi)}$.
\begin{definition}\label{defn_model}
A model on $(\cA,\cT, \cG)$ consists of a pair $(\Pi,\gG)$ given by:
\begin{itemize}
\item A map $\gG\colon \bR^2\times \bR^2\to \cG$ such that $\gG_{zz}=id$ and $\gG_{zv}\gG_{vw}=\gG_{zw}$ for any $z,v,w\in \bR^2$.
\item A collection $\Pi=\{\Pi_z \}_{z\in \bR^2}$ of linear maps $\Pi_z\colon \cT \mapsto \cS'(\bR^2)$ such that $\Pi_z=\Pi_v\gG_{vz}$ for any $z,v\in \bR^2$.
\end{itemize}
Furthermore, for every compact set $\cK\subset \bR^{2}$, one has
\begin{equation}\label{model1}
\norm{\Pi}_{\cK}:= \sup \left\{ \frac{\abs{(\Pi_z \tau) (\gh_z^\lambda)}}{\lambda^{\abs{\tau}}}\colon z \in \cK\,,\lambda \in (0,1]\,,\tau\in T ,\gh \in \cB_2\right\}<\infty\,,
\end{equation}
\begin{equation}\label{model2}
\norm{\gG }_{\cK}:= \sup \left\{ \frac{\abs{\gG_{zw}(\tau)}_{\gb}}{\norm{z-w}^{\abs{\tau}-\gb}}\colon z\neq w \in \cK\,,\norm{z-w}\leq 1\,,\tau\in T, \gb <\abs{\tau} \right\}<\infty\,,
\end{equation}
where the set of test functions $\cB_2 $ was already introduced in the Section \ref{elements}.
\end{definition}
This notion plays a fundamental role in the whole theory, because it associates to any $\tau\in T$ an explicit distribution $\Pi_z\tau$ belonging in some way to $\cC^{\abs{\tau}}$. To compare two different models defined on the same structure, we endow $\cM$, the set of all models on $(\cA,\cT,\cG)$, with the topology associated to the corresponding system of semi-distances induced by the conditions \eqref{model1} and \eqref{model2}:
\begin{equation}\label{DistModel}
\norm{(\Pi,\Gamma),(\bar\Pi,\bar\Gamma)}_{\cM(\cK)}:=\norm{\Pi - \bar \Pi}_{\cK} + \norm{\Gamma - \bar \Gamma}_{\cK}\;.
\end{equation}
Since we want to study the processes on a finite time horizon, it is sufficient to verify the conditions \eqref{model1} \eqref{model2} on a fixed compact set $\cK$ containing $[0, T]\times [0,1]$ and we will avoid any reference to it in the notation. In this way $(\cM,\norm{\cdot}_{\cM})$ becomes a complete metric space ($\cM$ is not a Banach space because the sum of models is not necessarily a model!). In particular if a sequence $(\Pi^n,\gG^n)$ converge to $(\Pi, \gG)$, then $\Pi^n_z\tau$ converges to $\Pi_z\tau$ in the sense of tempered distributions for any $z$, $\tau$. To define correctly a model over a symbol of the form $\cI(\gs)$, we need a technical lemma related to a suitable decompositions of $G$, the heat kernel on $\bR$, interpreted as a function $G\colon \bR^2\setminus \{0\}\to \bR$.
\begin{lemma}[First decomposition]\label{KandR1}
\emph{(see \cite[Lemma 5.5]{Hairer2014})} There exists a couple of functions $ K\colon\bR^2\setminus \{0\}\to \bR$, $ R\colon\bR^2\to \bR$ such that $G(z)= K(z)+ R(z)$ in such a way that $R$ is $C^{\infty}(\bR^{2})$ and $K$ satisfies:
\begin{itemize}
\item $K$ is a smooth function on $\bR^2\setminus \{0\}$, supported on the set $\{(t,x)\in \bR^2\colon x^2 + |t| \leq 1\}$ and equal to $G$ on $\{(t,x)\in \bR_+\times \bR\colon x^2 + t <1/2, \, t>0\}$ .
\item $K(t,x) = 0$ for $t \leq 0, x\neq0$ and $K(t,-x) = K(t,x)$.
\item For every polynomial $Q \colon \bR^2 \mapsto \bR$ of parabolic degree less than $\zeta$, one has
\begin{equation}\label{pol_0}
\int_{\bR^2} K(t,x) Q(t,x)\,dx\,dt = 0\;.
\end{equation}
\end{itemize}
\end{lemma}

\begin{remark}\label{convolution_K_reg}
Thanks to these lemmas, it is possible to \emph{localise} on a compact support the regularising action of the heat kernel. Indeed it is also possible to show (see \cite[Lem 5.19]{Hairer2014}) that the map $v\to K*v$ sends continuously $\cC^{\ga}$ in $\cC^{\ga+2}$ for any non integer $\ga\in \bR$ and any distribution $v$ not necessarily compactly supported.
\end{remark}
In what follows, for any given realisation of $\widetilde{\xi_{\gep}}$, the periodic extension of $\xi_{\gep}$, we will provide the construction of $(\hPi^{\gep},\hgG^{\gep}) $ a sequence of models associated to it and  converging to a model $(\hPi,\hgG) $ related to $\xi$. As a further simplification, we parametrise all possible models $(\Pi,\gG)$ on $(\cA, \cT, \gG) $ with a couple $(\PI,f)$ where $\PI\colon \cT\to \cS'(\bR^2)$ and $f\colon \bR^2 \rightarrow  \bR^{3}$. Indeed it is straightforward to check that for any given couple $(\PI, f)$ the operators
\begin{equation}\label{PI_and_Pi}
\Pi_z := \PI \gG_{ f(z)}\,, \quad \gG_{zz'}:= \gG_{f(z')- f(z)}.
\end{equation}
satisfy trivially the algebraic relationships in the Definition \ref{defn_model}, because of the identity \eqref{algebraic_Gamma}. Since any realisation of $\xi_{\gep}$ is smooth, we firstly introduce a model upon any deterministic smooth function $\zeta\colon \bR^2\to \bR$. 
\begin{proposition}\label{canonical}
Let $\zeta\colon \bR^2\to \bR$ be a smooth periodic function and we suppose that the map $\PI$ satisfies the conditions
\begin{align}\label{canonical1}
\mathbf{\Pi} \1(z) = 1& \,,\quad \mathbf{\Pi} \X^k\tau (z)= z^k \mathbf{\Pi}\tau(z)\, ,\\ \label{canonical2}\PI\cI_{k}(\gs)(z)=& \partial^k(K*\PI(\gs))(z)\,,\quad \mathbf{\Pi} \Xi(z)= \zeta(z)\,, \\ \label{canonical3} &\mathbf{\Pi} \bar{\tau}\tau (z)= \mathbf{\Pi}\bar{\tau}(z) \mathbf{\Pi}\tau(z)\,;
\end{align}
defined for any $k\in\bN^d$, $\tau, \bar{\tau}\in T$ such that $\tau\X^k\in T $, $\cI_k(\tau)\in T$ and $\tau\bar{\tau}\in T$. Then, there exists a unique couple $(\PI, f)$ such that, using the identifications \eqref{PI_and_Pi}, the associated operators $(\Pi,\gG)$ is a model. We call it the \emph{canonical model} of $\zeta$.
\end{proposition}
\begin{proof}
The hypotheses on $\zeta$ and the conditions \eqref{canonical1} \eqref{canonical2} implies straightforwardly that $\PI\tau$ is a smooth function for any $\tau\in T$ which is not a product of symbols. Therefore, the point-wise product on the right-hand side of the equation  \eqref{canonical3} is well defined and by linearity, the operator $\PI$ exists and it is unique. To conclude the proof it is sufficient to choose $f$ such that $(\Pi,\gG)$ satisfy the right analytic properties. We use \eqref{algebraic_Gamma} to compute explicitly
\begin{equation}\label{explicit_Pi_z}
\begin{split}
&\Pi_z(\gs\cI(\Xi)^m\X^k)(\bar{z})=\PI\gG_{f(z)}(\gs\cI(\Xi)^m\X^k)\\& =\PI(\gs)(\bar{z}) (\bar{z}+ (f(z))_{1,2})^{k}[(K*\xi)(\bar{z})+ (f(z))_3]^m \,.
\end{split}
\end{equation}
for any $z,\, \bar{z}\in\bR^2$, $\gs\in\{\cI_1(\Xi),\cI_1(\Xi)^2, \Xi, \1\}$ and $k$, $m$ as before. Imposing the condition
\begin{equation}\label{character_f_z}
f(z)_i= - z_i\,,\; i=1,2\qquad f(z)_3= -(K* \PI \Xi )(z)\,.
\end{equation}
we obtain immediately the bound $\Pi_z\tau (\bar{z})\leq C \norm{\bar{z}-z}^{\abs{\tau}}$ for some constant $C>0$ depending on $\xi$ and uniformly on $\tau \in T$. Thus the condition \eqref{model1} is satisfied. On the other hand, in order to check the second property \eqref{model2}, we can easily verify it using  when $\tau\in\{\cI(\Xi),X_1,X_2\} $ and  Applying the multiplicative property of $\gG$ (see the Remark \ref{multiplicative_gamma}) we conclude.
\end{proof} 
\begin{remark}\label{operator L}
The existence of a canonical model is a general result already proved in \cite[Prop 8.27]{Hairer2014} but we repeat a simplified version of that proof to take in account the slightly different notation of this article. Looking at the definition of $f$ in \eqref{character_f_z}, we remark that $f$ depends on $\PI$ but to define it we do not need the multiplicative property of $\PI$ \eqref{canonical3}, nor the smoothness of $\xi$. Therefore for any map $\PI\colon \cT\to \cS'(\bR^2)$, the conditions \eqref{character_f_z} and  \eqref{PI_and_Pi} identify uniquely a couple $\cL(\PI):=(\Pi, \gG)$. $\cL(\PI)$ is not necessarily a model but if $\PI\Xi=\xi\in \cC^{-3/2-\kappa}$ for some $0<\kappa<1/2$ and $\PI$ satisfies the properties \eqref{canonical1}, \eqref{canonical2}, then the proof of the Proposition \ref{canonical} implies also that the operators $\gG_{zz'}$ given by $\cL(\PI)$ will always satisfy the property \eqref{model2}. The choice of a kernel $K$ satisfying \eqref{pol_0} is due in order to be compatible with the  assumption that the symbols $\cI_k(\X^m)$ are identified with $0$.
\end{remark}

\begin{remark}
If $\xi$ is also periodic in the space variable it is straightforward to prove that the canonical model $(\Pi,\gG)$ associated to $\xi$ satisfies also
\begin{equation}\label{periodic}
\Pi_{(t,x+m)}\tau(t',x'+m)=\Pi_{(t,x)}\tau (t',x')\,,\quad \gG_{(t,x+m)(t',x'+m)}\tau= \gG_{(t,x)(t',x')}\tau 
\end{equation}
for any couple of space-time points $z=(t,x)$, $z'=(t',x')$, $m\in \bZ$, $\tau\in T$. Thus the canonical model is also \emph{adapted} to the action of translation on $\bR$ (for this definition see \cite[Definition 3.33]{Hairer2014}). Roughly speaking this property allows to apply the notion of models also for distributions periodic in space.
\end{remark}
\begin{remark}\label{tilde_model}
Recalling the inclusion of the regularity structure $\cT$ in $\cT'$ as explained in the Proposition \ref{inclusion}, we can immediately extend the Definition \ref{defn_model} to define a model $(\Pi^{'},\gG^{'})$ over the regularity structure $(\cA',\cT',\cG')$. This extension will be useful to define the so-called BPHZ renormalisation and the BPHZ model as In particular for any smooth function $\zeta\colon \bR^2\to \bR$  we can define again a canonical model in this context starting from an explicit function $\PI'\colon \mathfrak{T}(R')\to\cC^{\infty}$. Using the grafting operation the application $\PI'$ is defined recursively for any $k,m\in \bN^2$, $\tau, \tau'\in \mathfrak{T}$ 
\[\PI'(\bullet_{k})(z):= z^k\,,\quad \PI'(\cE^{J}_m(\tau))(z)= \PI'(\cE^{I}_m(\tau))(z)=\partial_k^m(K* \PI'(\tau))(z)\,,\]
\[ \PI'(\cE^{\Xi}_{m}(\bullet_{k}))(z):= \partial^m(\zeta(z)z^k)\,,\quad  \PI'(\tau \tau')(z):= \PI'(\tau)(z)\PI'( \tau')(z)\,.\]
These conditions allow  to define $\PI'$ without knowing in detail $R'$ and the existence of a model is provided by \cite[Prop.  6.12]{Bruned2019}. By construction when we restrict $\PI'$ on $T$ we obtain the properties \eqref{canonical1} \eqref{canonical2} \eqref{canonical3}. The map $\PI'$ will be important when we want to define the renormalisation of a model (see Theorem  \ref{BPHZ_explic}).
\end{remark}
\subsection{The BPHZ renormalisation and the BPHZ model}
For any $\gep>0$ we denote by $\PI^{\gep}$ and $\cL(\PI^{\gep}):=(\Pi^{\gep}, \gG^{\gep})$ the canonical model obtained by applying the Proposition \ref{canonical} where $\zeta$ is a fixed a.s. realisation of $\widetilde{\xi_{\gep}}$. Since $\xi_{\gep}$ converges to $\xi$ a.s. in the sense of distributions, we would like to define a model by studying the convergence of the sequence $(\Pi^{\gep},\gG^{\gep})$ as $\gep\to 0$. Unfortunately, it is well known from \cite{MartinHairer2015} that the sequence $\PI^{\gep} (\cI(\Xi)\Xi)$ does not converge as a distribution, implying that $\cL(\PI^{\gep})$ does not converge. A natural way to get rid of this ill-posedness and to prove a general convergence result is the main content of  \cite{Bruned2019} and \cite{chandra_analytic16}. The main consequence of these general results will be the existence of an explicit sequence of applications $F_{\gep}\colon \cM\to \cM$ such that the sequence $F_{\gep}(\cL(\PI^{\gep})):= (\hPi^{\gep},\hgG^{\gep})$ converges in probability to some random model. The model $(\hPi^{\gep},\hgG^{\gep})$ and the limiting model are referred in the literature as the \emph{BPHZ renormalisation} and the \emph{BPHZ model}. 

In order to satisfy the bounds \eqref{model2} for $\hgG^{\gep}$ uniformly on $\gep>0$, it is reasonable to write $(\hPi^{\gep},\hgG^{\gep})$ as  $\cL(\hPI^{\gep})$, for some admissible map $\hPI^{\gep}\colon \cT\to \cS'(\bR^2)$ (see the Remark \ref{operator L}).  This property can be obtained by defining a sequence of linear maps $\{A_{\gep}\}_{\gep>0}\colon \cT\to \cT$ satisfying
\begin{equation}\label{admissible2}
\begin{split}
 &A_{\gep} \1 = \1 \,,\quad A_{\gep}\cI_k(\tau) =  \cI_k(A_{\gep}\tau) \,,\\ 
& A_{\gep} \X^k\tau = \X^k A_{\gep}\tau\,,\quad A_{\gep} \Xi= \Xi\,,
\end{split}
\end{equation}
for any $k\in\bN^d$ and $\tau\in T$ such that $\tau\X^k\in T $, $\cI_k(\tau)\in T$. Indeed combining the properties of \eqref{admissible2} with the explicit definition of $\PI^{\gep}$ in the proposition \ref{canonical}, the application $\PI^{\gep}A_{\gep}$ is automatically admissible (by analogy we call $\{A_{\gep}\}$ an \emph{admissible renormalisation scheme}) and we can define the couple $\cL(\PI^{\gep}A_{\gep})$. However the conditions \eqref{admissible2} are not sufficient to prove that $\cL(\PI^{\gep}A_{\gep})$ is again a model. As a matter of fact, writing the elements of $\cT$ as trees and embedding $\cT$ in $\cT'$ (see the Proposition \ref{inclusion}), the BPHZ renormalisation is obtained from an explicit admissible renormalisation scheme $\{\widetilde{M}_{\gep}\}_{\gep>0}\colon \cT\to \cT$ such that imposing  $ \hPI^{\gep}:=\Pi^{\gep} \widetilde{M}_{\gep}$ the couple $\cL(\hPI^{\gep})$ is again a model. By construction $\widetilde{M}_{\gep}$ is a linear map $\widetilde{M}_{\gep}\colon \cT'\to \cT'$ but an important consequence of the Theorem \ref{BPHZ_explic} will imply $\widetilde{M}_{\gep} (\cT)\subset \cT$. Let us recall briefly the definition of $\widetilde{M}_{\gep} $ in terms of decorated trees as explained in \cite[Sec. 6]{Bruned2019} and \cite[Sec.4]{Bruned2018} starting from $\cT'$ and its basis $ \mathfrak{T}(R')$ (see the Proposition \ref{reg_struct_tree}).

Denoting  by $\centerdot$ the combinatorial operation of the disjoint union of graphs and by  $\emptyset$ the empty graph, we consider $\hat{T}_-$, the set of all graphs $\gs$ such that 
\[\gs =\tau_1 \centerdot \cdots \centerdot \tau_n\]
for some $n\geq 1$ and $\{\tau_i\}_{i=1,\cdots, n}\in \mathfrak{T}(R')\cup \{\emptyset\}$. The elements of $\hat{T}_- $ are called \emph{forests} and we denote by $\hat{\cT}_-$ the free vector space generated over $\hat{T}_- $. $(\hat{\cT}_-, \centerdot, \emptyset)$ is trivially a commutative algebra with unity. For any decorated tree $\tau^{\f{n}}_{\f{e}} $ we say that a forest $\gga\in \hat{T}_- $ is a \emph{subforest} of $\tau^{\f{n}}_{\f{e}} $ ($\gga\subset \tau^{\f{n}}_{\f{e}}$) if $\gga$ is an arbitrary subgraph of $\tau^{\f{n}}_{\f{e}} $ with  no isolated vertices. For instance let us consider
\[ \gga_1=\tikz[scale=1.3]{\node{}  child[noise]{node{} } child{node{} edge from parent[noise] child{node{}}}}\,,\qquad \gga_2=\tikz[scale=1.3]{\node{} child{ node{} child[noise]{node{} }}}\quad\tikz[scale=1.3]{\node{}}\,,\qquad \gga_3=\tikz[scale=1.3]{\node{} child[noise]{ node{} }}\quad\tikz[scale=1.3]{\node{} child{ node{} child[noise]{node{} }}}\;.\]
In this case $\gga_2$ and $\gga_3$ are both elements of $\hat{T}_- $  but only $\gga_3\subset \gga_1$ because $\gga_2$ has an isolated vertex. The empty forest $\emptyset$ is always a subforest. A decorated tree $ \tau^{\f{n}}_{\f{e}}$ and a subforest $\gga=\gs_1 \centerdot \cdots \centerdot \gs_n$ such that $\gga\subset \tau^{\f{n}}_{\f{e}}$ are used to define the contraction tree $\cK_{\gga} \tau^{\f{n}}_{\f{e}}=(\cK_{\gga}\tau, \cK_{\gga}\mathfrak{n},  \cK_{\gga}\mathfrak{e})$, where
\begin{itemize}
\item $\cK_{\gga}\tau$ is the tree obtained from $\tau$ replacing each $\gs_i$ with a node.
\item Denoting by $\bullet_1$,$\cdots$, $\bullet_n$ each node associated to the contraction of the tree $\gs_i$, the function $\cK_{\gga}\mathfrak{n}$ is equal to $\f{n}$ on every non contracted node of $\cK_{\gga}\tau$ and for every $i$, $\f{n}(\bullet_i)= \sum_{y\in N_{\gs_i}}\f{n}(y)$.
\item $\cK_{\gga}\mathfrak{e}\colon E_{\cK_{\gga}}\to \bN^2 $ is equal to $\f{e}$ on every non contracted edge of $\cK_{\gga}\tau$.
\end{itemize}
In the previous example we have $\cK_{\gga_{3}}\gga_1= \bullet$.

Once we give $\hat{\cT}_-$, we define $\cT_-:= \hat{\cT}_-/\cJ$ as the quotient algebra of $\hat{\cT}_-$ with respect to $\cJ$, the ideal of $\hat{\cT}_-$ generated by the set 
\[
J:=\{ \tau^{\mathfrak{n}}_{\mathfrak{e}}\in \mathfrak{T}(R')\colon \,\abs{\tau^{\mathfrak{n}}_{\mathfrak{e}}}_{\mathfrak{s}}> 0\;\}\subset \hat{T}_-\;.
\]
The map $\widetilde{M}_{\gep}$ is then defined  for any $\tau^{\f{n}}_{\f{e}}\in \mathfrak{T}(R')$ as
\begin{equation}\label{BPHZ_renorm}
 \widetilde{M}_{\gep}\tau^{\f{n}}_{\f{e}}:= (h_{\gep}\otimes id )\gD_-\tau^{\f{n}}_{\f{e}}\,.
\end{equation}
We will describe the objects $\gD_-$ and $h_{\gep}$ separately. First $\gD_-\colon \cT\to \cT_-\otimes \cT$ is a linear map which is explicitly given for any $\tau^{\f{n}}_{\f{e}}\in \mathfrak{T}(R')$ by the formula
\begin{equation}\label{defn_delta}
\gD_-\tau^{\f{n}}_{\f{e}} : =\sum_{\gga\subset \tau} \sum_{\mathfrak{e}_{\gga}, \mathfrak{n}_{\gga}\leq \f{n} }\frac{1}{\mathfrak{e}_{\gga}!} \binom{\f{n}}{\mathfrak{n}_{\gga}} p(\gga, \mathfrak{n}_{\gga} + \pi\mathfrak{e}_{\gga},\mathfrak{e}\vert_{\gga}) \otimes  (  \cK_{\gga}\tau, \cK_{\gga}(\mathfrak{n}- \mathfrak{n}_{\gga}),  \cK_{\gga}\mathfrak{e}+ \mathfrak{e}_{\gga})\,.
\end{equation}
Let us explain the meaning of the formula \eqref{defn_delta}. The first sum outside is done over all subforests $\gga\subset\tau$ and for any subforest $\gga$, denoting by $N_{\gga}$ and $\partial(\gga, \tau)$ respectively the set of the nodes of $\gga$ and the edges in $E_{\tau}$ that are adjacent to $N_{\gga}$, the second sum is done over all functions $\mathfrak{n}_{\gga}\colon N_{\gga}\to \bN^2 $ and $\mathfrak{e}_{\gga }\colon\partial(\gga, \tau)\to \bN^2$ such that for any $x\in N_{\gga}$ $\mathfrak{n}_{\gga}(x)\leq \mathfrak{n}(x)$, where with an abuse of notation we denote by $\leq $ the lexicographical order between vectors of $\bN^2$. 

A generic subforest $\gga$ is an element of $\hat{T}_-$ so we compose it with the canonical projection homomorphism $p\colon \hat{\cT}_-\to\cT_-$, where with an abuse of notation we identify all forests generating $\cJ$ to zero. Moreover, for any $\mathfrak{e}_{\gga }\colon\partial(\gga, \tau)\to \bN^2 $ the function $\pi \mathfrak{e}_{\gga }\colon N_{\gga}\to \bN^2$ is given by
\[\pi \mathfrak{e}_{\gga }(x):= \sum_{e\in \partial(\gga, \tau)\colon  x\in e}\mathfrak{e}_{\gga }(e)\,.\]
The remaining combinatorial coefficients are finally interpreted in a multinomial sense, that is for any function $l\colon S\to \bN^2$ where $S$ is a finite set we have 
\[l!:= \prod_{y\in S}(l(y))_1!(l(y))_2!\]
and similarly for the binomial coefficients. In principle, the summations over $\mathfrak{n}_{\gga} $ and $\mathfrak{e}_{\gga }$ are done over an infinite set of values but the projection $p$ and the constraints $\f{n}_{\gga}\leq \f{n}$, together with the subcritical hypothesis of the rule $R'$, make the sum finite. 

On the other hand, the map $h_{\gep}$ has the explicit form
\begin{equation}\label{character_BPHZ}
h_{\gep}:= g_{\gep}(\PI)\widehat{A}_-\,.
\end{equation} 
The first object in \eqref{character_BPHZ} is given by a linear map $\widehat{A}_-\colon \cT_-\to \hat{\cT}_{-}$. Its name is \emph{twisted-antipode} and it is characterised as the only homomorphism (so then $\widehat{A}_-(\emptyset)= \emptyset$) such that for any tree $\tau^{\f{n}}_{\f{e}}\neq \emptyset$, denoting by $M^{\centerdot}$ the forest product, one has the identity
\begin{equation}\label{pseudo_antipode}
\widehat{A}_-\tau^{\f{n}}_{\f{e}}= - M^{\centerdot}(\widehat{A}_-\otimes id) (\gD_-\tau^{\f{n}}_{\f{e}} - \tau^{\f{n}}_{\f{e}} \otimes \1 )\,.
\end{equation}
Finally the last object $g_{\gep}(\PI)\colon\hat{\cT}_{-}\to \bR$ is the only real character on the algebra $ \hat{\cT}_{-}$ such that for any tree $\tau^{\f{n}}_{\f{e}} \in \mathfrak{T}(R')$
\begin{equation}\label{stochastic_object}
g_{\gep}(\PI) (\tau^{\f{n}}_{\f{e}}):=  \bE ({\PI}^{\gep})' (\tau^{\f{n}}_{\f{e}})(0)\,,
\end{equation}
where  $(\PI^{\gep})'$ is the extension of $\PI^{\gep}$ over all the decorated trees as explained in the Remark \ref{tilde_model}. We combine all these definitions to obtain the explicit form of the application $\widetilde{M}_{\gep}$.
\begin{theorem}\label{BPHZ_explic}
By fixing $\kappa>0$ sufficiently small and restricting the map $\widetilde{M}_{\gep}$ defined in \eqref{BPHZ_renorm} on $\cT$, we have $\widetilde{M}_{\gep}=M_{\gep}$, where $M_{\gep}\colon \cT\to \cT$ is the unique linear map satisfying $M_{\gep}=\text{id}$ on $\cV_{\cI_1(\Xi)}\oplus\cU$ and for any couple of indexes  $m\geq 0$, $k\in \bN^2$
\begin{equation}\label{defn_M}
\begin{split}
M_{\gep}(\Xi \cI(\Xi)^m\X^k)&=\Xi \cI(\Xi)^m\X^k-(m C_{\gep}^{1}\cI(\Xi)^{m-1}\X^k)\1_{m\geq 1}\;,\\
M_{\gep}(\cI_1(\Xi)^2\cI(\Xi)^m\X^k)&=\cI_1(\Xi)^2\cI(\Xi)^m\X^k- C_{\gep}^{2}\cI(\Xi)^m\X^k\;,
\end{split}
\end{equation}
where the constants $C_{\gep}^{1}$ and $C^{2}_{\gep}$ are given by
\begin{align}\label{C_1}
C_{\gep}^{1}&:=\bE[\PI^{\gep}(\Xi \cI(\Xi))(0)]=  \int_{\bR^2} \gr_{\gep}(z)(K*\gr_{\gep})(z)dz \,,\\ 
\label{C_2}
C^{2}_{\gep}&:=\bE[\PI^{\gep}(\cI_1(\Xi)^2)(0)]= \int_{\bR^2} (K_x*\gr_{\gep})^2(z)dz \,.
\end{align}
\end{theorem}
\begin{proof}
Thanks to the result \cite[Theorem 6.17]{Bruned2019}, for any $k\in\bN^d$, $\tau\in T$ such that $\tau\X^k\in T $, $\cI_k(\tau)\in T$ the map $\widetilde{M}_{\gep}$ always satisfies 
\[
\begin{split}
 &\widetilde{M}_{\gep} \1 = \1 \,,\quad \widetilde{M}_{\gep}\cI_k(\tau) =  \cI_k(\widetilde{M}_{\gep}\tau) \,,\quad\widetilde{M}_{\gep} \X^k\tau = \X^k\widetilde{M}_{\gep}\tau\,\,.
\end{split}
\]
Therefore to prove the theorem it is sufficient to show  for any $m$ the identities 
\begin{equation}\label{renom_identities}
\begin{split}
\widetilde{M}_{\gep}(\cI_1(\Xi) \cI(\Xi)^m)&=\cI_1(\Xi)\cI(\Xi)^m \,,\quad \widetilde{M}_{\gep}(\cI(\Xi)^m)=\cI(\Xi)^m \\\widetilde{M}_{\gep}(\Xi \cI(\Xi)^m)&=\Xi \cI(\Xi)^m-(m C_{\gep}^{1}\cI(\Xi)^{m-1})\1_{m\geq 1}\;,\\
\widetilde{M}_{\gep}(\cI_1(\Xi)^2\cI(\Xi)^m)&=\cI_1(\Xi)^2\cI(\Xi)^m- C_{\gep}^{2}\cI(\Xi)^m\,.
\end{split}
 \end{equation}
Denoting by $W$ the set of symbols
\[W:= \{\cI_1(\Xi)^2\cI(\Xi)^m\,,\;\cI_1(\Xi) \cI(\Xi)^m\,,\; \Xi \cI(\Xi)^m\,,\; \cI(\Xi)^m \colon m\in \bN\}\,.\]
we have to calculate the operator $\gD_-$ and $h_{\gep}$ over the elements of $W$. To do that we need to know for any $w\in W$ what are the subforests $\gga\subset w$ and in principle, we should know explicitly the rule $R'$ and the forests which define $\hat{\cT}_-$. However, we remark that every subgraph $\gga$ included in $w$ with no isolated vertices can be expressed as a disjoint union of trees belonging only to $\mathfrak{T}(R)$. Thus the knowledge of $R'$ is unnecessary to calculate  $\gD_-$. Secondly we fix $\kappa>0$ sufficiently small such that the only trees of $\mathfrak{T}(R)$ with strictly negative homogeneity that are included in $W$ are the following  
\[W_-:=\left\{\tikz[scale=1.5]{\node{} child[noise]{node{}}}\,, \tikz[scale=1.2]{\node{} child{node{} edge from parent[noise] child {node{}}} child[noise]{node{} }}\, ,\,\tikz[scale=1.2]{\node{} child{node{} edge from parent[noise] child {node{}}}child{node{} edge from parent[noise] child {node{}}}child[noise]{node{} }} \, ,\, \tikz[scale=1.2]{\node{} child{node{} edge from parent[noise] child {node{}}}child{node{} edge from parent[noise] child {node{}}}child{node{} edge from parent[noise] child {node{}}}child[noise]{node{} }}\,, \, \tikz[scale=1.2]{\node{}  child[derivative]{node{} edge from parent[noise] child {node{}}}}\,,\, \tikz[scale=1.2]{\node{}  child[derivative]{node{} edge from parent[noise] child {node{}}}child[derivative]{node{} edge from parent[noise] child {node{}}}} \,,\tikz[scale=1.2]{\node{}  child[derivative]{node{} edge from parent[noise] child {node{}}}child{node{} edge from parent[noise] child {node{}}}}\,, \tikz[scale=1.2]{\node{}  child[derivative]{node{} edge from parent[noise] child {node{}}} child[derivative]{node{} edge from parent[noise] child {node{}}} child{node{} edge from parent[noise] child {node{}}}}\,,\tikz[scale=1.2]{\node{}  child[derivative]{node{} edge from parent[noise] child {node{}}} child[derivative]{node{} edge from parent[noise] child {node{}}} child{node{} edge from parent[noise] child {node{}}}child{node{} edge from parent[noise] child {node{}}}}\,\,\right\}.\]
Denoting by $\tau_m=\Xi\cI(\Xi)^m $, we calculate the quantity $\widetilde{M}_{\gep}(\tau_m)$ in case $m=0,1$ explaining all the passages. Firstly, we apply $\gD_-$ in  \eqref{defn_delta} and the recursive definition of $\widehat{A}_-$ in \eqref{pseudo_antipode} to obtain immediately
\[\gD_{-} \;\tikz[scale=1.2]{\node{} child[noise]{node{}}}= \emptyset \otimes \tikz[scale=1.2]{\node{} child[noise]{node{}}} + \tikz[scale=1.2]{\node{} child[noise]{node{}}}\otimes \1\,,\quad \widehat{A}_-\;\tikz[scale=1.2]{\node{} child[noise]{node{}}}= - \tikz[scale=1.2]{\node{} child[noise]{node{}}}\;,\]
\[ \gD_{-} \tikz[scale=1.2]{\node{} child{node{} edge from parent[noise] child {node{}}} child[noise]{node{} }}= \emptyset\otimes \tikz[scale=1.2]{\node{} child{node{} edge from parent[noise] child {node{}}} child[noise]{node{} }}+\;\tikz[scale=1.2]{\node{} child[noise]{ node{}} } \otimes \tikz[scale=1.2]{\node{} child{node{} edge from parent[noise] child {node{}}} }+ \tikz[scale=1.2]{\node{} child[noise]{ node{}} }_{\scriptstyle (0,1)} \otimes \,\tikz[scale=1.2]{\node{} child[derivative]{node{} edge from parent[noise] child {node{}}}} + \tikz[scale=1.2]{\node{} child[noise]{ node{}} }\otimes \tikz[scale=1.2]{\node{} child[noise]{ node{}} child{node{}} }+  \tikz[scale=1.2]{\node{} child[noise]{ node{}} }_{\scriptstyle (0,1)}\otimes \tikz[scale=1.2]{\node{} child[noise]{ node{}} child[derivative]{node{}} }+  \tikz[scale=1.2]{\node{} child[noise]{ node{}} }\;\tikz[scale=1.2]{\node{} child[noise]{ node{}} }\otimes \tikz[scale=1.2]{\node{} child{node{}} }+\, \tikz[scale=1.2]{\node{} child[noise]{ node{}} }_{\scriptstyle (0,1)}\,\tikz[scale=1.2]{\node{} child[noise]{ node{}} }_{\scriptstyle (0,1)}\otimes \tikz[scale=1.2]{\node{} child[derivative]{node{}} }+\tikz[scale=1.2]{\node{} child{node{} edge from parent[noise] child {node{}}} child[noise]{node{} }}\otimes \1\; ,\]
\[ \widehat{A}_{-}\tikz[scale=1.2]{\node{} child{node{} edge from parent[noise] child {node{}}} child[noise]{node{} }}=-\tikz[scale=1.2]{\node{} child{node{} edge from parent[noise] child {node{}}} child[noise]{node{} }}+\;\tikz[scale=1.2]{\node{} child[noise]{ node{}} } \;\tikz[scale=1.2]{\node{} child{node{} edge from parent[noise] child {node{}}} }- \widehat{A}_{-}\left(\,\tikz[scale=1.2]{\node{} child[noise]{ node{}} }_{\scriptstyle (0,1)} \right)\;\tikz[scale=1.2]{\node{} child[derivative]{node{} edge from parent[noise] child {node{}}}}+ \tikz[scale=1.2]{\node{} child[noise]{ node{}} }\;\;\tikz[scale=1.2]{\node{} child[noise]{ node{}} child{node{}} }- \widehat{A}_{-}\left(\,\tikz[scale=1.2]{\node{} child[noise]{ node{}} }_{\scriptstyle (0,1)}\right) \tikz[scale=1.2]{\node{} child[noise]{ node{}} child[derivative]{node{}} }- \tikz[scale=1.2]{\node{} child[noise]{ node{}} }\;\;\tikz[scale=1.2]{\node{} child[noise]{ node{}} }\;\; \tikz[scale=1.2]{\node{} child{node{}} }- \widehat{A}_{-}\left(\,\tikz[scale=1.2]{\node{} child[noise]{ node{}} }_{\scriptstyle (0,1)} \right)\widehat{A}_{-}\left(\,\tikz[scale=1.2]{\node{} child[noise]{ node{}} }_{\scriptstyle (0,1)}\right)\; \tikz[scale=1.2]{\node{} child[derivative]{node{}} }\,.\]
The terms with the extra decoration $(0,1)$ come from the sum over $\mathfrak{e}_{\gga}$ in the definition of $\gD_-$, combined with the projection operator $p$ and the map $\pi$ (we recall that all forests generating $\mathcal{J}$ are identified to zero). Using again the general definition of $\gD_-$ in \eqref{defn_delta} and the recursive identity \eqref{pseudo_antipode}, the calculations for the symbol $ \Xi \X^{(0,1)}$ are given by 
\[\gD_{-}\,\tikz[scale=1.2]{\node{} child[noise]{ node{}} }_{\scriptstyle (0,1)}=\emptyset\otimes  \,\tikz[scale=1.2]{\node{} child[noise]{ node{}} }_{\scriptstyle (0,1)}+ \tikz[scale=1.2]{\node{} child[noise]{node{}}}\otimes \bullet_{(0,1)}+ \,\tikz[scale=1.2]{\node{} child[noise]{ node{}} }_{\scriptstyle (0,1)}\otimes \1\;,\quad \widehat{A}_-\left(\,\tikz[scale=1.2]{\node{} child[noise]{ node{}} }_{\scriptstyle (0,1)}\right)= - \, \tikz[scale=1.2]{\node{} child[noise]{ node{}} }_{\scriptstyle (0,1)}+ \tikz[scale=1.2]{\node{} child[noise]{node{}}}\;\;\bullet_{(0,1)}\,,\]
To complete the calculation of $\widetilde{M}_{\gep}$, we need to apply $g_{\gep}(\PI)$ on the images of the pseudo antipode. By definition of $(\PI^{\gep})'$ one has
\begin{equation}\label{stochastic_semplification}
\begin{gathered}
g_{\gep}(\PI) (\,\tikz[scale=1.2]{\node{} child[noise]{node{} }}\,)= \bE\left[\int_{\bR^2}  \gr_{\gep}(-z_1)d\widetilde{W}_{z_1}\right]=0\,,\quad g_{\gep}(\PI) \left(\,\tikz[scale=1.2]{\node{} child[noise]{ node{}} }_{\scriptstyle (0,1)}\;\right)= 0\,,\\ g_{\gep}(\PI)\left(\tikz[scale=1.2]{\node{} child{node{} edge from parent[noise] child {node{}}} child[noise]{node{} }}\right)=\bE\left[\int_{\bR^2}  \gr_{\gep}(-z_1)d\widetilde{W}_{z_1}\int_{\bR^2}  K*\gr_{\gep}(-z_2)d\widetilde{W}_{z_2}\right]= C^1_{\gep}\,.
\end{gathered}
\end{equation}
Hence we conclude firstly
\begin{equation}\label{h_eps}
h_{\gep}(\;\tikz[scale=1.2]{\node{} child[noise]{node{} }}\;)=h_{\gep}\left(\,\tikz[scale=1.2]{\node{} child[noise]{ node{}} }_{\scriptstyle (0,1)}\,\right)=0 \,,\quad h_{\gep}\left(\tikz[scale=1.2]{\node{} child{node{} edge from parent[noise] child {node{}}} child[noise]{node{} }}\right)= - C^1_{\gep}\,.
\end{equation}
Plugging the formulae \eqref{h_eps} in the sums of $\gD_-$ we obtain the right identities of \eqref{renom_identities} for $\widetilde{M}_{\gep}\tau_m$ $m=0,1$. Let us pass to the calculation of $\widetilde{M}_{\gep}\tau_m$ $m=2,3$. Writing  $\gD_- \tau_m$ and $\widehat{A}_-\tau_m $, a deep consequence of \eqref{h_eps} and \eqref{stochastic_semplification} is then all the subforests containing the trees $\Xi$ or  $\Xi\X^{(0,1)}$ between the connected components will become zero after applying $h_{\gep}$ or $g_{\gep}(\PI)$, thereby not giving any contribution for $\widetilde{M}_{\gep}$. Denoting by $(\cdots)$ all these terms we have
\vspace{-0.2cm}

\[\gD_{-}\tikz[scale=1.2]{\node{} child{node{} edge from parent[noise] child {node{}}}child{node{} edge from parent[noise] child {node{}}}child[noise]{node{} }}=\emptyset\otimes\tikz[scale=1.2]{\node{} child{node{} edge from parent[noise] child {node{}}}child{node{} edge from parent[noise] child {node{}}}child[noise]{node{} }} +2\,\tikz[scale=1.2]{\node{} child{node{} edge from parent[noise] child {node{}}} child[noise]{node{} }}\otimes \tikz[scale=1.2]{\node{} child{node{} edge from parent[noise] child {node{}}}}+2\,\tikz[scale=1.2]{\node{} child{node{} edge from parent[noise] child {node{}}} child[noise]{node{} }}_{\!\!\!\scriptstyle (0,1)}\otimes \tikz[scale=1.2]{\node{} child[derivative]{node{} edge from parent[noise] child {node{}}}}+ \tikz[scale=1.2]{\node{} child{node{} edge from parent[noise] child {node{}}}child{node{} edge from parent[noise] child {node{}}}child[noise]{node{} }}\otimes\1 + (\cdots)\,,\]
\[\begin{split}
\gD_{-}\tikz[scale=1.2]{\node{} child{node{} edge from parent[noise] child {node{}}}child{node{} edge from parent[noise] child {node{}}}child{node{} edge from parent[noise] child {node{}}}child[noise]{node{} }}&= \emptyset\otimes \tikz[scale=1.2]{\node{} child{node{} edge from parent[noise] child {node{}}}child{node{} edge from parent[noise] child {node{}}}child{node{} edge from parent[noise] child {node{}}}child[noise]{node{} }}+3\,\tikz[scale=1.2]{\node{} child{node{} edge from parent[noise] child {node{}}} child[noise]{node{} }}\otimes \tikz[scale=1.2]{\node{} child{node{} edge from parent[noise] child {node{}}}child{node{} edge from parent[noise] child {node{}}}}+6\,\tikz[scale=1.2]{\node{} child{node{} edge from parent[noise] child {node{}}} child[noise]{node{} }}_{\!\!\!\scriptstyle (0,1)}\otimes \tikz[scale=1.2]{\node{}child{node{} edge from parent[noise] child {node{}}} child[derivative]{node{} edge from parent[noise] child {node{}}}}+  3\,\tikz[scale=1.2]{\node{} child{node{} edge from parent[noise] child {node{}}}child{node{} edge from parent[noise] child {node{}}}child[noise]{node{} }}\otimes \tikz[scale=1.2]{\node{} child{node{} edge from parent[noise] child {node{}}}}+ \tikz[scale=1.2]{\node{} child{node{} edge from parent[noise] child {node{}}}child{node{} edge from parent[noise] child {node{}}}child{node{} edge from parent[noise] child {node{}}}child[noise]{node{} }} \otimes \1  +(\cdots)\,,
\end{split}\]
\[\widehat{A}_{-}\tikz[scale=1.2]{\node{} child{node{} edge from parent[noise] child {node{}}}child{node{} edge from parent[noise] child {node{}}}child[noise]{node{} }}=-\tikz[scale=1.2]{\node{} child{node{} edge from parent[noise] child {node{}}}child{node{} edge from parent[noise] child {node{}}}child[noise]{node{} }} -2\widehat{A}_{-}\left(\,\tikz[scale=1.2]{\node{} child{node{} edge from parent[noise] child {node{}}} child[noise]{node{} }}\right)\;\,\tikz[scale=1.2]{\node{} child{node{} edge from parent[noise] child {node{}}}}-2\widehat{A}_{-}\left(\tikz[scale=1.2]{\node{} child{node{} edge from parent[noise] child {node{}}} child[noise]{node{} }}_{\!\!\!\scriptstyle (0,1)}\right)\;\,\tikz[scale=1.2]{\node{} child[derivative]{node{} edge from parent[noise] child {node{}}}}\;+ (\cdots)\,,\]
\[\begin{split}
\widehat{A}_{-}\tikz[scale=1.2]{\node{} child{node{} edge from parent[noise] child {node{}}}child{node{} edge from parent[noise] child {node{}}}child{node{} edge from parent[noise] child {node{}}}child[noise]{node{} }}&= -\tikz[scale=1.2]{\node{} child{node{} edge from parent[noise] child {node{}}}child{node{} edge from parent[noise] child {node{}}}child{node{} edge from parent[noise] child {node{}}}child[noise]{node{} }} -3\widehat{A}_{-}\left(\tikz[scale=1.2]{\node{} child{node{} edge from parent[noise] child {node{}}} child[noise]{node{} }}\right)\; \tikz[scale=1.2]{\node{} child{node{} edge from parent[noise] child {node{}}}child{node{} edge from parent[noise] child {node{}}}}+6\widehat{A}_{-}\left(\,\tikz[scale=1.2]{\node{} child{node{} edge from parent[noise] child {node{}}} child[noise]{node{} }}_{\!\!\!\scriptstyle (0,1)}\right)\; \tikz[scale=1.2]{\node{}child{node{} edge from parent[noise] child {node{}}} child[derivative]{node{} edge from parent[noise] child {node{}}}} -2\widehat{A}_{-}\left(\tikz[scale=1.2]{\node{} child{node{} edge from parent[noise] child {node{}}}child{node{} edge from parent[noise] child {node{}}}child[noise]{node{} }}\right)\; \tikz[scale=1.2]{\node{} child{node{} edge from parent[noise] child {node{}}}}+ (\cdots)\;. 
\end{split}\]
Similarly we also have
\[\gD_-\tikz[scale=1.2]{\node{} child{node{} edge from parent[noise] child {node{}}} child[noise]{node{} }}_{\!\!\!\scriptstyle (0,1)}= \emptyset \otimes \tikz[scale=1.2]{\node{} child{node{} edge from parent[noise] child {node{}}} child[noise]{node{} }}_{\!\!\!\scriptstyle (0,1)} + \tikz[scale=1.2]{\node{} child{node{} edge from parent[noise] child {node{}}} child[noise]{node{} }}_{\!\!\!\scriptstyle (0,1)}\otimes \1 + (\cdots)\,,\quad \widehat{A}_{-}\left(\,\tikz[scale=1.2]{\node{} child{node{} edge from parent[noise] child {node{}}} child[noise]{node{} }}_{\!\!\!\scriptstyle (0,1)}\right)= -\tikz[scale=1.2]{\node{} child{node{} edge from parent[noise] child {node{}}} child[noise]{node{} }}_{\!\!\!\scriptstyle (0,1)}+ (\cdots)\,.\]
\vspace{-0.4cm}

\noindent Therefore the calculation of $\widetilde{M}_{\gep}\tau_m$ is obtained once we know the constants
\[g_{\gep}(\PI)\left(\tikz[scale=1.2]{\node{} child{node{} edge from parent[noise] child {node{}}} child[noise]{node{} }}_{\!\!\!\scriptstyle (0,1)}\right),\;g_{\gep}(\PI)\left(\tikz[scale=1.2]{\node{} child{node{} edge from parent[noise] child {node{}}}child{node{} edge from parent[noise] child {node{}}}child[noise]{node{} }}\right),\;    g_{\gep}(\PI)\left(\tikz[scale=1.2]{\node{} child{node{} edge from parent[noise] child {node{}}}child{node{} edge from parent[noise] child {node{}}}}\right),\;g_{\gep}(\PI)\left(\tikz[scale=1.2]{\node{} child{node{} edge from parent[noise] child {node{}}}child{node{} edge from parent[noise] child {node{}}}child{node{} edge from parent[noise] child {node{}}}child[noise]{node{} }}\right).\]
The first two constants from the left are zero because we are taking the expectations over a product of an \emph{odd} number of centred Gaussian variables. On the other hand, using the shorthand notation $K_{\gep}= K*\gr_{\gep}$ we have the identity
\[ g_{\gep}(\PI)\left(\tikz[scale=1.2]{\node{}child{node{} edge from parent[noise] child {node{}}} child{node{} edge from parent[noise] child {node{}}}}\right)= \bE \left[ \int_{\bR^2}  K_{\gep}(-z_1)d\widetilde{W}_{z_1} \right]^2= \int_{\bR^2}  (K_{\gep}(z))^2 dz\,.\]
Applying the Wick's formula for the product of four Gaussian random variables one has
\[\begin{split}
&g_{\gep}(\PI)\left(\tikz[scale=1.2]{\node{} child{node{} edge from parent[noise] child {node{}}}child{node{} edge from parent[noise] child {node{}}}child{node{} edge from parent[noise] child {node{}}}child[noise]{node{} }}\right)=\\&=\bE\left[\int_{\bR^2}  K_{\gep}(-z_1)d\widetilde{W}_{z_1} \int_{\bR^2}  K_{\gep}(-z_2)d\widetilde{W}_{z_2}\int_{\bR^2}  K_{\gep}(-z_3)d\widetilde{W}_{z_3}\int_{\bR^2}  \gr_{\gep}(-z_4)d\widetilde{W}_{z_4}\right]\\&= 3\bE\left[\int_{\bR^2}  K_{\gep}(-z_1)d\widetilde{W}_{z_1} \int_{\bR^2}  K_{\gep}(-z_2)d\widetilde{W}_{z_2}\right]\bE\left[\int_{\bR^2}  K_{\gep}(-z_3)d\widetilde{W}_{z_3}\int_{\bR^2}  \gr_{\gep}(-z_4)d\widetilde{W}_{z_4}\right]\\&=3 \int_{\bR^2}  (K_{\gep}(z))^2 dz\int_{\bR^2}  (K_{\gep}(z))\gr_{\gep}(z) dz=3\, g_{\gep}(\PI)\left(\tikz[scale=1.2]{\node{}child{node{} edge from parent[noise] child {node{}}} child{node{} edge from parent[noise] child {node{}}}}\right) C^{1}_{\gep} \,.
\end{split}\]
By replacing the values of $g_{\gep}(\PI)$ in the above calculations of $\widehat{A}_-\tau_m$ we yield to 
\begin{equation}\label{h_eps2}
h_{\gep}\left(\tikz[scale=1.2]{\node{} child{node{} edge from parent[noise] child {node{}}} child[noise]{node{} }}_{\!\!\!\scriptstyle (0,1)}\right)=h_{\gep}\left(\tikz[scale=1.2]{\node{} child{node{} edge from parent[noise] child {node{}}} child{node{} edge from parent[noise] child {node{}}} child[noise]{node{} }}\right)= h_{\gep}\left(\tikz[scale=1.2]{\node{} child{node{} edge from parent[noise] child {node{}}}child{node{} edge from parent[noise] child {node{}}}child{node{} edge from parent[noise] child {node{}}}child[noise]{node{} }}\right)=0\;.
\end{equation}
Moreover the values of $\widetilde{M}_{\gep}\tau_m$ coincide with \eqref{renom_identities} for $ m\leq 3$. Looking at $\widetilde{M}_{\gep}\tau_m$ if $m> 3$ and $ \widetilde{M}_{\gep}(\cI(\Xi)^m)$, the subforests contained in $\tau_m$ and $\cI(\Xi)^m$ that are not identified to zero in the quotient $\cT_-$ contain at least one of the following trees in $W_-$
\[\left\{\tikz[scale=1.5]{\node{} child[noise]{node{}}}\,, \tikz[scale=1.2]{\node{} child{node{} edge from parent[noise] child {node{}}} child[noise]{node{} }}\, ,\,\tikz[scale=1.2]{\node{} child{node{} edge from parent[noise] child {node{}}}child{node{} edge from parent[noise] child {node{}}}child[noise]{node{} }} \, ,\, \tikz[scale=1.2]{\node{} child{node{} edge from parent[noise] child {node{}}}child{node{} edge from parent[noise] child {node{}}}child{node{} edge from parent[noise] child {node{}}}child[noise]{node{} }}\,, \tikz[scale=1.2]{\node{} child[noise]{ node{}} }_{\scriptstyle (0,1)}\,, \,\tikz[scale=1.5]{\node{} child{node{} edge from parent[noise] child {node{}}} child[noise]{node{} }}_{\!\!\!\!\scriptstyle (0,1)} \right\}\,,\quad \left\{\tikz[scale=1.5]{\node{} child[noise]{node{}}}\,,  \tikz[scale=1.5]{\node{} child[noise]{ node{}} }_{\scriptstyle (0,1)}\right\}\,.\]
Since we know from \eqref{h_eps} \eqref{h_eps2} the values of $h_{\gep}$ on these trees, we obtain the following identity
\[\gD_{-}\tau_{m}= \emptyset\otimes \tau_{m}+ m\;\tikz[scale=1.3]{\node{} child{node{} edge from parent[noise] child {node{}}} child[noise]{node{}}}\otimes \underbrace{\tikz[scale=1.5,level/.style={level distance=1ex, sibling distance=1em)}]{\node{}child{node{} edge from parent[noise] child {node{}}} child{node[label={[label distance=0.1cm]170  : ... }]{} edge from parent[noise] child {node{}}}}}_{\text{m-1 }}+ (\cdots)\,,\quad \gD_{-} \underbrace{\tikz[scale=1.5,level/.style={level distance=1ex, sibling distance=1em)}]{\node{}child{node{} edge from parent[noise] child {node{}}} child{node[label={[label distance=0.1cm]170  : ... }]{} edge from parent[noise] child {node{}}}}}_{\text{m }}= \emptyset\otimes  \underbrace{\tikz[scale=1.5,level/.style={level distance=1ex, sibling distance=1em)}]{\node{}child{node{} edge from parent[noise] child {node{}}} child{node[label={[label distance=0.1cm]170  : ... }]{} edge from parent[noise] child {node{}}}}}_{\text{m }}+ (\cdots) \,,\]
where the term $(\cdots)$ contains some terms that becomes zero after we apply $h_{\gep}$. The combinatorial factor $m$ appears because the tree associated to $\Xi\cI(\Xi)$ appears $m$ times inside $\tau_m$. Therefore we prove the first part of the equations \eqref{renom_identities}. We pass to the terms of the form $\gs_m=\cI_1(\Xi)^2\cI(\Xi)^m$ and $\gh_k= \cI_1(\Xi)\cI(\Xi)^k$ for $m=0$ and $k\leq 1$. Adopting the same notation as before to denote the terms that becomes zero after applying $h_{\gep}$ for $\gD_-$ or $g_{\gep}(\PI)$ for $\widehat{A}_-$, we have 
\[\gD_{-}\tikz[scale=1.2]{\node{}  child[derivative]{node{} edge from parent[noise] child {node{}}}}=\emptyset\otimes \tikz[scale=1.2]{\node{}  child[derivative]{node{} edge from parent[noise] child {node{}}}} +\tikz[scale=1.2]{\node{}  child[derivative]{node{} edge from parent[noise] child {node{}}}}\otimes\1+ (\cdots)\,,\quad \gD_{-}\tikz[scale=1.2]{\node{}  child[derivative]{node{} edge from parent[noise] child {node{}}} child{node{} edge from parent[noise] child {node{}}}}=\emptyset\otimes \tikz[scale=1.2]{\node{}  child{node{} edge from parent[noise] child {node{}}} child[derivative]{node{} edge from parent[noise] child {node{}}}} +\tikz[scale=1.2]{\node{}  child[derivative]{node{} edge from parent[noise] child {node{}}}}\otimes\tikz[scale=1.2]{\node{}  child{node{} edge from parent[noise] child {node{}}}}+\tikz[scale=1.2]{\node{}  child{node{} edge from parent[noise] child {node{}}} child[derivative]{node{} edge from parent[noise] child {node{}}}}\otimes \1 + (\cdots)\,,\]
\[\gD_{-}\tikz[scale=1.2]{\node{} child[derivative]{node{} edge from parent[noise] child {node{}}} child[derivative]{node{} edge from parent[noise] child {node{}}}}=\emptyset\otimes \tikz[scale=1.2]{\node{} child[derivative]{node{} edge from parent[noise] child {node{}}} child[derivative]{node{} edge from parent[noise] child {node{}}}} +2\,\tikz[scale=1.2]{\node{} child[derivative]{node{} edge from parent[noise] child {node{}}} } \otimes \tikz[scale=1.2]{\node{} child[derivative]{node{} edge from parent[noise] child {node{}}} }+\tikz[scale=1.2]{\node{} child[derivative]{node{} edge from parent[noise] child {node{}}} child[derivative]{node{} edge from parent[noise] child {node{}}}}\otimes\1+ (\cdots)\,,\]
\[\widehat{A}_-\tikz[scale=1.2]{\node{}  child[derivative]{node{} edge from parent[noise] child {node{}}}}=-\, \tikz[scale=1.2]{\node{}  child[derivative]{node{} edge from parent[noise] child {node{}}}} +(\cdots)\,,\quad \widehat{A}_-\tikz[scale=1.2]{\node{}  child[derivative]{node{} edge from parent[noise] child {node{}}} child{node{} edge from parent[noise] child {node{}}}}=-\tikz[scale=1.2]{\node{}  child{node{} edge from parent[noise] child {node{}}} child[derivative]{node{} edge from parent[noise] child {node{}}}} +\tikz[scale=1.2]{\node{}  child[derivative]{node{} edge from parent[noise] child {node{}}}}\;\;\tikz[scale=1.2]{\node{}  child{node{} edge from parent[noise] child {node{}}}}+(\cdots)\,,\quad \widehat{A}_-\tikz[scale=1.2]{\node{} child[derivative]{node{} edge from parent[noise] child {node{}}} child[derivative]{node{} edge from parent[noise] child {node{}}}}=-\tikz[scale=1.2]{\node{} child[derivative]{node{} edge from parent[noise] child {node{}}} child[derivative]{node{} edge from parent[noise] child {node{}}}} +2\,\tikz[scale=1.2]{\node{} child[derivative]{node{} edge from parent[noise] child {node{}}} }\;\,\tikz[scale=1.2]{\node{} child[derivative]{node{} edge from parent[noise] child {node{}}} }\,+(\cdots)\,.\]
Applying the map $(\PI^{\gep})'$ we obtain also
\begin{equation}\label{stochastic_semplification2}
\begin{gathered}
g_{\gep}(\PI) \left(\,\tikz{\node{} child[derivative]{node{} child[noise]{node{} }}}\,\right)=0\,,\; g_{\gep}(\PI)\left(\tikz[scale=1.2]{\node{} child[derivative]{node{} edge from parent[noise] child {node{}}} child[derivative]{node{} edge from parent[noise] child {node{}}}}\right)=\bE\left[\int_{\bR^2}   \partial_xK_{\gep}(-z_1)d\widetilde{W}_{z_1}\right]^2= C^2_{\gep}\,,\\ g_{\gep}(\PI)\left(\tikz[scale=1.2]{\node{}child{node{} edge from parent[noise] child {node{}}} child[derivative]{node{} edge from parent[noise] child {node{}}}}\right)= \bE \left[ \int_{\bR^2}  K_{\gep}(-z_1)d\widetilde{W}_{z_1} \int_{\bR^2}  \partial_xK_{\gep}(-z_2)d\widetilde{W}_{z_2}\right]\\= \int_{\bR^2}  K_{\gep}(z)\partial_xK_{\gep}(z) dz=0\,,
\end{gathered}
\end{equation}
where the first and the last identity of \eqref{stochastic_semplification2} are obtained because we take the expectation of a centred Gaussian variable and the function $x\to K_{\gep}(t,x)\partial_xK_{\gep}(t,x)$ is odd in $x$ for any $t>0$. Then we obtain
\begin{equation}\label{h_eps3}
h_{\gep}\left(\;\tikz{\node{} child[derivative]{node{} child[noise]{node{} }}}\;\right)=h_{\gep}\left(\tikz[scale=1.2]{\node{}child{node{} edge from parent[noise] child {node{}}} child[derivative]{node{} edge from parent[noise] child {node{}}}}\right)= 0\,,\quad h_{\gep}\left(\tikz[scale=1.2]{\node{}child[derivative]{node{} edge from parent[noise] child {node{}}} child[derivative]{node{} edge from parent[noise] child {node{}}}}\right)=-C^2_{\gep}\,,
\end{equation}
and consequently the identities \eqref{renom_identities} for $\widetilde{M}_{\gep}\gs_m$ and $\widetilde{M}_{\gep}\gh_k$. Passing to the calculation of $\widetilde{M}_{\gep}\gs_m$ for $m=1,2$ we have
\[\gD_{-}\tikz[scale=1.2]{\node{} child[derivative]{node{} edge from parent[noise] child {node{}}} child[derivative]{node{} edge from parent[noise] child {node{}}}child{node{} edge from parent[noise] child {node{}}}}=\emptyset\otimes \tikz[scale=1.2]{\node{} child[derivative]{node{} edge from parent[noise] child {node{}}} child[derivative]{node{} edge from parent[noise] child {node{}}}child{node{} edge from parent[noise] child {node{}}}} +\tikz[scale=1.2]{\node{} child[derivative]{node{} edge from parent[noise] child {node{}}} child[derivative]{node{} edge from parent[noise] child {node{}}}}\,\otimes \, \tikz[scale=1.2]{\node{} child{node{} edge from parent[noise] child {node{}}}}+ \tikz[scale=1.2]{\node{} child[derivative]{node{} edge from parent[noise] child {node{}}} child[derivative]{node{} edge from parent[noise] child {node{}}}}_{\!\!(0,1)}\otimes \tikz[scale=1.2]{\node{} child[derivative]{node{} edge from parent[noise] child {node{}}}}+\tikz[scale=1.2]{\node{} child[derivative]{node{} edge from parent[noise] child {node{}}} child[derivative]{node{} edge from parent[noise] child {node{}}}child{node{} edge from parent[noise] child {node{}}}}\otimes\1+ (\cdots)\,,\]
\[\begin{split}
\gD_{-}\tikz[scale=1.2]{\node{} child[derivative]{node{} edge from parent[noise] child {node{}}} child[derivative]{node{} edge from parent[noise] child {node{}}} child{node{} edge from parent[noise] child {node{}}}child{node{} edge from parent[noise] child {node{}}}}&=\emptyset\otimes \tikz[scale=1.2]{\node{} child[derivative]{node{} edge from parent[noise] child {node{}}} child[derivative]{node{} edge from parent[noise] child {node{}}}child{node{} edge from parent[noise] child {node{}}} child{node{} edge from parent[noise] child {node{}}} } +\tikz[scale=1.2]{\node{} child[derivative]{node{} edge from parent[noise] child {node{}}} child[derivative]{node{} edge from parent[noise] child {node{}}}}\otimes  \tikz[scale=1.2]{\node{}child{node{} edge from parent[noise] child {node{}}} child{node{} edge from parent[noise] child {node{}}}}+2\;\tikz[scale=1.2]{\node{} child[derivative]{node{} edge from parent[noise] child {node{}}} child[derivative]{node{} edge from parent[noise] child {node{}}}}_{\!\!(0,1)}\otimes  \tikz[scale=1.2]{\node{}child{node{} edge from parent[noise] child {node{}}} child[derivative]{node{} edge from parent[noise] child {node{}}}} +2\,\tikz[scale=1.2]{\node{} child[derivative]{node{} edge from parent[noise] child {node{}}} child[derivative]{node{} edge from parent[noise] child {node{}}}child{node{} edge from parent[noise] child {node{}}}}\otimes \tikz[scale=1.2]{\node{} child{node{} edge from parent[noise] child {node{}}}}+\tikz[scale=1.2]{\node{} child[derivative]{node{} edge from parent[noise] child {node{}}} child[derivative]{node{} edge from parent[noise] child {node{}}}child{node{} edge from parent[noise] child {node{}}}child{node{} edge from parent[noise] child {node{}}}}\otimes\1+ (\cdots)\,,
\end{split}\]
\[\widehat{A}_-\tikz[scale=1.2]{\node{} child[derivative]{node{} edge from parent[noise] child {node{}}} child[derivative]{node{} edge from parent[noise] child {node{}}}child{node{} edge from parent[noise] child {node{}}}}=- \tikz[scale=1.2]{\node{} child[derivative]{node{} edge from parent[noise] child {node{}}} child[derivative]{node{} edge from parent[noise] child {node{}}}child{node{} edge from parent[noise] child {node{}}}} -\widehat{A}_-\left(\tikz{\node{} child[derivative]{node{} edge from parent[noise] child {node{}}} child[derivative]{node{} edge from parent[noise] child {node{}}}}\right)\; \tikz[scale=1.2]{\node{} child{node{} edge from parent[noise] child {node{}}}}+(\cdots)\,,\]
\[\begin{split}
\widehat{A}_-\;\tikz[scale=1.2]{\node{} child[derivative]{node{} edge from parent[noise] child {node{}}} child[derivative]{node{} edge from parent[noise] child {node{}}} child{node{} edge from parent[noise] child {node{}}}child{node{} edge from parent[noise] child {node{}}}}&=-\tikz[scale=1.2]{\node{} child[derivative]{node{} edge from parent[noise] child {node{}}} child[derivative]{node{} edge from parent[noise] child {node{}}}child{node{} edge from parent[noise] child {node{}}} child{node{} edge from parent[noise] child {node{}}} } -\widehat{A}_-\left(\tikz{\node{} child[derivative]{node{} edge from parent[noise] child {node{}}} child[derivative]{node{} edge from parent[noise] child {node{}}}}\right)\; \tikz[scale=1.2]{\node{}child{node{} edge from parent[noise] child {node{}}} child{node{} edge from parent[noise] child {node{}}}}-2\,\widehat{A}_-\left(\tikz[scale=1.2]{\node{} child[derivative]{node{} edge from parent[noise] child {node{}}} child[derivative]{node{} edge from parent[noise] child {node{}}}child{node{} edge from parent[noise] child {node{}}}}\right)\,\tikz[scale=1.2]{\node{}  child{node{} edge from parent[noise] child {node{}}}} +(\cdots)\,.
\end{split}\]
Using again the same notations
\[\gD_-\tikz[scale=1.2]{\node{} child[derivative]{node{} edge from parent[noise] child {node{}}} child[derivative]{node{} edge from parent[noise] child {node{}}}}_{\!\!\!(0,1)} = \emptyset \otimes \tikz[scale=1.2]{\node{} child[derivative]{node{} edge from parent[noise] child {node{}}} child[derivative]{node{} edge from parent[noise] child {node{}}}}_{\!\!\!(0,1)}+ \tikz[scale=1.2]{\node{} child[derivative]{node{} edge from parent[noise] child {node{}}} child[derivative]{node{} edge from parent[noise] child {node{}}}}_{\!\!\!(0,1)}\otimes \1 + (\cdots)\,,\quad \widehat{A}_{-}\left(\,\tikz[scale=1.2]{\node{} child[derivative]{node{} edge from parent[noise] child {node{}}} child[derivative]{node{} edge from parent[noise] child {node{}}}}_{\!\!\!(0,1)}\right)=- \tikz[scale=1.2]{\node{} child[derivative]{node{} edge from parent[noise] child {node{}}} child[derivative]{node{} edge from parent[noise] child {node{}}}}_{\!\!\!(0,1)}+ (\cdots)\,.\]
Similarly as before, we apply Wick's formula and the definition of $(\PI^{\gep})'$ to obtain
\[\begin{split}
g_{\gep}(\PI)\left(\tikz[scale=1.2]{\node{} child[derivative]{node{} edge from parent[noise] child {node{}}} child[derivative]{node{} edge from parent[noise] child {node{}}}}_{\!\!\!(0,1)}\right)&=g_{\gep}(\PI)\left(\tikz[scale=1.2]{\node{} child[derivative]{node{} edge from parent[noise] child {node{}}} child[derivative]{node{} edge from parent[noise] child {node{}}}child{node{} edge from parent[noise] child {node{}}}}\right)=g_{\gep}(\PI)\left(\tikz[scale=1.2]{\node{} child{node{} edge from parent[noise] child {node{}}}}\right)=0\,,\\ g_{\gep}(\PI)\left(\tikz[scale=1.2]{\node{} child{node{} edge from parent[noise] child {node{}}} child{node{} edge from parent[noise] child {node{}}}  child[derivative]{node{} edge from parent[noise] child {node{}}} child[derivative]{node{} edge from parent[noise] child {node{}}}}\right)&= 2\left(g_{\gep}(\PI) \left(\,\tikz{\node{} child[derivative]{node{} child[noise]{node{} }}child{node{} child[noise]{node{} }}}\,\right)\right)^2 +g_{\gep}(\PI) \left(\,\tikz{\node{} child[derivative]{node{} child[noise]{node{} }}child[derivative]{node{} child[noise]{node{} }}}\,\right)g_{\gep}(\PI) \left(\,\tikz{\node{} child{node{} child[noise]{node{} }}child{node{} child[noise]{node{} }}}\,\right)\\&= g_{\gep}(\PI) \left(\,\tikz{\node{} child[derivative]{node{} child[noise]{node{} }}child[derivative]{node{} child[noise]{node{} }}}\,\right)g_{\gep}(\PI) \left(\,\tikz{\node{} child{node{} child[noise]{node{} }}child{node{} child[noise]{node{} }}}\right)\,.
\end{split}\]
Thus yielding finally 
\begin{equation}\label{h_eps4}
h_{\gep}\left(\,\tikz[scale=1.2]{\node{} child[derivative]{node{} edge from parent[noise] child {node{}}} child[derivative]{node{} edge from parent[noise] child {node{}}}}_{\!\!\!(0,1)} \right)= h_{\gep}\left(\tikz[scale=1.2]{\node{} child[derivative]{node{} edge from parent[noise] child {node{}}} child[derivative]{node{} edge from parent[noise] child {node{}}}child{node{} edge from parent[noise] child {node{}}}}\right)=h_{\gep}\left(\tikz[scale=1.2]{\node{} child{node{} edge from parent[noise] child {node{}}} child{node{} edge from parent[noise] child {node{}}}  child[derivative]{node{} edge from parent[noise] child {node{}}} child[derivative]{node{} edge from parent[noise] child {node{}}}}\right)=0\,,
\end{equation}
and the identities \eqref{renom_identities} when $m\leq 2$, $k\leq 1$. In case $m>2$ or $k>1$, the terms in the left factor of $\gD_- \gs_m$ and $\gD_- \gh_k$ are respectively forests composed by the trees
\[\left\{\tikz[scale=1.5]{\node{} child[noise]{node{}}}\,,\tikz[scale=1.5]{\node{} child[noise]{ node{}} }_{(0,1)} \, ,\,\tikz[scale=1.2]{\node{}  child[derivative]{node{} edge from parent[noise] child {node{}}}}\,,\, \tikz[scale=1.2]{\node{}  child[derivative]{node{} edge from parent[noise] child {node{}}}child[derivative]{node{} edge from parent[noise] child {node{}}}} \,,\tikz[scale=1.2]{\node{} child[derivative]{node{} edge from parent[noise] child {node{}}} child[derivative]{node{} edge from parent[noise] child {node{}}}}_{\!\!\!(0,1)}\,,\tikz[scale=1.2]{\node{}  child[derivative]{node{} edge from parent[noise] child {node{}}}child{node{} edge from parent[noise] child {node{}}}}\,, \tikz[scale=1.2]{\node{}  child[derivative]{node{} edge from parent[noise] child {node{}}} child[derivative]{node{} edge from parent[noise] child {node{}}} child{node{} edge from parent[noise] child {node{}}}}\,,\tikz[scale=1.2]{\node{}  child[derivative]{node{} edge from parent[noise] child {node{}}} child[derivative]{node{} edge from parent[noise] child {node{}}} child{node{} edge from parent[noise] child {node{}}}child{node{} edge from parent[noise] child {node{}}}}\,\,\right\}\,\quad \text {or} \quad\left\{\tikz[scale=1.5]{\node{} child[noise]{node{}}}\,,  \tikz[scale=1.5]{\node{} child[noise]{ node{}} }_{(0,1)}\,,\tikz[scale=1.2]{\node{} child[derivative]{child[noise]{node{}}}}\,,\tikz[scale=1.2]{\node{}  child[derivative]{node{} edge from parent[noise] child {node{}}}child{node{} edge from parent[noise] child {node{}}}}\right\}\,.\]
Applying the identities \eqref{h_eps} \eqref{h_eps3} and \eqref{h_eps4}, the only relevant terms in the sums become
\[\gD_- \gs_m= \emptyset \otimes \gs_m+ \tikz[scale=1.2]{\node{} child[derivative]{node{} edge from parent[noise] child {node{}}} child[derivative]{node{} edge from parent[noise] child {node{}}}}\otimes \underbrace{\tikz[scale=1.5,level/.style={level distance=1ex, sibling distance=1em)}]{\node{}child{node{} edge from parent[noise] child {node{}}} child{node[label={[label distance=0.1cm]170  : ... }]{} edge from parent[noise] child {node{}}}}}_{\text{m }}\,+ (\cdots)\,,\quad \gD_- \gh_k=\emptyset\otimes \gh_k+ (\cdots)\,.\]
Thus we obtain the final part of the identities \eqref{renom_identities} and we conclude.
\end{proof}

We will henceforth fix the parameter $\kappa$ in order to keep the Theorem \ref{BPHZ_explic} true. By construction of the BPHZ renormalisation (see \cite[Sec. 6]{Bruned2019}) the application $\widetilde{M}_{\gep}$ is an admissible renormalisation scheme and, denoting by $\hPI^{\gep}=\PI\widetilde{M}_{\gep}$, the couple $\cL(\hPI^{\gep})= (\hPi^{\gep},\hgG^{\gep})$ obtained from the Remark \ref{operator L} is always a model for any $\gep>0$. The explicit form of the map $\widetilde{M}_{\gep}$ obtained in the Theorem \ref{BPHZ_explic} allows us to write explicitly also $(\hPi^{\gep},\hgG^{\gep})$.
\begin{proposition}\label{interaction_model}
For any $z\in \bR^2$ and $z,z'\in\bR^2$ one has
\begin{equation}\label{interaction}
\hPi_z^{\gep}=\Pi^{\gep}_zM_{\gep}\,,\qquad \hgG_{zz'}^{\gep}=\gG_{zz'}^{\gep}\,.
\end{equation}
Furthermore the model $(\hPi^{\gep},\hgG^{\gep})$ is also adapted to the action of translation on $\bR$
\end{proposition}
\begin{proof}
By definition of $\cL(\hPI^{\gep})$, the model $(\hPi^{\gep},\hgG^{\gep})$ can be represented as the couple $(\hPI^{\gep}, \hat{f}_{\gep})$, where the function $\hat{f}_{\gep}\colon \bR^2\to \bR^3 $ is defined as
\begin{equation}
\hat{f}_{\gep}(z)_i= - z_i\,,\; i=1,2 \quad \hat{f}_{\gep}(z)_3= -(K* \hPI^{\gep} \Xi )(z)= -(K* \PI^{\gep} \Xi )(z)\,.
\end{equation}
Thus the function $\hat{f}_{\gep}$ coincides with $f_{\gep}$, the same function obtained from the decomposition of the canonical model $(\Pi^{\gep},\gG^{\gep})$ as $(\PI^{\gep}, f_{\gep})$. By definition of $\gG$ we have straightforwardly $\hgG_{zz'}^{\gep}=\gG_{zz'}^{\gep}$. In case of  $\hPi^{\gep}$ we can apply immediately the identity \eqref{PI_and_Pi} with the Theorem \ref{BPHZ_explic}  to obtain
\[ \hPi_z^{\gep}= \hPI^{\gep}\gG_{\hat{f}_{\gep}(z)}= \PI^{\gep}M_{\gep} \gG_{\hat{f}_{\gep}(z)}= \PI^{\gep}M_{\gep} \gG_{f_{\gep}(z)}\,.\]
Then the formula \eqref{interaction} holds as long as for any $h\in \bR^3$ and $\gep>0$ one has
\begin{equation}\label{cointeraction}
M_{\gep} \gG_h=\gG_h M_{ \gep}\;.
\end{equation}
Let us verify the identity \eqref{cointeraction} for any $\tau \in V_{\Xi}\sqcup V_{\cI_1(\Xi)^2}\sqcup V_{\cI_1(\Xi)}\sqcup U$. In case $\tau \in V_{\cI_1(\Xi)}\sqcup U $, this identity holds trivially because $\gG_h$ leaves invariant the subspace $\cV_{\cI_1(\Xi)}\sqcup \cU $ (see equation \eqref{explicit_gamma}) and $M_{\gep}$ is the identity when it is restricted to this subspace. On the other hand if $\tau \in V_{\Xi}\sqcup V_{\cI_1(\Xi)^2}$, the multiplicative property of $\gG_h$ (see the Remark \ref{multiplicative_gamma}) 
and the behaviour of $M_{\gep}$ on the polynomials in \eqref{defn_M} reduces to verify \eqref{cointeraction} over the symbols $\cI_1(\Xi)^2\cI(\Xi)^m$ and $\Xi \cI(\Xi)^m$ for any $m\geq 1$. Writing $h=(h_1,h_2,h_3)$ we obtain 
\[\begin{split}
M_{\gep} \gG_h (\cI_1(\Xi)^2\cI(\Xi)^m)=(\cI_1(\Xi)^2 - C^2_{\gep}) \sum_{n=0}^m\binom{m}{n} \cI(\Xi)^nh_3^{m-n}=  \gG_hM_{ \gep} (\cI_1(\Xi)^2  \cI(\Xi)^m)\,,\end{split}\]
\[\begin{split}M_{ \gep} \gG_h (\Xi\cI(\Xi)^m)&=\sum_{n=0}^m\binom{m}{n}(\Xi\cI(\Xi)^n- nC^1_{\gep}\cI(\Xi)^{n-1} )h_3^{m-n}\\&=\sum_{n=0}^m\binom{m}{n} \Xi\cI(\Xi)^nh_3^{m-n}- m C^1_{\gep}\sum_{n'=0}^{m-1}\binom{m-1}{n'} \cI(\Xi)^{n'} h_3^{m-1-n'}\\& =\Xi(\cI(\Xi)+ h_3 \1)^m - m C^1_{\gep} (\cI(\Xi)+ h_3\1)^{m-1}=  \gG_hM_{ \gep}(\Xi  \cI(\Xi)^m)\,.\end{split}\]
Thus yielding the result. The identity \eqref{interaction} implies immediately the properties in the identity \eqref{periodic}. Therefore $(\hPi^{\gep}, \hgG^{\gep})$ is adapted to the action of translations. 
\end{proof}
We study the convergence of $ \cL(\hPi^{\gep})$ in the space of models. Embedding the regularity structure $\cT$ into $\cT'$ as explained in the Proposition \ref{inclusion}, it is possible to prove the convergence of $ \cL(\hPi^{\gep})$ using the general criterion exposed in \cite[Thm. 2.15]{chandra_analytic16}. We introduce some notation to apply this statement. Representing all the elements  $\tau \in T$ as decorated trees, we denote by $E_{\Xi}(\tau)$ the set of edges labelled by $\Xi$. By construction every element $e\in E_{\Xi}(\tau)$ is written uniquely as $e=\{e_{\Xi}, e^{\Xi}\}$, where $e^{\Xi}$ is one of the leaves of $\tau$.  This decomposition allows to define the sets
\[N_{\Xi}(\tau):=\{e_{\Xi}\colon e\in E_{\Xi}(\tau)\}\,,\quad N^{\Xi}(\tau):=\{e^{\Xi}\colon e\in E_{\Xi}(\tau)\}\,,\quad N(\tau):= N_{\tau}\setminus N^{\Xi}(\tau)\,.\]
Moreover, expressing $\tau$ as $\tau^{\f{n}}_{\f{e}} $  for some decoration $\f{n}$, $\f{e}$ we write $\tau^{0}_{\f{e}}$ to denote the decorated tree whose decoration $\f{n}$ is replaced by zero in every node. Let us express the convergence theorem in this context.
\begin{theorem}\label{convergence}
There exists a random model $(\hat{\Pi},\hat{\gG})$ such that
\begin{equation}
(\hat{\Pi}_{\gep},\hat{\gG}_{\gep})\overset{\bP}{\to}(\hat{\Pi},\hat{\gG})
\end{equation}
with respect to the metric $\norm{\cdot}_{\cM}$. We call $(\hat{\Pi},\hat{\gG})$ the BPHZ model.
\end{theorem}
\begin{proof}
This theorem is a direct consequence of \cite[Thm. 2.15]{chandra_analytic16}. Expressing the hypotheses of this theorem in our context, we obtain the thesis after checking  for any $\tau^{\f{n}}_{\f{e}} \in T$ and every subtree $\gs^{\f{n}}_{\f{e}}$ included in $\tau^{\f{n}}_{\f{e}} $ satisfying $\sharp\{N(\gs)\}\geq 2$ the following property:
\begin{itemize}
\item[1)] For any non-empty subset $A\subset E_{\Xi}(\tau)$  such that $\sharp\{ A\}+ \sharp \{N_{\Xi}(\gs)\}$ is even, one has 
\begin{equation}\label{check}
|\gs^{0}_{\f{e}}|_{\mathfrak{s}} + \sum_{e\in A}\mathfrak{s} (\mathfrak{t}(e))+3\sharp\{ A\}   > 0\,,
\end{equation}
\item[2)] For any $e\in E_{\Xi}$, $|\gs^{0}_{\f{e}}|_{\mathfrak{s}}-\mathfrak{s} (\mathfrak{t}(e))>0$.
\item[3)]$|\gs^{0}_{\f{e}}|_{\mathfrak{s}} > -3/2$.
\end{itemize}
Since $\mathfrak{t}(e)=\Xi$ for every $e\in E_{\Xi}$ and $\mathfrak{s} (\Xi)=-3/2-\kappa$, it is sufficient to prove:
\begin{itemize}
\item[1')] For any non-empty subset $A\subset E_{\Xi}(\tau)$  such that $\sharp\{ A\}+ \sharp \{N_{\Xi}(\gs)\}$ is even, one has 
\[ |\gs^{0}_{\f{e}}|_{\mathfrak{s}}+\left(\frac{3}{2}-\kappa\right)\sharp\{ A\}    > 0\,.\]
\item[2')] $|\gs^{0}_{\f{e}}|_{\mathfrak{s}} > -\frac{3}{2}$ for any every subtree $\gs^{\f{n}}_{\f{e}}$ included in $\tau^{\f{n}}_{\f{e}} $  satisfying $\sharp\{N(\gs)\}\geq 2$.
\end{itemize}
It follows easily from the structure of $T$  that the trees $\gs^{\f{n}}_{\f{e}}$ satisfying $\sharp\{N(\gs)\}\geq 2$ such that $|\gs^{0}_{\f{e}}|_{\mathfrak{s}}$ is minimal are given by
\[\gs^{0}_{\f{e}}=\tikz[scale=1.2]{\node{} child{node{} edge from parent[noise] child {node{}}} child[noise]{node{} }}\, \;\text{or}\;\,\gs^{0}_{\f{e}}= \tikz[scale=1.2]{\node{}  child[derivative]{node{} edge from parent[noise] child {node{}}}child[derivative]{node{} edge from parent[noise] child {node{}}}} \,,\]
where $|\gs^{0}_{\f{e}}|_{\mathfrak{s}}=-1 -2\kappa$ in both cases. Therefore we can chose the parameter $\kappa>0$ sufficiently small to satisfy 1') and 2') trivially.
\end{proof}
\begin{remark}\label{Ito_model}
The BPHZ model $(\hat{\Pi},\hat{\gG})$ obtained from the Theorem \ref{convergence} is an example of a random model with a.s. values on distributions and it will be the main object to formulate a new type of It\^o formula for $u$. Recalling the inclusion of the space $\cV_{\Xi}\oplus \cU$ into the regularity structure $\cT_{HP}$ defined in \cite{MartinHairer2015} (see the Remark \ref{positive_renom}), we can easily check that the renormalisation map $M_{\gep}$ defined in \eqref{defn_M} and the model $(\hat{\Pi}_{\gep},\hat{\gG}_{\gep})$ restricted to the sector $\cV_{\Xi}\oplus \cU$ coincide exactly with the renormalisation procedure developed in \cite[Thm. 4.5]{MartinHairer2015} to define what in this context is called the \emph{It\^o model}. By uniqueness of the limit on this sector we can apply directly this result and we obtain immediately $\hPi_z\Xi= \widetilde{\xi}$, the periodic extension of $\xi$ and for every $\tau\in U$, $z=(t,x)\in \bR^2$ and every smooth test function $\psi$ such that for any $s< t $ $\psi(s,y)=0$ we have immediately 
\begin{equation}\label{stochastic integral}
\hPi_z  (\tau\Xi)(\psi)= \int_t^{\infty}\int_{\bR}\hPi_{z}\tau(s,y)\psi(s,y)d\widetilde{W}_{s,y}
\end{equation}
We stress that equation \eqref{stochastic integral} holds only when the test function is supported in the future. Otherwise, the right-hand side integrand will not be adapted and we cannot interpret $\hPi_z \Xi \tau(\psi)$ as a Skorohod integral. An explicit formula to describe the law of $\hPi_z\tau$ in its full generality has been developed in \cite[Prop 4.22]{chandra_analytic16}. Moreover we recall that in our case we have $\hgG_{zz'}= \gG_{\hat{f}(z')- \hat{f}(z)}$ where $\hat{f}\colon \bR^2\to \bR^3$ is given by
\[\hat{f}(z)_i= - z_i\,,\; i=1,2 \quad \hat{f}(z)_3= -(K* \widetilde{\xi} )(z)\,.\]
The model  $(\hat{\Pi},\hat{\gG})$ is also adapted to the action of translation on $\bR$, as a consequence of the Proposition \ref{interaction_model} on the converging sequence $(\hPi^{\gep}, \hgG^{\gep})$.
\end{remark}
\section{Calculus on regularity structures}\label{calculus_reg_struct}
In this section, we will show how the models $(\hat{\Pi}^{\gep},\hgG^{\gep})$ and $(\hat{\Pi},\hgG)$ can be used to describe respectively  $u_{\gep}$ and $u$ and, more generally, what kind of analytical operations we can define on a the regularity structure $\cT$.
\subsection{Modelled distributions}
The main function of a regularity structure and a model upon that is to provide a coherent framework to approximate random distributions similar to how polynomials approximate smooth functions via Taylor's formula. Since for any function $f\colon \bR\to \bR$ it is possible to describe the condition $f\in\cC^{\gga}$ in terms of $F\colon \bR^2\to \bR^{\lfloor\gga\rfloor}$, the vector of its derivatives (see \cite{hairer_phi15} for further details), we introduce an equivalent version of this space in our general context.
\begin{definition}\label{defn_dgamma}
For any given parameters $\gga>0 $, $\gh\in (-2,\gga)$ and   $(\Pi,\gG)$ a  model upon $ (\cA, \cT, \cG)$, we define $\cD^{\gamma,\eta}$ as the set of all function $U \colon \bR^2 \to\cQ_{<\gga}\cT$ such that for every compact set $\cK\subset \bR^2$, one has
\begin{equation}\label{norm_dgamma}
\vert U \vert_{\gamma,\eta} := \sup_{z \in \cK}\sup_{\alpha < \gamma} \frac{\abs{ U(z)}_{\alpha}}{|t|^{(\frac{\eta - \alpha}{2})\wedge 0}} + \sup_{(z,z') \in \cK^{(2)}}\sup_{\alpha < \gamma} \frac{\vert  U(z) - \Gamma_{zz'} U(z')\vert_{\alpha}} {\bigl(|t|\wedge |t'|\bigr)^{\frac{\eta-\gamma}{2}} |z-z'|^{\gamma - \alpha}} < +\infty\;,
\end{equation}
where $\cK^{(2)}$ denotes the set of points $(z,z')\in \cK^{2}$ such that $\abs{z-z'}\leq 1/2\sqrt{\abs{t}\wedge \abs{t'}}$. The elements of $\cD^{\gamma,\eta}$ are called \emph{modelled distributions}.
\end{definition}
\begin{remark}
The definition of the set $\cD^{\gga,\gh}$ does depend crucially on the underlying model $(\Pi, \gG)$. To remark this dependency, we will adopt for the same set the alternative notation $\cD^{\gga,\gh}(\gG)\,$. Similarly we recall that the quantities $\Vert U \Vert_{\gamma,\eta}$ depend on the compact set $\cK$ but we avoid to put the symbol $\cK$ in the notation because of our finite time horizon setting, we will henceforth prove the results on a fixed compact set $\cK\subset \bR^2$ containing $[0, T]\times [0,1]$. The presence of an extra parameter $\gh$ allows more freedom than the classical $\cC^{\gga}$ spaces. In this way the coordinates of $U$ are allowed to blow at rate $\gh$ near the set $P=\{(t,x)\in \bR^2\colon \;t=0\}$ and  the condition $\gh>-2$ is put to keep this singularity integrable. By definition of $ \cD^{\gga,\gh} $, for any value $\gga\geq \gga'>0$ and $U\in \cD^{\gga,\gh}$ the projection $\cQ_{<\gga'}U\in \cD^{\gga',\gh}$.
\end{remark}
For any given model $(\Pi, \gG)$ the couple $(\cD^{\gga,\gh}(\gG),\vert \cdot \vert_{\gamma,\eta} )$ is clearly a Banach space. Since we will consider modelled distributions belonging to different models, for any couple of models $(\Pi,\gG)$ and $(\Pi',\gG')$ and modelled distributions $U\in \cD^{\gga,\gh}(\gG)$, $U'\in\cD^{\gga,\gh}(\gG')$ we define the quantity
\[
\norm{U, U'}_{\gga,\gh}:=\sup_{z,w, \ga} \frac{\vert  U(z) - U'(z)- \gG_{zw}U(w)+ \gG'_{zw}U'(w)\vert_{\alpha}} {\bigl(|t|\wedge |t'|\bigr)^{\frac{\eta-\gamma}{2}} |z-z'|^{\gamma - \alpha}} +\sup_{z, \ga} \frac{\abs{ U'(z)- U'(z)}_{\alpha}}{|t|^{(\frac{\eta - \alpha}{2})\wedge 0}}  \;,
\]
where the parameters $z,w,\ga$ belong to the same sets as the quantity \eqref{norm_dgamma}. This function, together with the norm $\norm{\cdot}_{\cM}$ on models endows the fibred space.
\[\cM \ltimes \cD^{\gga,\gh}:=\{((\Pi,\gG),U)\colon (\Pi,\gG)\in \cM,\;\; U\in \cD^{\gga,\gh}(\gG)\}\]
of a complete metric structure using the distance $\norm{\cdot,\cdot}_{\gga,\gh}+ \norm{\cdot,\cdot }_{\cM}$. Combining the knowledge of a model $(\Pi, \gG)\in \cM $ and $U\in \cD^{\gga,\gh}(\gG) $, it is possible to define uniquely a distribution such that the coordinates of $U$ have the same role of the derivatives of a function in the Taylor's formula. This association is called the reconstruction theorem and it is one of the main theorem in the theory of regularity structures (for its proof see \cite[Sec. 3, Sec. 6]{Hairer2014}).
\begin{theorem}[Reconstruction theorem]\label{reconstruction_theorem}
For any $(\Pi,\gG)\in \cM$ there exists a unique map $\cR\colon \cD^{\gga,\gh}(\gG)\to \cS'(\bR^{2})$, called the reconstruction operator, satisfying the following properties:
\begin{itemize}
\item (Generalised Taylor expansion) for any compact set $\cK\subset\bR^2$ there exists a constant $C>0$ such that 
\begin{equation}\label{taylor_decom}
\abs{\left(\cR U- \Pi_z U(z)\right)( \gh_z^{\gl})}\leq C \gl^{\gga}
\end{equation}
uniformly over $\gh\in \cB_2$, $\gl \in (0,1]$ and $z\in \cK$;
\item the distribution $\cR U\in \cC^{\ga_U\wedge \gh}$ where $\ga_U:= \min\{a\in \cA \colon \cQ_a U\neq 0\}$ and in case $\ga\wedge \gh=0 $ we set by convention $\cC^{0}$ the space of locally bounded functions;
\item (local  Lipschitz property) for any fixed $R>0$ and all couples $(\Pi',\gG'),(\Pi,\gG)\in \cM $, $U\in \cD^{\gga,\gh}(\gG)$, $U'\in\cD^{\gga,\gh}(\gG')$ such that $\norm{U; U'}_{\gga,\gh} +  \norm{(\Pi,\gG); (\Pi',\gG') }_{\cM}< R$ and $\ga_U= \ga_{U'}= \ga$, denoting by $\cR$ and $\cR'$ the respective reconstruction operators, there exists a constant $C>0$ depending on $R$ such that
\begin{equation}\label{local_lipschitz}
\norm{\cR' U' -\cR U}_{\cC^{\ga\wedge \gh}}\leq C \left( \norm{U, U'}_{\gga,\gh} +  \norm{(\Pi,\gG), (\Pi',\gG') }_{\cM}\right)\,.
\end{equation}
\end{itemize}
\end{theorem}
\begin{remark}\label{reconstruction_more}
The reconstruction map has in some rare cases an explicit expression. For instance if $\Pi_z \tau $ is a continuous function for every $\tau\in T$ (like the model $ \cL(\hPI^{\gep})$ or $\cL(\PI^{\gep})$ for any $\gep>0$) and $U\in \cD^{\gga,\gh}(\gG)$, then  $\cR U$ is a continuous function given explicitly by
\begin{equation}\label{simpl_reconstruct}
\cR(U)(z)= \Pi_z (U(z))(z)\,.
\end{equation}
Introducing the space $\cD^{\gga,\gh}_{\cU}$ of all modelled distributions taking values in $\cU$, the identity \eqref{simpl_reconstruct} holds also if $(\Pi, \gG)$ is a generic model and $U\in \cD^{\gga,\gh}_{\cU}(\gG)$, because the elements of the canonical basis of $\cU$ have all non negative homogeneity (see for further details in \cite[Sec. 3.4]{Hairer2014}). We finally conclude that for any value $\gga\geq \gga'>0$ and $U\in \cD^{\gga,\gh}$ we have the identity $\cR\cQ_{<\gga'}U= \cR U$, therefore to define correctly the distribution $\cR U$ is sufficient to fix $\gga>0$ such that $ \cQ_{<\gga}\cT$ is generated by  the set $\{\tau \in T\colon \abs{\tau}\leq 0\}$.
\end{remark}
\begin{remark}
Concerning the regularity of $\cR U$, the result stated in the Theorem \ref{reconstruction_theorem} is optimal because of the presence the parameter $\gh$ in the definition and the possible explosion of the components of $U$. However if we forget the behaviour at $0$ it is also possible to prove $\cR U\in \cC^{\gb_{U} }(\bR^2\setminus P)$ where $\gb_U:= \min\{a\in \cA\setminus \bN \colon \cQ_a U\neq 0\}$ (see \cite[Sec. 6]{Hairer2014}) and the local Lipschitz property \eqref{local_lipschitz} holds on the same space $\cC^{\gb_{U} }(\bR^2\setminus P)$. We stress that the local Lipschitz continuity of the reconstruction operator $\cR$ as given in \eqref{local_lipschitz} is only apparent in the proofs contained in \cite[Sec. 3, Sec. 6]{Hairer2014} but not explicitly mentioned in the statements.
\end{remark}
\begin{remark}
If a model $(\Pi, \gG)$ is adapted to the action of the translations (see the equations \eqref{periodic}) we and the function $U$ is periodic in the space variable on $\bR^2$, then  using the general result \cite[Prop. 3.38]{Hairer2014} we obtain also $\cR U=\widetilde{u}$ for some $u \in \cC^{\ga_U\wedge \gh}(\bR\times \bT)$, with an abuse of notation we can identify $\cR U$ with $u$. 
\end{remark}
In what follows, we will denote by $\hat{\cR}_{\gep}$  and $\hat{\cR}$ the reconstruction operator associated to the space $\cD^{\gga,\gh}(\hgG^{\gep})$  and $\cD^{\gga,\gh}(\hgG)$. Using the shorthand notation $\1_{+}= \1_{(0,\infty)\times \bR}$, we introduce the function $\1_{+}\Xi\colon\bR^2\to \cT $, defined for any $z=(t,x)\in \bR^2$
\[(\1_{+}\Xi)(z):=\1_{+}(z)\Xi= \left\{\begin{array}{cl}\Xi & \text{if } t>0, \\
      0 & \text{Otherwise.}
\end{array}
\right. \]
For any fixed realisation of $\xi$ and any choice of the parameters $\gga>0$ and $-2<\gh<\gga$, the definitions of $\hgG$ and $\gG^{\gep}$ implies immediately $\1_{+}\Xi\in \cD^{\gga,\gh}(\hgG^{\gep})$ for all $\gep>0$ and $ \1_{+}\Xi\in \cD^{\gga,\gh}(\hgG)$. The reconstruction of $\1_{+}\Xi$ in both cases can be explicitly calculated.
\begin{proposition}
for any $z\in [0,T]\times \bT$ one has 
\begin{equation}\label{xi_reconstruct}
\hat{\cR}_{\gep}(\1_{+}\Xi)(z)=\1_{+}(z)\xi_{\gep}(z)\,,\quad \hat{\cR}(\1_{+}\Xi)= \1_{[0,\infty)}\xi\,.
\end{equation}
where the second identity holds a.s. as distributions.
\end{proposition}
\begin{proof}
As we recalled in the Remark \ref{reconstruction_more}, to prove the first part \eqref{xi_reconstruct} we can apply directly the identity \eqref{simpl_reconstruct} obtaining the result trivially. Using the Theorem \ref{convergence}, related to the convergence of models and the local Lipschitz continuity of the reconstruction map, the distribution $\hat{\cR}_{\gep}(\1_{+}\Xi) $ converges in probability to $\hat{\cR}(\1_{[0,+\infty)}\Xi)$ with respect to the topology of $\cC^{-3/2-\kappa}(\bR\times \bT)$. Since $\xi_{\gep}$ converges in probability to $ \xi$ with respect to the topology $\cC^{-3/2-\kappa}(\bR\times \bT)$ (see \cite[Lem 10.2]{Hairer2014}) and the operator $\1_{[0,+\infty)}$ (introduced in the section \ref{elements}) extends continuously the multiplication with the indicator $\1_{+}$, then $\1_{+}(z)\xi_{\gep}$ converges in probability to $\1_{[0,+\infty)}\xi$ with respect to the same topology. We conclude by uniqueness of the limit.
\end{proof}

\begin{remark}\label{xi_canonical}
Denoting by $\cR_{\gep}$ the reconstruction operator with respect to the canonical model $(\Pi^{\gep},\gG^{\gep})$ we have also $\1_{+}\Xi\in \cD^{\gga,\gh}(\gG^{\gep})$ and $\cR_{\gep}\1_{+}\Xi=\hat{\cR}_{\gep}\1_{+}\Xi$, because $\hPi_{z}^{\gep}\Xi= \Pi_{z}^{\gep}\Xi$ for any $\gep>0$. Using the same argument above one has $\cR_{\gep}\1_{+}\Xi$ converges in probability to $\1_{[0,+\infty)}\xi$ as before. Nevertheless the sequence $(\Pi^{\gep},\gG^{\gep})$ does not converge and we cannot interpret  $\1_{[0,+\infty)}\xi$ as the reconstruction of some modelled distribution, unless we study the model $(\hPi^{\gep},\hgG^{\gep})$.
\end{remark}
\subsection{Operations with the stochastic heat equation}
Although modelled distributions look very unusual, the reconstruction theorem associates to them a distribution, which is a classical analytical object. Under this identification, it is possible to lift some operations on the $\cC^{\gga}$ spaces directly at the level of the modelled distributions as it was explained in detail in \cite[Sec. 4, 5, 6]{Hairer2014}. Moreover, this ``lifting" procedure is also continuous with respects to the intrinsic topology of the modelled distributions. In what follows, we will briefly recall them to put them in relation with the stochastic heat equation.
\subsubsection*{Convolution}
The first operation to define is the convolution with $G$, the heat kernel on $\bR$. In other terms, we analyse under which conditions we can associate to any $((\Pi,\gG),V)\in \cM \ltimes \cD^{\gga,\gh}$ one modelled distribution $\cP(V)\in \cD^{\bar{\gga},\bar{\eta}}(\Pi)$ in a continuous way such that
\begin{equation}\label{convolution}
\cR(\cP(V))=G*\cR V\,.
\end{equation}
For our purposes we are not interested to describe this operation in general. Indeed recalling the formulae \eqref{analytic_representation} and \eqref{xi_reconstruct}, it is sufficient to define $\cP$ only in the case of the modelled distribution $V=\1_{+}\Xi$ to have an expression of $u_{\gep}$ and $u$, the solution of \eqref{eq3} and \eqref{eqSHE}, as the reconstruction of some modelled distributions. In this case, we can restate the convolution with $G$ as the convolution with two other kernels thanks to this technical lemma (the proof is a direct consequence of \cite[Lemma 7.7]{Hairer2014}).
\begin{lemma}[Second decomposition]\label{KandR2}
For any fixed $T>0$, there exists a function $\bar{R}\colon\bR^2\to \bR $ such that
\begin{itemize}
\item For every distribution $v\in \cC^{\gb}(\bR\times \bT)$ with $\gb>-2$ non integer and supported on $[0,+\infty)$ one has 
\begin{equation}\label{sum_K+R}
(G*\widetilde{v} )(z) = (K*\widetilde{v})(z) + (\bar{R} * \widetilde{v})(z)\,,
\end{equation}
where $K$ is the function introduced in the Lemma \ref{KandR1}, $z\in  (-\infty, T+1] \times \bR $ and $\widetilde{v}$ is the periodic extension of $v$.
\item $\bar{R}$ is smooth, $\bar{R}(t,x) = 0$ for $t \leq 0$ and it is compactly supported.
\end{itemize}
\end{lemma}
Thanks to this decomposition, it is sufficient to write $\cP=\mathfrak{K}+\mathfrak{R}$ for some operators $\mathfrak{K}$ and $\mathfrak{R}$ satisfying
\begin{equation}\label{reconstruction_KandR}
\cR(\mathfrak{K}(V))=K*\cR V\,,\quad \cR (\mathfrak{R} (V))= \bar{R}*\cR V\,.
\end{equation}
Considering the case of $\mathfrak{R}$, we remark that $\bar{R}*v$ will always be a smooth function for any distribution $v$ supported on positive times. Thus for any fixed couple $((\Pi,\gG),V)\in \cM \ltimes \cD^{\gga,\gh}$ such that  $\cR V$ is supported on $\bR_+\times \bR$, the operator $\mathfrak{R}$ can be easily defined for any $\bar{\gga}>0$ as the lifting of the $\bar{\gga}$-th order Taylor polynomial of $\bar{R}*\cR V$, that is:
\begin{equation}\label{operatorR}
 \mathfrak{R}(V)(z):= \sum_{\abs{k}< \bar{\gga}}(\partial^k \bar{R}*(\cR V))(z)\frac{\X^k}{k!}\,.
\end{equation}
From this definition it is straightforward to check $\mathfrak{R}(V) \in \cD^{\bar{\gga},\bar{\gh}}(\Pi)$ for any $ \bar{\gga}>0$, $ -2<\bar{\gh}< \bar{\gga}$ and $\mathfrak{R}(V)$ satisfies the second identity of \eqref{reconstruction_KandR}. Moreover the application $\mathfrak{R} \colon \cM \ltimes \cD^{\gga,\gh}\to  \cD^{\bar{\gga},\bar{\gh}}$ is also continuous with respect to the topology of $\cM \ltimes \cD^{\gga,\gh}$, as a consequence of \cite[Lem. 7.3]{Hairer2014}. This continuity property is a consequence of the compact support of $\bar{R} $ and it is the main reason to introduce Lemma \ref{KandR1}. On the other hand, the kernel $K$ is not a smooth and the definition of $ \mathfrak{K}$ depends on the model, as a consequence of this general result (for its proof see the ``Extension theorem" \cite[Thm. 5.14]{Hairer2014} and the ``Multi-level Schauder estimates" \cite[Thm. 5.14, Prop 6.16]{Hairer2014}).
\begin{proposition}\label{convolution_K}
For any  couple $((\Pi,\gG),V)\in \cM \ltimes \cD^{\gga,\gh}$ where $(\Pi,\gG)$ is of the form $\cL(\PI)$ for some admissible map $\PI$ and $\gga>0 $, $3/2-\kappa<\gh<\gga$ are not integers, there exists a regularity structure $(\cA_2,\mathcal{T}_2, \cG_2)$ including $(\cA,\mathcal{T}, \cG)$, a linear map $\widetilde{I}\colon \cT\to \cT_2$ satisfying $\widetilde{I}(\Xi)=\cI(\Xi) $ and a model $(\Pi^{2}, \gG^{2})$ extending $(\Pi,\gG) $ on $\cT_2$ such that, imposing $\bar{\gga}= \gga+2$, $\bar{\eta}=\ga_V\wedge \gh+ 2$, where $\ga_V=\min\{a\in \cA \colon \cQ_a V\neq 0\}$ the applications $\cN \colon \cM \ltimes \cD^{\gga,\gh}\to \cD^{\bar{\gga},\bar{\eta}}_{\cU}$,  $J\colon \cT\to \cD^{\bar{\gga},\bar{\eta}}_{\cU}$
\begin{equation}\label{operatorN}
\cN (V)(z):=\sum_{\abs{k}< \gga+2}\bigg((\partial^k K)*(\cR V- \Pi_zV(z))\bigg) (z)\frac{\X^k}{k!} \,,
\end{equation}
\begin{equation}
J(z)\tau:=\sum_{\abs{k}< \abs{\tau}+2}\bigg((\partial^k K*\Pi_z\tau)\bigg) (z)\frac{\X^k}{k!}\,,
\end{equation}
are well defined and the application
\begin{equation}\label{defn_conv_K}
\mathfrak{K}(V)(z):= \cQ_{<\bar{\gga}}(\widetilde{I} (V))(z)+ J(z)V(z)+ \cN (V)(z)\,,
\end{equation}
is a map $\mathfrak{K}\colon \cM \ltimes \cD^{\gga,\gh}\to \cD^{\bar{\gga},\bar{\eta}}(\gG_2)$ satisfying the first identity of \eqref{reconstruction_KandR} without any restriction on the support of $\cR(V)$. Moreover $\mathfrak{K}$ is also continuous with respect to the topology of  $\cM$.
\end{proposition}
Choosing in the definition of $\mathfrak{R}$ the same parameters $\bar{\gga}$ and $\bar{\gh}$ of $\mathfrak{K}$, the application $\cP=\mathfrak{K}+ \mathfrak{R}$ is a well defined map $\cP\colon  \cM \ltimes \cD^{\gga,\gh}\to \cD^{\bar{\gga},\bar{\gh}}$ which depends continuously on the topology of the models. We will denote by $\hat{\mathfrak{K}}_{\gep}$, $\hat{\mathfrak{R}}_{\gep}$, $\hat{\cP}_{\gep}$ (resp. $\hat{\mathfrak{K}}$, $\hat{\mathfrak{R}}$, $\hat{\cP}$) the operators $\mathfrak{K}$, $\mathfrak{R}$ and $\cP$ associated to the model $(\hat{\Pi}^{\gep}, \hgG^{\gep})$ (resp. $ (\hat{\Pi}, \hgG)$). Let us we calculate  $\hat{\cP}_{\gep}(\1_{+}\Xi)$ and $\hat{\cP}(\1_{+}\Xi)$ in this case.
\begin{proposition}\label{big_U}
For any $\gga> 0$  and every $ -3/2+\kappa< \gh< \gga$ non integer, using the shorthand notation  $\bar{\gga}= \gga+2$, the modelled distribution $U_{\gep}:= \hat{\cP}_{\gep}(\1_{+}\Xi)$ and $U:=\hat{\cP}(\1_{+}\Xi)$ belong respectively to $\cD^{\bar{\gga}, 1/2-\kappa}_{\cU}(\hgG^{\gep})$ and $\cD^{\bar{\gga}, 1/2-\kappa}_{\cU}(\hgG)$ and they are both given explicitly for any $z=(t,x)\in [0,T]\times \bR$ by the formulae
\begin{equation}\label{lifting_U1}
U_{\gep}(z)= \widetilde{u}_{\gep}(z)\1 + \1_{+}(z)\,\tikz[scale=1.2]{\node{} child{node{} edge from parent[noise] child {node{}}}}+  \sum_{0<\abs{k}<\bar{\gga}}v^k_{\gep}(z)\frac{\X^k}{k!}\,,
\end{equation}
\begin{equation}\label{lifting_U2}
U(z)= \widetilde{u}(z)\1 + \1_{+}(z)\,\tikz[scale=1.2]{\node{} child{node{} edge from parent[noise] child {node{}}}}+  \sum_{0<\abs{k}<\bar{\gga}}v^k(z)\frac{\X^k}{k!}\,,
\end{equation}
where $v^k_{\gep}(z)=(\partial_k \bar{R}*\widetilde{\1_{+}\xi_{\gep}})(z)$ and $v^k(z)=(\partial_k \bar{R}*\widetilde{\1_{[0,+\infty)}\xi})(z)$. Moreover for any $z\in [0,T]\times \bT$, $\hat{\cR}_{\gep}(U_{\gep})(z)=u_{\gep}(z)$  and $\hat{\cR}( U)(z)=u(z)$.
\end{proposition}
\begin{proof}
The proposition is a direct consequence of the definition of $\mathfrak{K}$ in the Proposition \ref{convolution_K}. In particular we have immediately $ \hat{\mathfrak{K}}_{\gep}(\1_{+}\Xi)\in \cD^{\bar{\gga}, 1/2-\kappa}_{\cU}(\hat{\gG}^{\gep})$ and $\hat{\mathfrak{K}}(\1_{+}\Xi)\in \cD^{\bar{\gga}, 1/2-\kappa}_{\cU}(\hat{\gG}) $ because $\ga_{\1_{+}\Xi}= -3/2 -\kappa$. Considering the explicit formula \eqref{defn_conv_K}, which defines $\mathfrak{K}$, by definition of $\1_{+}\Xi $ we have for any $z\in \bR^2$  
\[\hat{\cR}_{\gep}(\1_{+}\Xi)(z)= \1_{+}(z)\hPi^{\gep}_z(\Xi)(z)\,,\quad  \hat{\cR}(\1_{+}\Xi)= \1_{[0,+\infty)}\hPi_z(\Xi)\,.\]
Hence the function $\cN(\1_{+} \Xi) $ defined in \eqref{operatorN} is constantly equal to zero in case of $ \hat{\mathfrak{K}}_{\gep}$ and $ \hat{\mathfrak{K}}$. Summing up the definition of $\widetilde{I}$, the definition of $J$ and the identity \eqref{xi_reconstruct}, we obtain
\[\begin{split}
\hat{\mathfrak{K}}_{\gep}(\1_{+}\Xi)(z)&=(K* \widetilde{\1_{+}\xi_{\gep}})(z)\1+ \1_{+}(z)\,\tikz[scale=1.3]{\node{} child{node{} edge from parent[noise] child {node{}}}}\;,\\ \hat{\mathfrak{K}}(\1_{+}\Xi)(z)&=(K* \widetilde{\1_{[0,+\infty)}\xi})(z)\1+ \1_{+}(z)\,\tikz[scale=1.3]{\node{} child{node{} edge from parent[noise] child {node{}}}}\,.
\end{split}\]
Applying again the identity \eqref{xi_reconstruct} and the definition of $\mathfrak{R}$ in \eqref{operatorR}, the formulae \eqref{lifting_U1} \eqref{lifting_U2} follws from the distributional identities \eqref{sum_K+R} and \eqref{analytic_representation2}. The last identities on the reconstruction follow straightforwardly from the general identity \eqref{convolution} and the property that the kernels $K$ and $R$ are $0$ for negative times. Thus for any $z=(t,x)\in [0,T]\times \bR$ one has 
\[(K* \widetilde{\1_{[0,+\infty)}\xi})(z)=(K*\widetilde{\1_{[0,t]}\xi})(z)\,,\quad(\bar{R}* \widetilde{\1_{[0,+\infty)}\xi})(z)= (\bar{R}*\widetilde{\1_{[0,t]}\xi})(z)\]
and similarly with $\xi_{\gep}$. Thereby obtaining the thesis.
\end{proof}
\begin{remark}
For any $\gep>0$ it is also possible to consider $\cP_{\gep}$, the convolution operator associated to the canonical model $(\Pi^{\gep}, \gG^{\gep})$. Following the Remark \ref{xi_canonical} related to the modelled distribution $\1_{+}\Xi$  in the case of the canonical model and the proof of the Proposition \ref{big_U}, we obtain also that $\cP_{\gep}(\1_{+}\Xi)\in \cD^{\bar{\gga}, 1/2-\kappa}_{\cU}(\gG^{\gep})$ and the identity $\cP_{\gep}(\1_{+}\Xi)(z)= U_{\gep}(z)$ for any $z\in \bR^2$, implying $\cR_{\gep}U_{\gep}=u_{\gep}$. Following Proposition \ref{big_U}, the hypothesis $\gga>0$ implies $\bar{\gga}>2$. However, to reconstruct $u_{\gep}$ and $u$ from $U_{\gep}$ and $U$, as explained in the Remark \ref{reconstruction_more}, we can relax this condition by writing $U_{\gep}$ and $U$ as elements of $\cD^{\gga', 1/2-\kappa}$ for some $0<\gga'\leq \bar{\gga}$.
\end{remark}
Writing $u_{\gep}$ and $u$ as the reconstruction of some modelled distribution, we obtain immediately the following convergence.
\begin{proposition}\label{convergence_left}
Let $u_{\gep}$ and $u$ be the solutions respectively of the equations \eqref{eq3} and \eqref{eqSHE}. Then as $\gep\to 0^{+}$
\begin{equation}\label{convergence_left_eq}
 \sup_{(t,x)\in [0,T]\times \bT}\vert u_{\gep}(t,x) -u(t,x)\vert \overset{\bP}{\to} 0 \,.
 \end{equation}
Moreover $u_{\gep}\to u$ in probability with respect to the topology of $\cC^{1/2-\kappa}((0,T)\times \bT)$.
\end{proposition}
\begin{proof}
Thanks to the Proposition \ref{big_U}, the Proposition \ref{convolution_K} and the local Lipschitz property of the reconstruction map, there exists a continuous map $\Psi\colon \cM\to \cC^{0}([0,T]\times \bT)$ such that $u_{\gep}= \Psi((\hPi^{\gep},\hgG^{\gep}))$ and $u=\Psi((\hPi,\hgG))$. Thus the limit \eqref{convergence_left_eq} is a direct consequence of the Theorem \ref{convergence}. Restricting $u_{\gep}$ and $u$ on $(0,T)\times \bT$ and following the Remark \ref{reconstruction_more} on the regularity of the reconstruction operator outside the origin, we obtain that $\Psi$ is also a continuous map $\Psi\colon \cM\to \cC^{1/2-\kappa}((0,T)\times \bT)$, concluding in the same way.
\end{proof}
\subsubsection*{Composition}
For any  $((\Pi,\gG),V)\in \cM \ltimes \cD^{\gga,\gh}_{\cU}$, the general property of the reconstruction operator ensures us that $\cR V$ is a function (see the Remark \ref{reconstruction_more}). In particular for any function  $h\colon \bR\to \bR$ sufficiently smooth we can find a modelled distribution $H(V)$ such that 
\begin{equation}\label{composition_reconstruct}
\cR (H(V))=h\circ\cR V\,.
\end{equation}
We call this operation the \emph{lifting of $h$} and we write it  as a linear map  $H\colon \cD^{\gga,\gh}_{\cU}\to \cD^{\gga,\gh}_{\cU}$ (the lifting of a function $f$ will always be denoted in capital letters $F$). For any smooth function $h$ the function  $ H(V)\colon \bR^2\to \cU $ is given by
\begin{equation}\label{composition}
H(V)(z):=\cQ_{<\gga}\sum_{ k\geq 0}\frac{h^{(k)}(v(z))}{k!} (V(z)-\cR(V)(z)\1)^{ k}\,,
\end{equation}
where the exponent  $k$ is the product in $\cU$. Denoting by $C^n_b(\bR) $ the space of $C^n$ functions with all bounded derivatives up to the $n$-the order, we apply the general theory to deduce a sufficient condition to define the lifting $H(V)$.
\begin{proposition}\label{prop_composition}
For any $\gga>0 $, $0\leq\gh<\gga$, the lifting of $h$ in \eqref{composition} is well defined and it depends continuously on the topology of $\cM \ltimes \cD^{\gga,\gh}_{\cU}$ if $h\in C^{\gb}_b(\bR)$ where $\gb$ is the smallest integer $\gb\geq  ((\gga/\gb_V)\vee 1)+ 1$, where we recall the notation $\gb_V:= \min\{a\in \cA\setminus \bN \colon \cQ_a V\neq 0\}$.
\end{proposition}
\begin{proof}
Following the general results \cite[Thm. 6.13]{Hairer2014}, \cite[Prop. 3.11]{MartinHairer2015}, the map $H\to H(V)$ is local Lipschitz with respect to the metric $ \norm{\cdot,\cdot}_{\gga,\gh}+ \norm{\cdot,\cdot }_{\cM}$ as long as $h$ is a $\lambda$-H\"older function where $\gl\geq ((\gga/\gb_V)\vee 1)+ 1$. Thus we obtain the thesis.
\end{proof}

\begin{remark}\label{comp_U}
Applying this proposition in case of $U_{\gep}$ and $U$ we obtain   easily $\gb_{U_{\gep}}= \gb_{U}=\abs{\cI(\Xi)} = 1/2-\kappa$. Thus when we consider for any $\gga'>0$ the projection on $\cD^{\gga', 1/2-\kappa}$ of the modelled distributions $U_{\gep}$, $U$ introduced in \eqref{lifting_U1} and \eqref{lifting_U2}, the theorem applies for any $h\in C^{\gb}_b(\bR)$ where $\gb$ is the smallest integer $\gb\geq  ((2\gga'/1-2\kappa)\vee 1)+ 1$. Since this operation depends only on the algebraic structure, we have also the same result on $U_{\gep}$, interpreted as a modelled distribution with respect to the canonical model $(\Pi^{\gep}, \gG^{\gep})$.
\end{remark}
\subsubsection*{Space derivative}
Thanks to its definition, the regularity structure $\cT$ allows us to define easily a linear map $D_x\colon \cU\to \cT$, which behaves like a space derivative on abstract symbols. Indeed it is sufficient to characterise $D_x$ as the unique linear map satisfying
\begin{equation}
\begin{gathered}
 D_x \1=0\,,\quad D_x X_1=0 \,,\quad D_x X_2= \1\,\quad D_x \cI(\Xi)=\cI_1(\Xi)\,,\\ 
 D_x (\tau\gs) =(D_x  \tau) \gs + (D_x\gs)\tau\,.
\end{gathered}
\end{equation}
for any couple $\tau$, $\gs$ such that $\gs\tau\in U$. Thus by composition we can define for any couple $((\Pi,\gG),V)\in \cM \ltimes \cD^{\gga,\gh}_{\cU}$ the function $D_xV\colon \bR^2\to \cT$. This abstract operation can pass directly at the level of the reconstruction, thanks to the explicit structure of the models we are considering.
\begin{proposition}\label{space_derivative_prop}
For any model $(\Pi,\gG)$ of the form $\cL(\PI)$ for some admissible map $\PI$, the operator $D_x$ is an abstract gradient which is compatible with  $(\Pi,\gG)$, as explained in the definitions \cite[Def. 5.25, Def. 5.26]{Hairer2014}. Moreover for any $V\in \cD^{\gga,\gh}_{\cU}$ such that $ \gga>1$, $0\leq\gh<\gga $ the application $V\to D_x V$ is an application $D_{x}\colon  \cD^{\gga,\gh}_{\cU}(\gG)\to \cD^{\gga-1,\gh-1}(\gG)$ depending continuously on the topology of $\cM \ltimes \cD^{\gga,\gh}_{\cU}$ such that
\begin{equation}\label{space_derivative}
\cR (D_xV)= \partial_x (\cR V)\,,
\end{equation}
where the equality is interpreted in the sense of distributions.
\end{proposition}
\begin{proof}
By construction of the application $D_x$ and using the multiplicative property of $\gG_h$ (see the Remark \ref{multiplicative_gamma}), it is straightforward to prove recursively for any $\gb\in \cA$ and all $h\in \bR^3$ the following identities
\begin{equation}\label{space_derivative1}
D_x(\cQ_{\gb}\cU) \subset \cQ_{\gb-1}\cU\,,\quad  D_x\gG_h=\gG_hD_x\,.
\end{equation}
Hence $D_{x}$ is an abstract gradient operator, as defined in \cite[Def. 5.25]{Hairer2014}. Let us fix a model $(\Pi,\gG)$ of the form $\cL(\PI)$ for some admissible map $\PI$, then the conditions \eqref{canonical1} \eqref{canonical2} imply that $\PI\colon \cU\to \cS'(\bR^2)$ is well defined and for any  $u\in U$
\begin{equation}\label{space_derivative2}
\PI D_x u= \partial_{x}\PI u\,,
\end{equation}
where the derivative $\partial_x$ is interpreted in the sense of distributions. Summing up the properties \eqref{space_derivative1} \eqref{space_derivative2} and recalling the definition of $\Pi_z$ in \eqref{PI_and_Pi}, for any $z\in \bR^2$ and $u\in U$ we obtain
\[\Pi_zD_x u= \PI\gG_{f(z)}D_x u= \PI D_x\gG_{f(z)}u= \partial_x(\PI\gG_{f(z)}u)= \partial_x\Pi_zu\,.\]
Therefore  $D_{x}$ is an abstract gradient operator which is compatible with $(\Pi,\gG)$, as explained in  \cite[Def. 5.26]{Hairer2014}. The remaining part of the statement follows from \cite[Prop 6.15]{Hairer2014}. The continuous dependency on $\cM \ltimes \cD^{\gga,\gh}_{\cU}$ comes immediately from the definition of the metric of $\cM \ltimes \cD^{\gga,\gh}_{\cU}$.
\end{proof}
Applying the Proposition \ref{space_derivative_prop} to $U_{\gep}$ and $U$, we can write $\partial_x u_{\gep} $ and $\partial_x u$ as the reconstruction of some modelled distributions.
\begin{corollary}\label{cor_space}
For any $\gga'>1$ let  $U_{\gep}$, $U$ be the projection on $\cD^{\gga', 1/2-\kappa}$ of the modelled distributions introduced in \eqref{lifting_U1} and \eqref{lifting_U2} for any fixed realisation of $\xi$. Then the modelled distributions $D_{x}U_{\gep}$ and $D_{x}U $ belong respectively to $\cD^{\gga'-1, -1/2-\kappa}$ and for any $\gep>0$ one has
 \begin{equation}\label{space_derivative_U}
\cR_{\gep} (D_xU)= \hat{\cR}_{\gep} (D_xU)=\partial_x u_{\gep}\,,\quad  \hat{\cR} (D_xU)=\partial_x u
\end{equation}
where the second identity holds on $\cC^{-1/2-\kappa}((0,T)\times \bT)$.
\end{corollary}
\subsubsection*{Product}
We conclude the list of operation on modelled distributions with the notion of product between modelled distribution. Even if $\cT$ is not an algebra with respect to the juxtaposition product $m$ introduced in the section \ref{abstract_reg_struct}, we can still consider $m$ as a well defined bilinear map on some subspaces of $\cT$ such as $m\colon \cU\times \cT\to \cT$ or $m\colon (\cV_{\cI_1(\Xi)}\oplus \cU )\times (\cV_{\cI_1(\Xi)}\oplus \cU)\to \cT$. Therefore for any couple of modelled distribution $V_1$, $V_2$ and $\gga>0$ we define the function $V_1 V_2\colon \bR^2 \to \cT$ as
\begin{equation}\label{product_operation}
V_1V_2 (z):= \cQ_{<\gga} (V_{1}(z)V_2(z))\,,
\end{equation}
as long as the point-wise product on the right-hand side of \eqref{product_operation} is well defined. The behaviour of this operation is described in 
\cite[Proposition 6.12]{Hairer2014}, which we recall here.
\begin{proposition}\label{product_prop}
Let $(\Pi, \gG)\in \cM$ and $V_{1}\in \cD^{\gga_1,\gh_1}(\gG)$, $V_2\in  \cD^{\gga_2,\gh_2}(\gG)$ be a couple of modelled distributions such that the point-wise product is well defined. If the parameters
\begin{equation}\label{parameters_product}
\gga=( \gga_1+\ga_{V_2}) \wedge( \gga_2+\ga_{V_1})\,,\quad \gh= (\gh_1+\gh_2)\wedge (\gh_1+\ga_{V_2})\wedge (\gh_2+\ga_{V_1})\,,
\end{equation}
satisfy the conditions $\gga>0$ and $-2<\gh<\gga $, then the function $V_1V_2$ is a well defined element of $\cD^{\gga,\gh}$. This operation is continuous with respect to the topology of $\cM \ltimes \cD^{\gga,\gh}_{\cU}$.
\end{proposition}
\begin{remark}
Differently to the other operations we defined before, where we related the reconstruction operator to some classical operations on distribution, we cannot define directly the reconstruction $\cR(V_1V_2)$   as an analytical operation between $\cR(V_1)$ and $\cR(V_2)$, because there is no classical notion of product between distributions. However in case of the canonical model $(\Pi^{\gep}, \gG^{\gep})$ for any fixed $\gep>0$, we can apply the multiplicative property of $\Pi^{\gep}_z$ on symbols and the explicit form of the reconstruction operator in \eqref{simpl_reconstruct} to obtain for any couple of $V_1,V_2\in \cD^{\gga, \gh}(\gG^{\gep}) $ the general identity 
\begin{equation}\label{product_canonical}
\cR_{\gep}(V_1V_2)= \cR_{\gep}V_1\cR_{\gep}V_2\,.
\end{equation}
But this property does not hold any more with the operators $\hat{\cR}_{\gep}$ and $\hat{\cR}$.
\end{remark}
Summing up all the operations we defined before, we show the existence of two specific modelled distribution, related to $U_{\gep}$ and $U$.
\begin{proposition}\label{product_function}
Let $U_{\gep}$, $U$ be the projection on $\cD^{\gga', 1/2-\kappa}$ of the modelled distributions introduced in \eqref{lifting_U1} and \eqref{lifting_U2} for any fixed realisation of $\xi$ and $\gga'>0$. Choosing $\gga'=3/2 +2\kappa$ for any $\gp\in C^7_b(\bR)$ the modelled distributions $\gP'(U_{\gep})\Xi $, $\gP''(U_{\gep})(D_xU_{\gep})^2$ and $\gP'(U)\Xi $,  $\gP''(U)(D_xU)^2 $ are respectively well defined element of $\cD^{\kappa, -1-2\kappa}(\hgG^{\gep}) $  for any fixed $\gep>0$ and $\cD^{\kappa,  -1-2\kappa}(\hPi) $. Moreover as $\gep\to 0$ we have 
\begin{equation}\label{conv_prod}
\norm{\gP'(U_{\gep})\Xi, \gP'(U)\Xi }_{\kappa,  -1-2\kappa}\overset{\bP}{\to} 0\,,\quad \norm{\gP''(U_{\gep})(D_xU_{\gep})^2, \gP''(U)(D_xU)^2}_{\kappa, -1-2\kappa}\overset{\bP}{\to} 0\,.
\end{equation}
\end{proposition}
\begin{proof}
Using the Proposition \ref{prop_composition} and the Remark  \ref{comp_U} to $\gp'$ and $\gp''$, the modelled distributions $\gP''(U_{\gep}) $ $\gP'(U_{\gep})$ are well defined  if $ \gp',\, \gp''\in C^{\gb}_b(\bR)$ where $\gb$ is the smallest integer such that $\gb\geq ((2\gga'/1-2\kappa)\vee 1)+ 1$. Choosing $ \gga'=3/2 +2\kappa$ we have $((2\gga'/1-2\kappa)\vee 1)+ 1>4$ and $\gb= 5$. Thus by hypothesis on $\gp$ we can lift the functions $ \gp',\, \gp''$ to modelled distributions. By construction of $\cT$, the the functions $\gP''(U_{\gep})(D_xU_{\gep})^2 $, $\gP'(U_{\gep})\Xi$, $\gP''(U)(D_xU)^2 $ and $\gP'(U)\Xi $ are well defined and the result will follow by applying Proposition \ref{product_prop} to these products. In case of $\gP'(U_{\gep})\Xi$ and $ \gP'(U)\Xi$, supposing that $\Xi\in \cD^{\delta, \nu}$ for some $\delta>0$ and $\nu\in (0, \delta)$, we can choose $\kappa$ sufficiently small such that $\gP'(U_{\gep})\Xi$ and $ \gP'(U)\Xi$ belong to $\cD^{\gga, \eta}$ where
\begin{equation}\label{par1}
\gga= (\gga'+ \ga_{\Xi})\wedge \delta= \kappa \,,\quad \gh= (1/2- \kappa+\ga_{\Xi})\wedge \nu= -1-2\kappa\,,
\end{equation}
On the other hand, when we consider $\gP''(U_{\gep})(D_xU_{\gep})^2$ and $\gP''(U)(D_xU)^2$ we are doing two products. Firstly the products $(D_x U_{\gep})^2$, $(D_x U)^2 $ belong to  $\cD^{\gga_1, \eta_1}$  where
\[ \gga_1= \gga'-1 +\ga_{D_xU_{\gep}}=\kappa\,, \quad \eta_1= -1-2\kappa\]
Multiplying it with $ \gP''(U)$ and $\gP''(U_{\gep})$, we obtain that $\gP''(U_{\gep})(D_xU_{\gep})^2$ and $\gP''(U)(D_xU)^2$ belong to $\cD^{\gga, \eta}$ where
\begin{equation}
\gga= \kappa\wedge (\gga'-1-2\kappa)\,, \quad \gh= (-1-2\kappa)\wedge (-1/2- 3\kappa)\,
\end{equation}
which becomes equal to the same result of \eqref{par1}, by fixing $\kappa$ sufficiently small. Thus the modelled distribution are well defined and the convergence property \eqref{conv_prod} is a direct consequence of the Theorem \ref{convergence} and the continuity of the product operation in the topology of $\cM \ltimes \cD^{\gga,\gh}_{\cU}$.
\end{proof}
\begin{remark}
Following the proof of the Proposition \eqref{product_function}, the choice of the parameter $\gga'$ and $\gp$ in the statement could be replaced by a generic value $\gga'> 3/2+ \kappa$ and a function $\gp$ with the right number of bounded derivatives. The value $ 3/2+ 2\kappa $ was simply chosen in order to find the smallest subspace where the modelled distributions  $\gP'(U_{\gep})\Xi $, $\gP''(U_{\gep})(D_xU_{\gep})^2$, $\gP'(U)\Xi $,  $\gP''(U)(D_xU)^2 $ are well defined.
\end{remark}
\section{It\^o formula}\label{ito_form}
We combine the explicit knowledge of the sequence $(\hPi^{\gep},\hgG^{\gep})$ with the operations on the modelled distributions defined in the section \ref{calculus_reg_struct} to describe the random  distribution $(\partial_{t}- \partial_{xx})\gp(u)$ and $\gp(u)$, when $u$ is the solution of \eqref{eqSHE} and $\gp$ is a sufficiently smooth function, as explained in the introduction. The resulting formulae will be called differential and integral It\^o formula, in accordance to the formal definitions given in the equations \eqref{example_differential} and \eqref{example_integral}.
\subsection{Pathwise It\^o formulae}
The first type of identities we show are called pathwise differential  It\^o Formula and pathwise integral It\^o Formula. We choose this adjective in their denomination because these identities involve in their terms the reconstruction of some modelled distribution, an object which is defined only pathwise.
\begin{theorem}[Pathwise differential  It\^o Formula]\label{Differential_Ito}
Let $u$ be the solution of \eqref{eqSHE} and $\gp\in C^7_b(\bR)$. Then we have the identity
\begin{equation}\label{differential}
(\partial_t - \partial_{xx} )(\gp(u))=\hat{\cR}(\gP'(U)\Xi)-\hat{\cR}(\gP''(U)(D_xU)^2)\,,
\end{equation}
where the equality holds a.s. as elements of $\cC^{-3/2-\kappa}((0,T)\times \bT)$.
\end{theorem}
\begin{proof}
The identity \eqref{differential} will be obtained by rearranging the equality \eqref{parabolic_phi} in terms of modelled distributions and sending $\gep\to 0$. Recalling the Proposition \ref{big_U} and \ref{product_function} we write $u_{\gep}=\hat{\cR_{\gep}}U_{\gep}$ where $U_{\gep}$ is the projection on $\cD^{3/2+2\kappa, 1/2-\kappa}(\hgG^{\gep})$ of the modelled distributions introduced in \eqref{lifting_U1}. The hypothesis on $\gp $ and the definition of $U_{\gep}$ allow to lift $\gp'$ and $\gp''$ at the level of the modelled distributions and we can rewrite the identity \eqref{parabolic_phi} as
\begin{equation}\label{almost_differential}
(\partial_t - \partial_{xx})\gp(u_{\gep})=(\hat{\cR_{\gep}}\gP'(U_{\gep}))(\hat{\cR_{\gep}}\Xi)-(\hat{\cR_{\gep}}\gP''(U_{\gep}))(\hat{\cR_{\gep}}D_x U_{\gep})^2\,.
\end{equation}
On the other hand Proposition \ref{product_function} implies that $\gP'(U_{\gep})\Xi$ and $\gP''(U_{\gep})(D_xU_{\gep})^2$ belong to $\cD^{\kappa,-1-2\kappa} (\hgG^{\gep})$. Calculating explicitly $\hat{\cR}_{\gep}(\gP''(U_{\gep})(D_x U_{\gep})^2)$ and $\hat{\cR}_{\gep}(\gP'(U_{\gep})\Xi)$, we obtain
\begin{equation}\label{functions_nices}
\begin{split}
\hat{\cR}_{\gep}(\gP'(U_{\gep})\Xi)(z) &=\Pi_z^{\gep}\big( M_{\gep}\gP'(U_{\gep})\Xi(z)\big) (z)\,,\\ \hat{\cR}_{\gep}(\gP''(U_{\gep})(D_x U_{\gep})^2)(z)&= \Pi_z^{\gep}\big( M_{\gep}\gP''(U_{\gep})(D_x U_{\gep})^2(z)\big) (z)\,,
\end{split}
\end{equation}
for any $z\in (0,T)\times \bT$ as a consequence of the equation \eqref{simpl_reconstruct} and the Proposition \ref{interaction_model}. From these equalities we deduce an explicit relation between the functions on the left-hand side of \eqref{functions_nices} and the right-hand side of \eqref{almost_differential}. To lighten the notation we write down $\gP'(U_{\gep})\Xi$, and $\gP''(U_{\gep})(D_xU_{\gep})^2$ on the canonical basis of $ \cQ_{<\kappa}\cT$ and $ (D_xU)^2$ on the canonical bases $\cQ_{<1/2 +2\kappa}\cT$ without referring explicitly to $z$, the indicator~$\1_{+}$ and the periodic extension, obtaining the identities
\[\begin{split}
\gP'(U_{\gep})\Xi= \gp'(u_{\gep})\tikz[scale=1.5]{\node{} child[noise]{node{} }} + \gp''(u_{\gep})\tikz[scale=1.2]{\node{} child{node{} edge from parent[noise] child {node{}}}child[noise]{node{} }}+  \gp''(u_{\gep})v^{(0,1)}_{\gep}\;\tikz[scale=1.5]{\node{} child[noise]{ node{}} }_{(0,1)}  +\frac{\gp'''(u_{\gep})}{2}\tikz[scale=1.2]{\node{} child{node{} edge from parent[noise] child {node{}}}child{node{} edge from parent[noise] child {node{}}}child[noise]{node{} }} \\+ \gp'''(u_{\gep})v^{(0,1)}_{\gep}\;\tikz[scale=1.2]{\node{} child{node{} edge from parent[noise] child {node{}}} child[noise]{node{} }}_{\!\!\!\scriptstyle (0,1)} + \frac{\gp^{(4)}(u_{\gep})}{6}\tikz[scale=1.2]{\node{} child{node{} edge from parent[noise] child {node{}}} child{node{} edge from parent[noise] child {node{}}} child{node{} edge from parent[noise] child {node{}}}child[noise]{node{} }}\,,
\end{split}\]
\[\begin{split}
(D_xU_{\gep})^2= \tikz[scale=1.2]{\node{} child[derivative]{node{} edge from parent[noise] child {node{}}} child[derivative]{node{} edge from parent[noise] child {node{}}}}+2\,v^{(0,1)}_{\gep}\;\tikz[scale=1.2]{\node{} child[derivative]{node{} edge from parent[noise] child {node{}}}}+(v^{(0,1)}_{\gep})^2\1\,,
\end{split}\]
\[\begin{split}
 \gP''(U_{\gep})(D_xU_{\gep})^2= \gp''(u_{\gep})(D_xU_{\gep})^2+ \gp'''(u_{\gep})\tikz[scale=1.2]{\node{} child[derivative]{node{} edge from parent[noise] child {node{}}} child[derivative]{node{} edge from parent[noise] child {node{}}} child{node{} child[noise]{node{}}}}+ \gp'''(u_{\gep})v^{(0,1)}_{\gep}\;\tikz[scale=1.2]{\node{} child[derivative]{node{} edge from parent[noise] child {node{}}} child[derivative]{node{} edge from parent[noise] child {node{}}}}_{\!\!\!(0,1)}\\ +2\gp'''(u_{\gep})v^{(0,1)}_{\gep}\;\tikz[scale=1.2]{\node{} child{node{} child[noise]{node{}}}child[derivative]{node{} edge from parent[noise] child {node{}}}}+\frac{\gp^{(4)}(u_{\gep})}{2} \tikz[scale=1.2]{\node{} child[derivative]{node{} edge from parent[noise] child {node{}}} child[derivative]{node{} edge from parent[noise] child {node{}}} child{node{} child[noise]{node{}}} child{node{} child[noise]{node{}}}}\;.
\end{split}\]
Then we apply of the renormalisation map $M_{\gep}$
\[\begin{split}
M_{\gep}\big(\gP'(U_{\gep})\Xi\big)&= \gp'(u_{\gep})\tikz[scale=1.2]{\node{} child[noise]{node{} }} + \gp''(u_{\gep})\left(\tikz[scale=1.2]{\node{} child{node{} edge from parent[noise] child {node{}}}child[noise]{node{} }} - C_{\gep}^{1}\1\right)+  \gp''(u_{\gep})v^{(0,1)}_{\gep} \;\tikz[scale=1.5]{\node{} child[noise]{ node{}} }_{(0,1)} \\&+\frac{\gp'''(u_{\gep})}{2}\left(\tikz[scale=1.2]{\node{} child{node{} edge from parent[noise] child {node{}}}child{node{} edge from parent[noise] child {node{}}}child[noise]{node{} }}- 2 C^{1}_{\gep}\;\tikz[scale=1.2]{\node{} child{node{} edge from parent[noise] child {node{}}}}\right) +\gp'''(u_{\gep})v^{(0,1)}_{\gep}\,\left(\tikz[scale=1.2]{\node{} child{node{} edge from parent[noise] child {node{}}} child[noise]{node{} }}_{\!\!\!\scriptstyle (0,1)}- C^{1}_{\gep} \X^{(0,1)}\right)\\&+ \frac{\gp^{(4)}(u_{\gep})}{6}\left(\tikz[scale=1.2]{\node{} child{node{} edge from parent[noise] child {node{}}} child{node{} edge from parent[noise] child {node{}}} child{node{} edge from parent[noise] child {node{}}}child[noise]{node{} }}- 3 C^{1}_{\gep}\;\tikz[scale=1.2]{\node{}  child{node{} edge from parent[noise] child {node{}}}child{node{} edge from parent[noise] child {node{}}}}\right)\;,
\end{split}\]
\[\begin{split}
M_{\gep}((D_xU_{\gep})^2)= \tikz[scale=1.2]{\node{} child[derivative]{node{} edge from parent[noise] child {node{}}} child[derivative]{node{} edge from parent[noise] child {node{}}}}+2\,v^{(0,1)}_{\gep}\;\tikz[scale=1.2]{\node{} child[derivative]{node{} edge from parent[noise] child {node{}}}}+(v^{(0,1)}_{\gep})^2\1 -C^{2}_{\gep}\,\1= (D_xU_{\gep})^2 -C^{2}_{\gep}\,\1\;,
\end{split}\]
\[\begin{split}
M_{\gep}\big(\gP''(U_{\gep})(D_xU_{\gep})^2\big)&= \gp''(u_{\gep})\left((D_xU_{\gep})^2 -C^{2}_{\gep}\,\1\right)+2\gp'''(u_{\gep})v^{(0,1)}_{\gep}\;\tikz[scale=1.2]{\node{} child{node{} child[noise]{node{}}}child[derivative]{node{} edge from parent[noise] child {node{}}}} \\&+ \gp'''(u_{\gep})\left(\tikz[scale=1.2]{\node{} child[derivative]{node{} edge from parent[noise] child {node{}}} child[derivative]{node{} edge from parent[noise] child {node{}}} child{node{} child[noise]{node{}}}} - C^{2}_{\gep}\; \tikz[scale=1.2]{\node{} child{node{} edge from parent[noise] child {node{}}}}\right)+ \gp'''(u_{\gep})v^{(0,1)}_{\gep}\left(\tikz[scale=1.2]{\node{} child[derivative]{node{} edge from parent[noise] child {node{}}} child[derivative]{node{} edge from parent[noise] child {node{}}}}_{\!\!\!(0,1)}- C^2_{\gep}\X^{(0,1)}\right)\\ &+ \frac{\gp^{(4)}(u_{\gep})}{2} \left(\tikz[scale=1.2]{\node{} child[derivative]{node{} edge from parent[noise] child {node{}}} child[derivative]{node{} edge from parent[noise] child {node{}}} child{node{} child[noise]{node{}}} child{node{} child[noise]{node{}}}} - C^2_{\gep}\,\tikz[scale=1.2]{\node{} child{node{} child[noise]{node{}}} child{node{} child[noise]{node{}}}}\right)\,.
\end{split}\]

To conclude the calculation we apply the operator $\Pi^{\gep}_z \cdot (z)$ on both sides of the above equations. As a consequence of the definition of $\Pi_z^{\gep} $, one has $\Pi^{\gep}_z \tau (z)=0$ for every $\tau\in T$ of the form $\gs_1 \gs_2$ with $\abs{\gs_1}>0$. Hence most of the terms in the expansion above are discarded and we last with the identities
\begin{equation}\label{first_rough}
\hat{\cR}_{\gep}(\gP'(U_{\gep})\Xi)= \gp'(u_{\gep})\xi_{\gep}-\gp''(u_{\gep}) C_{\gep}^{1}=	(\hat{\cR_{\gep}}\gP'(U_{\gep}))(\hat{\cR_{\gep}}\Xi)-\gp''(u_{\gep}) C_{\gep}^{1}\,,
\end{equation}
\begin{equation}\label{second_rough}
\begin{split}
\hat{\cR_{\gep}}(\gP''(U_{\gep})(D_xU_{\gep})^2) &=\gp''(u_{\gep})\left(\Pi_z^{\gep}(D_xU_{\gep})^2 (z) - C_{\gep}^{2}\right)\\&=(\hat{\cR_{\gep}}\gP''(U_{\gep}))\left(\Pi_z^{\gep}(D_xU_{\gep})^2 (z) - C_{\gep}^{2}\right).
\end{split}
\end{equation}
Writing $\Pi_z^{\gep}(D_xU_{\gep})^2(z)= \cR_{\gep}((D_xU_{\gep})^2)  $, the multiplicative property of $\cR_{\gep}$ in \eqref{product_canonical} and the identity \eqref{space_derivative_U} imply that the equality \eqref{second_rough} becomes
\begin{equation}\label{third_rough}
\hat{\cR_{\gep}}(\gP''(U_{\gep})(D_xU_{\gep})^2) =  (\hat{\cR_{\gep}}\gP''(U_{\gep}))(\hat{\cR_{\gep}}D_x U_{\gep})^2 -\gp''(u_{\gep}) C_{\gep}^{2}\,.
\end{equation}
Resuming up the equations \eqref{first_rough} and \eqref{third_rough}, we obtain the final rearrangement
\begin{equation}\label{approx_right}
(\partial_t - \partial_{xx})\gp(u_{\gep})=\hat{\cR_{\gep}}(\gP'(U_{\gep})\Xi)-\hat{\cR_{\gep}}(\gP''(U_{\gep})(D_xU_{\gep})^2)+\gp''(u_{\gep})\left(C_{\gep}^{1}-C_{\gep}^{2}\right)\,.
\end{equation}
Let us now send $\gep\to 0^+$. the left-hand side of \eqref{approx_right} converges in probability to $(\partial_t - \partial_{xx} )\gp(u)$ thanks to the Proposition \ref{convergence_left} and the fact that the derivative is a continuous operation between H\"older spaces. On the other hand, the local Lipschitz property of the reconstruction operator $\cR$ in \eqref{local_lipschitz} and the convergence \eqref{conv_prod} imply 
\begin{equation}\label{convergence_recon}
\hat{\cR_{\gep}}(\gP'(U_{\gep})\Xi)\overset{\bP}{\to}\hat{\cR}(\gP'(U)\Xi)\,,\quad\hat{\cR_{\gep}}(\gP''(U_{\gep})(D_xU_{\gep})^2)\overset{\bP}{\to}\hat{\cR}(\gP''(U)(D_xU)^2)\,,
\end{equation}
with respect to the topology of  $\cC^{-3/2- \kappa}((0,T)\times \bT)$. Thus the theorem holds as long as the deterministic sequence $C^1_{\gep}- C^{2}_{\gep}$ converges to $0$, which is the main consequence of Lemma  \ref{constants}.
\end{proof}
\begin{remark}\label{div_proba}
Looking at the identities \eqref{first_rough} and \eqref{third_rough} separately and the convergence result \eqref{convergence_recon}, we obtain the existence of two sequences of random variables $X^{1}_{\gep}\,, \,X^2_{\gep}\in \cC^{-3/2- \kappa}((0,T)\times \bT)$ converging in probability such that 
\[\gp'(u_{\gep})\xi_{\gep}= X^{1}_{\gep} +\gp''(u_{\gep})C_{\gep}^{1}\,,\quad  \gp''(u_{\gep})(\partial_xu_{\gep})^2=X^2_{\gep} + \gp''(u_{\gep})C_{\gep}^{2}\,.\]
Since we know from the Lemma \ref{constants} that the deterministic sequences $C^1_{\gep}$ and $ C^{2}_{\gep}$ are both diverging, we obtain easily
\[\norm{\gp'(u_{\gep})\xi_{\gep}}_{\cC^{-3/2- \kappa}((0,T)\times \bT)}\overset{\bP}{\to}  +\infty\,,\quad\norm{\gp''(u_{\gep})(\partial_x u_{\gep})^2}_{\cC^{-3/2- \kappa}((0,T)\times \bT)}\overset{\bP}{\to}  +\infty\,.\]
Thus we can justify rigorously the calculations done in the introduction. 
\end{remark}
From the formula \eqref{differential} we can identify $\gp(u)$ with the solution of the following equation
\[
\begin{cases}
\partial_t v- \partial_{xx}v = \hat{\cR}(\gP'(U)\Xi)-\hat{\cR}(\gP''(U)(D_xU)^2)\,,\\
v(0, x)= \gp(0)\,\\
v(t,0)=v(t,1)\,\\
\partial_xv(t,0)=\partial_x v(t,1)
\end{cases}
\]
Using the general results contained in section \ref{elements} we obtain immediately.
\begin{corollary}[Pathwise integral It\^o Formula]
For any $\gp\in C^7_b(\bR)$ and $(t,x)\in [0,T]\times \bT$ we have 
\[
\begin{split}
\gp(u(t,x))=& \gp(0)+ (P*\1_{[0,t]}\hat{\cR}(\gP'(U)\Xi)) (t,x)- (P*\1_{[0,t]}\hat{\cR}((\gP''(U)(D_xU)^2)))(t,x)\,.
\end{split}
\]
\end{corollary}

\subsection{Identification of the differential formula}
Thanks to the explicit Gaussian structure involving the definition of $u$ in \eqref{sto_conv}, in order to obtain the Theorem \ref{Integral_Ito}, we can identify the terms $\hat{\cR}(\gP'(U)\Xi))$ and  $\hat{\cR}(\gP''(U)(D_xU)^2)$ appearing in the formula \eqref{differential} with some explicit classical operations of stochastic calculus (the so called identification theorems of the introduction). In case of $\hat{\cR}(\gP'(U)\Xi))$, this identification is done by means of a general result contained in \cite{MartinHairer2015}. In what follows, we will denote by $(\cF_{t})_{t\in \bR}$ the natural filtration of $\xi$, that is $\cF_t:=\gs(\{\xi(\psi)\colon \psi\vert_{(t,+\infty)\times\bT}=0\,; \psi\in L^{2}(\bR\times \bT) \})$.
\begin{proposition}\label{ident1}
Let $(\hat{\Pi},\hat{\gG})$ be the BPHZ model and $\gp\in C^7_b(\bR)$. Then for any smooth function $\psi\colon \bR\times \bT \to \bR$ with supp $(\psi) \subset (0,+\infty) \times \bT$, one has for any $t\in (0,T]$
\begin{equation}\label{ident1_equation}
\bigl(\1_{[0,t]}\hat{\cR}(\gP'(U)\Xi)\bigr)(\psi) = \int_0^t  \int_{\bT} \gp'(u(s,y))\psi(s,y) dW_{s,y}\;.
\end{equation}
\end{proposition}
\begin{proof}
Thanks to the inclusion of $\cV_{\Xi}$ and $\cU$ into the regularity structure $\cT_{HP}$ and the identification of the BPHZ model $(\hat{\Pi},\hat{\gG})$ with the \emph{It\^o model}, both defined in \cite{MartinHairer2015} (see the Remark \ref{positive_renom} and \ref{Ito_model}), the identity \eqref{ident1_equation} is a consequence of \cite[Theorem 6.2]{MartinHairer2015} applied to the modelled distribution $\gP'(U)\in \cD^{\gga',\gh'}_{\cU}$, with $\gga'= 3/2+ 2\kappa$, $\gh'= 1/2- \kappa$ given in Proposition \ref{product_function}. Let us check that $\gP'(U)$ satisfy the hypotheses of this theorem. For any $z\in [0,T]\times \bR$, $z= (t,x)$ we apply the definition \eqref{composition} to the explicit form of $U$ in \eqref{lifting_U2} to obtain
\[\begin{split}
\gP'(U)(z)&= \gp'(\widetilde{u}(z))\1 + \gp''(\widetilde{u}(z))\1_{+}(z)\,\tikz[scale=1.2]{\node{} child{node{} edge from parent[noise] child {node{}}}}+  \gp''(\widetilde{u}(z))v^{(0,1)}(z)\X^{(0,1)}  \\&+ \frac{\gp'''(\widetilde{u}(z))}{2}v^{(0,1)}(z)\1_{+}(z)\,\tikz[scale=1.2]{\node{} child{node{} edge from parent[noise] child{node{}}}}_{\,(0,1)}+\frac{\gp'''(\widetilde{u}(z))}{2}\1_{+}(z)\,\tikz[scale=1.2]{\node{} child{node{} edge from parent[noise] child {node{}}}child{node{} edge from parent[noise] child {node{}}}} \\&+ \frac{\gp^{(4)}(\widetilde{u}(z))}{6}\1_{+}(z)\,\tikz[scale=1.1]{\node{} child{node{} edge from parent[noise] child {node{}}} child{node{} edge from parent[noise] child {node{}}} child{node{} edge from parent[noise] child {node{}}}}\,.
\end{split}\]
Since $\{u(t,x)\}_{(t,x)}$ and $\{v^{(0,1)}(t,x)\}_{(t,x)}$ are adapted to the filtration $\cF_t$, so is the random field $\{\gP'(U)(t,x)\}_{(t,x)}$. Moreover $ \bE (\vert \gP'(U)\vert_{\gga',\gh'})^p<+\infty$  for any $p>2$ because $\gP'$ is a local Lipschitz map on the Banach space $(\cD^{\gga,\gh}(\hgG),\vert \cdot \vert_{\gamma,\eta} )$ and one has the bounds
\begin{equation}\label{integrability_u}
\bE\left(\sup_{z\in [0,T]\times \bT} \abs{u(z)}\right)^p<+ \infty\,,\quad \bE\left(\sup_{z\in [0,T]\times \bT} \abs{v^{(0,1)}(z)}\right)^p<+  \infty\,.
\end{equation}
These estimates can be easily obtained by applying the general Borell-TIS inequality to the random fields $u$ and $v^{(0,1)}$, which have bounded variance on $[0,T]\times \bT$.
\end{proof}
We pass to the identification of $\hat{\cR}(\gP''(U)(D_xU)^2)$. In this case no general result can be applied and following the same procedure of \cite{Zambotti2006}, we can identify this random distribution through a different approximation of the process $u$ using the heat semigroup on $u$. For any $\gep>0$ we define the process
\begin{equation}\label{approx_martingale}
u^{\gep}_t(x):=\int_{\bT}P_{\gep}( x-y)u(t,y) dy\,.
\end{equation}
\begin{lemma}\label{properties_martingale}
For any $\gep>0$ the process $u^{\gep}$ satisfies the following properties:
\begin{itemize}
\item for any $t>0$ the process $\{u^{\gep}_t(x)\}_{x\in\bT}$ has a.s. smooth trajectories, satisfying for any integer $m\geq 0$ the a.s. identity
 \begin{equation}\label{derivation_under}
\partial_x^m u^{\gep}_t(x)=\int_{0}^t\int_{\bT}\partial_x^mP_{\gep+t-s}( x-y)  dW_{s,y} \,.
\end{equation}
\item for any $x\in \bT$ the process $\{u^{\gep}_t(x)\}_{t\in[0, T]}$ is the solution of the stochastic differential equation
\begin{equation}\label{evolution_time}
 \begin{cases}
d u^{\gep}_t(x)= \partial_{xx}( u^{\gep}_t)(x) dt + dW^{\gep}_t(x)\\ u_0^{\gep}(x)=0\,,
 \end{cases}
\end{equation}
where $W^{\gep}_t(x)$ is the $(\cF_t)$-martingale
\[W^{\gep}_t(x)= \int_0^t\int_{\bT}P_{\gep}( x-y) dW_{s,y}\,,\quad  \langle W^{\gep}_{\cdot}(x)\rangle_t=t \int_{\bT}P(\gep, x-y)^2 dy=t \norm{ P_{\gep}(\cdot)}^2_{L^2(\bT)}\,. \]
\item By sending $\gep\to 0$ one has
\begin{equation}\label{convergence_u^eps}
\sup_{(t,x)\in [0,T]\times \bT}\abs{u^{\gep}(t,x)- u(t,x)}\to 0 \quad \text{a.s.}
\end{equation}
\end{itemize}
\end{lemma}
\begin{proof}
We start by considering the trajectories of $ x\to u_t^{\gep}(x)$ for any fixed $t>0$. Since $u$ is a.s. a continuous function, the regularisation property of the heat kernel $P$ implies the desired property on its trajectories. Moreover, for any integer $m\geq 0$ we can pass the derivative under the Lebesgue integral to obtain
\[
\partial_x^m u^{\gep}_t(x)= \int_{\bT}\partial_x^mP_{\gep}( x-y)u(t,y) dy \quad \text{a.s.}
\]
Using the straightforward bound
\[ \int_{\bT}\int_{0}^t\int_{\bT}\left(\partial_x^mP_{\gep}(x-y)P_{t-s}( y-v)\right)^2 ds\, dv\,dy< +\infty\,,\]
we can obtain the formula \eqref{derivation_under} by writing the stochastic integral in \eqref{derivation_under} as a Wiener integral and applying the stochastic Fubini theorem for Wiener integral, as explained in \cite[Thm. 5.13.1]{peccati2011}. For any fixed $x\in\bT$ we study the process $ t\to u_t^{\gep}(x)$. By definition of mild solution for the equation \eqref{eqSHE}, $u$ satisfies the equality  \eqref{weak_form} for any smooth function $l\colon \bT\to \bR$, thus the identity \eqref{evolution_time} follows by simply setting $l(y)=P_{\gep}(x-y)$ in \eqref{weak_form}. Finally for any $(t,x)\in[0,T]\times \bT$ the a.s. H\"older continuity of  $u$ in the space and time implies the convergence \eqref{convergence_u^eps}, using the classical property of the heat semigroup on continuous functions.
\end{proof}
\begin{theorem}\label{ident2}
Let $(\hat{\Pi},\hat{\gG})$ be the BPHZ model and $\gp\in C^7_b(\bR)$. Then for any smooth function $\psi\colon \bR\times \bT \to \bR$ with supp $(\psi) \subset (0,+\infty) \times \bT$, one has for any $t\in (0,T]$ the a.s. identity
\begin{equation}\label{formula_ident}
\1_{[0,t]}\hat{\cR} \bigl(\gP''(U)(D_xU)^2\bigr)(\psi) = - \frac{1}{2}\int_{0}^{t}\int_{\bT}\psi(s,y)\gp''(u(s,y)) C(s)dy\, ds \end{equation}
\[+\int_{[0,t]^2\times \bT^2}\left[\int_{\mathbf{s}_2\vee \mathbf{s}_1}^t \int_{\bT}\psi(s,y)\gp''(u(s,y))\partial_xP_{s-\mathbf{s}_1}( y- \mathbf{y}_1)\partial_xP_{s-\mathbf{s}_2}( y- \mathbf{y}_2) dyds \right]dW^{2}_{\mathbf{s},\mathbf{y}}\,,\]
where $C\colon(0,T)\to \bR$  is the deterministic integrable function $C(s):=\norm{ P_s(\cdot)}^2_{L^2(\bT)}$.
\end{theorem}
\begin{proof}
We prove firstly the result when $\psi= h\otimes l$ where $h\colon [0,t]\to \bR$ is a compactly supported smooth function and $l\colon \bT\to \bR $. $ \psi$ is compactly supported up to time $t$. Therefore we we can forget the operator $\1[0,t]$ on the left
hand side of \eqref{formula_ident} and we can apply the Theorem \ref{Differential_Ito} and the Proposition \ref{ident1} obtaining
\[\hat{\cR} \bigl(\gP''(U)(D_xU)^2\bigr)(\psi)= \left(-\partial_t (\gp(u))+ \partial^2_x (\gp(u))+\hat{\cR}(\gP'(U)\Xi)\right)(\psi)=\]
\begin{equation}\label{reverse_equation}
\begin{split}
&\int_{0}^t\int_{\bT}\bigg( \gp(u(s,y)) h'(s)l(y) + \gp(u(s,y)) h(s)l''(y) \bigg)ds \,dy\\&+ \int_0^t\int_{\bT}  \gp'(u(s,y))h(s)l(y) dW_{s,y}\,.
\end{split}
\end{equation}
Let us recover the right-hand side of \eqref{reverse_equation} via a different approximation. Using the process $u^{\gep}$ defined in \eqref{approx_martingale}, we can apply the It\^o formula to the semimartingale $h(s)\gp(u^{\gep}_s(y))$ for some fixed $y$ and we obtain
\begin{equation}\label{classical_ito}
\begin{split}
h(t)\gp(u^{\gep}_t(y))- &h(0)\gp(u^{\gep}_0(y))= \int_{0}^t h'(s)\gp(u^{\gep}_s(y))\,ds+ \int_0^t h(s)\partial_{xx}( u^{\gep}_s)(y)\gp'(u^{\gep}_s(y)) ds \\&+\int_0^t h(s) \gp'(u^{\gep}_s(y))dW^{\gep}_s(y)+ \frac{1}{2}\norm{ P_{\gep}(\cdot)}^2_{L^2(\bT)}\int_0^t h(s) \gp''(u^{\gep}_s(y))ds\,.
\end{split}
\end{equation}
The left-hand side of \eqref{classical_ito} is a.s. equal to zero by hypothesis on $h$ and we can still apply the formula \eqref{chain2} with $u^{\gep}$ instead of $u_{\gep}$. Hence we can rewrite the equation \eqref{classical_ito} as
\begin{equation}\label{classical_ito2}
\begin{split}
\int_{0}^t\bigg( \gp(u^{\gep}_s(y))h'(s)+  \partial_{xx}( \gp(u^{\gep}_s(y)))h(s)\bigg) ds +\int_0^t h(s) \gp'(u^{\gep}_s(y))dW^{\gep}_s(y)&\\ = \int_0^t\left[( \partial_xu^{\gep}_s)^2(y)-\frac{C_{\gep}(y)}{2}\right] h(s) \gp''(u^{\gep}_s(y))ds \,.
\end{split}
\end{equation}
By multiplying both sides of \eqref{classical_ito2} with $l$ and integrating by part over $\bT$ to transfer the second derivative on $l$, the equation \eqref{classical_ito2} becomes
\begin{equation}\label{classical_ito3}
\begin{split}
&\int_{\bT} \bigg(\int_{0}^t\bigg( \gp(u^{\gep}_s(y))h'(s)l(y)+  \gp(u^{\gep}_s(y)) h(s)l''(y)\bigg) ds  \\&+\int_0^t  \gp'(u^{\gep}_s(y))h(s)l(y)dW^{\gep}_s(y)\bigg)dy  \\&= \int_0^t \int_{\bT} l(y) h(s) \gp''(u^{\gep}_s(y)) \left(( \partial_xu^{\gep}_s)^2(y)-\frac{\norm{ P_{\gep}(\cdot)}^2_{L^2(\bT)}}{2}\right)ds\,dy\,.
\end{split}
\end{equation}
Writing the integral with respect to $dW^{\gep}_s(x)$ as a Walsh integral, we can apply the boundedness of $\gp'$ and $\gp$ to apply a stochastic Fubini's theorem on $dW_{s,y}$ (see \cite[Thm. 65]{Cairoli1975})
\begin{equation}\label{stoch_int}
\begin{split}
\int_{\bT}\int_0^t  &\gp'(u^{\gep}_s(y))h(s)l(y)dW^{\gep}_s(y)dy=\int_{\bT}\left(\int_0^t \int_{\bT} P_{\gep}( y-z)\gp'(u^{\gep}_s(z))h(s)l(y)dW_{s,z}\right)\!dy \\&=\int_0^t\int_{\bT} \left(\int_{\bT}P_{\gep}( z-y)\gp'(u^{\gep}_s(y))h(s)l(y)dy\right) dW_{s,z}\,.
\end{split}
\end{equation}
Let us prove that the left-hand side of \eqref{classical_ito3} converges in $L^2(\gO)$ to the right-hand side of \eqref{reverse_equation}. From the uniform convergence \eqref{convergence_u^eps} of $u^{\gep}$, it is straightforward to show as $\gep\to 0$
\[\int_{0}^t\int_{\bT}\gp(u^{\gep}_s(y))h'(s)l(y)dy ds\to \int_{0}^t\int_{\bT}\gp(u_s(y))h'(s)l(y)dy ds\quad \text{a.s.}\]
\[\int_{0}^t\int_{\bT} \gp(u^{\gep}_s(y)) h(s)l''(y) dy ds\to \int_{0}^t\int_{\bT} \gp(u_s(y)) h(s)l''(y) dy ds \quad \text{a.s.}\]
and the convergence holds also in $L^2(\gO)$ because these random variables are also uniformly bounded. In  case of the stochastic integral in \eqref{stoch_int}, the same uniform convergence of $u^{\gep}$ in \eqref{stoch_int} implies that
\[\sup_{(s,z)\in [0,T]\times \bT}\left|\int_{\bT}P_{\gep}( z-y)\gp'(u^{\gep}_s(y))h(s)l(y)dy- \gp'(u_s(z))h(s)l(z)\right|\to 0\quad \text{a.s.}\]
and bounding these quantity by some constant we obtain by dominated convergence 
\[\bE\left[\int_{0}^t\int_{\bT}\left(\int_{\bT}P_{\gep}( z-y)\gp'(u^{\gep}_s(y))h(s)l(y)dy -\gp'(u_s(z))h(s)l(z)\right)^2 ds dz\right]\to 0\,.\]
Hence the proof is complete as long as the right-hand side of \eqref{classical_ito3} converges in $L^2(\gO)$ to the right-hand side of \eqref{formula_ident}. Using the shorthand notations $p^{\gep}_s(y)=\partial_xP_{\gep+s}( y)$, $P^{\gep}_{s}(y):= P_{\gep+s}(y)$, the formula \eqref{derivation_under} on $u^{\gep}$ when $m=1$ becomes
\[\partial_x u^{\gep}_s(y)=\int_{0}^s\int_{\bT} p^{\gep}_{s-r}(y-z)  dW_{r,z\,.}\]
Thus writing it as a Wiener integral, we can express $(\partial_x u^{\gep}_s(y))^2$ using the Wiener chaos decomposition of a product (see \cite[Prop. 1.1.2]{nualart1995malliavin}) obtaining
\[ \begin{split}
(\partial_x u^{\gep}_s(y))^2&= I_2\bigg(\1_{[0,s]^2\times \bT^2}p^{\gep}_{s-\cdot}(y-\cdot)p^{\gep}_{s-\cdot}(y-\cdot)  \bigg)+ \int_0^s\int_{\bT}[p^{\gep}_{s-r}(y-z) ]^2dz\, dr\,.
\end{split}\]
Hence, recalling the invariance by  translations of  $\zeta\to \int_{\bT}P(s, \zeta-y)^2 dy $ we write the right-hand side of $(\ref{classical_ito3})$ as  $A_1^{\gep}+ A_{2}^{\gep}$ where 
\[\begin{split}
A_1^{\gep}&=\int_0^t \int_{\bT} l(y) h(s) \gp''(u^{\gep}_s(y))\bigg(\int_0^s\int_{\bT} [p^{\gep}_{s-r}(y-z) ]^2dr dz-\frac{1}{2}\int_{\bT}P^{\gep}_{s}(y-x)^2 dy  \bigg)dyds \\A_2^{\gep}&=\int_0^t \int_{\bT} l(y) h(s) \gp''(u^{\gep}_s(y))I_2\big(\1_{[0,s]^2\times \bT^2}p^{\gep}_{s-\cdot}(y-\cdot)p^{\gep}_{s-\cdot}(y-\cdot) \big)dy\,ds\,.
\end{split}\]
We treat both terms separately. In case of  $A_1^{\gep}$ we apply the integration by parts on the $z$ variable and the smoothness of $P$ outside the origin imply for any $y\in \bT$
\[\int_0^s\int_{\bT} [p^{\gep}_{s-r}(y-z) ]^2 dr dz=-\int_{\gep}^{s+\gep}\int_{\bT} P_{s-r}(y-z) \partial_{t}P_{s-r}(y-z) dr dz=\]
\[=-\int_{\gep}^{s+\gep}\partial_t\left( \int_{\bT}\frac{P_{r}(y-z)^2}{2} dz \right)dr=  \int_{\bT}\frac{P_{\gep}(y-z)^2}{2} dz-\int_{\bT}\frac{P_{s+\gep}(y-z)^2}{2} dz\,.\]
Using again the invariance by  translations of  $\zeta\to \int_{\bT}P(s, \zeta-y)^2 dy $, we can rewrite
\[\begin{split}
A^{\gep}_1&=- \frac{1}{2}\int_0^t \int_{\bT} l(y) h(s) \gp''(u^{\gep}_s(y)) \left(\int_{\bT}P_{s+\gep}(y-z)^2 dz \right)ds\,dx\\&= - \frac{1}{2}\int_0^t \int_{\bT} l(y) h(s) \gp''(u^{\gep}_s(y)) \norm{P_{s+\gep}(\cdot) }^2_{L^2(\bT)}\,ds\,dx\,,
\end{split}\]
where the function $C$ is given is defined in the statement. Therefore from the convergence \eqref{convergence_u^eps} we obtain
\[A^{\gep}_1\to - \frac{1}{2}\int_{0}^{t}\int_{\bT}\psi(s,y)\gp''(u(s,y)) C(s)dy\, ds\quad \text{a.s.}\]
and the convergence holds also in $L^2(\gO)$ because the sequence $A^{\gep}_1$ is uniformly bounded. We pass to the treatment of $A_2^{\gep}$. In order to identify its limit, we interpret the double Wiener integral $I_2$ as a multiple Skorohod integral of order $2$. Then we want to rewrite the quantity
\[ l(y) h(s) \gp''(u^{\gep}_s(y))I_2\bigg(\1_{[0,s]^2\times \bT^2}p^{\gep}_{s-\cdot}(y-\cdot)p^{\gep}_{s-\cdot}(y-\cdot) \bigg)\]
using the product formula \eqref{product_formula} for $\gd^2 $ to commute the deterministic integral in $ds\, dy$ with the stochastic integration. Defining $ U^{\gep}_s(y):=l(y) h(s) \gp''(u^{\gep}_s(y))$, one has $ U^{\gep}_s(y)\in \mathbb{D}^{2,2} $ because $u^{\gep}\in \mathbb{D}^{2,2}$  belongs to some fixed Wiener chaos and $\gp''$ has both two derivatives bounded. Applying  the chain rule formula for the Malliavin derivative (see \cite[Proposition 1.2.3]{nualart1995malliavin}) we have
\begin{equation}\label{malliavin_derivatives}
\begin{split}
\nabla_{\mathbf{s}_1,\mathbf{y}_1} U^{\gep}_{s}(y)&= l(y) h(s) \gp'''(u^{\gep}_s(y)) P^{\gep}_{s-\mathbf{s}_1} (y-\mathbf{y}_1)\,,\\\nabla_{\mathbf{s}_1,\mathbf{y}_1,\mathbf{s}_2 ,\mathbf{y}_2 }^2 U^{\gep}_{s}&=l(y) h(s) \gp^{(4)}(u^{\gep}_s(y)) P^{\gep}_{s-\mathbf{s}_1} (y-\mathbf{y}_1)P^{\gep}_{s-\mathbf{s}_2} (y-\mathbf{y}_2)\,.
\end{split}
\end{equation}
For any $\gep>0 $ it is straightforward to check that the hypothesis of the product formula \eqref{product_formula} are satisfied, therefore we can write
\begin{equation}\label{product_applied}
\begin{split}
&U^{\gep}_{s}(y)I_2\big(\1_{[0,s]^2\times \bT^2}p^{\gep}_{s-\cdot}(y-\cdot)p^{\gep}_{s-\cdot}(y-\cdot) \big)=\\&= \int_{[0,s]^2}\int_{\bT^2}U^{\gep}_{s}(y)p^{\gep}_{s-\mathbf{s}_1}(y-\mathbf{y}_1)p^{\gep}_{s-\mathbf{s}_2}(y-\mathbf{y}_2)dW^2_{\mathbf{s}_1\mathbf{y}_1\mathbf{s}_2\mathbf{y}_2}\\&+ 2\int\limits_{[0,s]\times \bT} \bigg(\int_{[0,s]\times\bT}\nabla_{\mathbf{s}_1,\mathbf{y}_1}U^{\gep}_{s}(y)p^{\gep}_{s-\mathbf{s}_1}(y-\mathbf{y}_1)p^{\gep}_{s-\mathbf{s}_2}(y-\mathbf{y}_2)d\mathbf{s}_1d\mathbf{y}_1\bigg)dW_{\mathbf{s}_2\mathbf{y}_2}\\& +\int\limits_{[0,s]^2\times \bT^2} \nabla_{\mathbf{s}_1,\mathbf{y}_1,\mathbf{s}_2 ,\mathbf{y}_2 }^2 U^{\gep}_{s}(y)p^{\gep}_{s-\mathbf{s}_1}(y-\mathbf{y}_1)p^{\gep}_{s-\mathbf{s}_2}(y-\mathbf{y}_2)d\mathbf{s}_1d\mathbf{y}_1 d\mathbf{s}_2d\mathbf{y}_2\,.
\end{split}
\end{equation}
Looking at the deterministic deterministic integrals in the right-hand side of \eqref{product_applied}, they are both zero as a consequence of the trivial identity
\begin{equation}\label{integral_zero}
\int_{\bT}P^{\gep}_{s-r}(y-z)p^{\gep}_{s-r}(y-z)dz= \int_{\bT}\partial_x\frac{(P^{\gep}_{s-r}(y-z))^2}{2} dz=0\,.
\end{equation}
Thus we can interchange the product of $U^{\gep}_s(y)$ with the multiple Shorokod integral of order $2$. For any $\gep>0$ the stochastic integrand inside $dW^2_{\mathbf{s}_1\mathbf{y}_1\mathbf{s}_2\mathbf{y}_2} $ is a smooth fnction in all its variables $\mathbf{s}_1\,,\mathbf{y}_1\,,\mathbf{s}_2\,,\mathbf{y}_2\,, s\,,y$, then it is square integrable when we integrate it on its referring domain. Therefore we can apply a Fubini type theorem for Skorohod integrals (see e.g. \cite{Nualart1998}) to finally obtain
\[A^{\gep}_{2}= \int_{[0,t]^2\times \bT^2}\bigg(\int_{\mathbf{s}_1\vee \mathbf{s}_2}^t \int_{\bT}h(s)l(y)\gp''(u^{\gep}_s(y))p^{\gep}_{s-\mathbf{s}_1}(y-\mathbf{y}_1)p^{\gep}_{s-\mathbf{s}_2}(y-\mathbf{y}_2) dy\, ds\bigg) dW^2_{\mathbf{s},\mathbf{y}}\,.\]
Let us explain the convergence of $A^{\gep}_{2}$ to the multiple Skorohod integral of order two in the final formula \eqref{formula_ident}. On one hand we proved that all the previous terms in the formula converge in $L^2(\gO)$. Then if the sequence of functions
\[
F^{\gep}(\mathbf{s},\mathbf{y}):=\int_{\mathbf{s}_1\vee \mathbf{s}_2+\gep}^{t+\gep}\int_{\bT} h(s-\gep)l(y)\gp''(u^{\gep}_{s-\gep}(y))p^{0}_{s-\mathbf{s}_1}(y-\mathbf{y}_1)p^{0}_{s-\mathbf{s}_2}(y-\mathbf{y}_2) dy\, ds\,,
\]
where $\mathbf{s}=(\mathbf{s}_1,\mathbf{s}_2)$ and $\mathbf{y}=(\mathbf{y}_1,\mathbf{y}_2)$, converges in $L^2(\Omega\times[0,t]^2\times \bT^2)$ to the function
\[
F(\mathbf{s},\mathbf{y}):=\int_{\bT}\int_{\mathbf{s}_1\vee \mathbf{s}_2}^{t} h(s)l(y)\gp''(u_{s}(x))p^{0}_{s-\mathbf{s}_1}(y-\mathbf{y}_1)p^{0}_{s-\mathbf{s}_2}(y-\mathbf{y}_2) dy\, ds\,,
\]
the theorem will follow because the multiple Skorohod integral is a closed operator. From the a.s. convergence of $u^{\gep}$ in \eqref{convergence_u^eps} it is straightforward to prove that $F^{\gep}$ converges to $F$ a.s. and a.e. Then we conclude by dominated convergence by proving that $\norm{F^{\gep}}^{2}_{L^2}$, the square norm of $F^{\gep}$ in $L^{2}([0,t]^2\times \bT^2)$ is uniformly bounded in $\gep$. Using the symmetry of  $F^{\gep}$ in the variables $\mathbf{s}_1$ and $\mathbf{s}_2$ we introduce the set $\gD_{2,t}= \{0<\mathbf{s}_1<\mathbf{s}_2<t\}$ and writing the square of an integral as a double integral one has
\begin{equation}\label{simplification}
\begin{split}
&\norm{F^{\gep}}_{L^2}^{2}=2\int_{\gD_{2,t}}\int_{ \bT^2}F^{\gep}(\mathbf{s},\mathbf{y})^2d\mathbf{s}d\mathbf{y}=\\&=2 \int_{\gep}^{t+\gep}ds\,\int_{\gep}^{t+\gep}dr\,\int_{\bT}\, dy\int_{\bT} \,dx\; h(s-\gep)h(r-\gep)l(y) l(x)\gp''(u^{\gep}_{s-\gep}(x))\gp''(u^{\gep}_{r-\gep}(y))\\&\times \int_{0}^{s\wedge r}\int_{\bT} \int_{0}^{\mathbf{s}_2}\int_{\bT} p^{0}_{s-\mathbf{s}_1}(x-\mathbf{y}_1) p^{0}_{s-\mathbf{s}_2}(x-\mathbf{y}_2)p^{0}_{r-\mathbf{s}_1}(y-\mathbf{y}_1)p^{0}_{r-\mathbf{s}_2}(y-\mathbf{y}_2)  d\mathbf{s} \,d\mathbf{y}\,,
\end{split}
\end{equation}
where we adopted the shorthand notation $ d\mathbf{s}=  d\mathbf{s}_1 \,d\mathbf{s}_2 $, $d\mathbf{y}= d\mathbf{y}_1 \,d\mathbf{y}_2$ and we applied the Fubini theorem. Integrating by parts with respect to $\mathbf{y}_1$ and $\mathbf{y}_2$ and applying the semigroup property of $P$, we obtain
\[\begin{split}
&\int_{0}^{s\wedge r}\int_{\bT} \int_{0}^{\mathbf{s}_2}\int_{\bT} p^{0}_{s-\mathbf{s}_1}(x-\mathbf{y}_1) p^{0}_{s-\mathbf{s}_2}(x-\mathbf{y}_2)p^{0}_{r-\mathbf{s}_1}(y-\mathbf{y}_1)p^{0}_{r-\mathbf{s}_2}(y-\mathbf{y}_2) d\mathbf{s}\,d\mathbf{y}=\\&=\int_{0}^{s\wedge r}\int_{\bT} \int_{0}^{\mathbf{s}_2}\int_{\bT} [\partial_tP^{0}_{s-\mathbf{s}_1}](x-\mathbf{y}_1)P^{0}_{r-\mathbf{s}_1}(y-\mathbf{y}_1)P^{0}_{s-\mathbf{s}_2}(x-\mathbf{y}_2) [\partial_tP^{0}_{r-\mathbf{s}_2}](y-\mathbf{y}_2) d\mathbf{s}\,d\mathbf{y}\\&=\int_{0}^{s\wedge r} \int_{0}^{\mathbf{s}_2} \partial_tP^{0}_{s+r-2\mathbf{s}_1}(y-x)\partial_t P^{0}_{s+r-2\mathbf{s}_2}(y-x)d\mathbf{s}_1\,d\mathbf{s}_2\\&= \frac{1}{2} (P^{0}_{\abs{r-s}}(y-x))^2 - (P^{0}_{s+r}(y-x))^2
\end{split}\]
Bounding the terms involving the functions $\gp$, $l$ and $h$ with a deterministic constant and applying the rough estimate
\[(P^{0}_{\vert s- r\vert}(y-x)- P^{0}_{s+ r}(y-x))^2\leq (P^{0}_{\vert s- r\vert}(y-x))^2 +  (P^{0}_{s+ r}(y-x))^2,\]
there exists a constant $M>0$ such that for any $\gep>0$ one has
\[\norm{F^{\gep}}^{2}_{L^2}\leq  M\int_0^{T}\int_{\bT^2}(P^{0}_{s}(y-x))^2 \,ds \, dy\, dx = M\int_0^{T}C (s) \,ds <+ \infty\,.\]
Thereby obtaining the thesis. To conclude the result when $\psi $ is a generic smooth function supported on $(0,t)\times \bT$, we apply the formula \eqref{formula_ident} with a sequence of test functions $h_N\otimes l_N\colon (0,t)\times \bT\to \bR$ converging to $\psi$ as rapidly decreasing functions. This convergence is very strong and writing $\hat{\cR} \bigl(\gP''(U)(D_xU)^2\bigr)(h_N\otimes l_N)$ as the right-hand side of \eqref{reverse_equation} we can use the same argument  as before to prove
\[\hat{\cR} \bigl(\gP''(U)(D_xU)^2\bigr)(h_N\otimes l_N)\overset{L^2(\gO)}{\longrightarrow} \hat{\cR} \bigl(\gP''(U)(D_xU)^2\bigr)(\psi)\,.\]
\[- \frac{1}{2}\int_{0}^{t}\int_{\bT}(h_N(s) l_N(y))\gp''(u(s,y)) C(s)dy\overset{L^2(\gO)}{\longrightarrow}- \frac{1}{2}\int_{0}^{t}\int_{\bT}\psi(s,y)\gp''(u(s,y)) C(s)dy\,.\]
Then we can repeat the same argument above to prove that the double Skorohod integral converges to the respective quantity. When $\psi $ is a generic test function defined on $(0,+\infty)\times \bT$ we repeat the same calculations with the sequence of tests function $\gp_N \psi$ where $\gp_N $ is introduced in \eqref{indicator} and it converges a.e. to the indicator function $\1_{(0,t)\times \bT}$.
\end{proof}
\begin{remark}\label{indicator_rk}
The Proposition \ref{ident1} and the Theorem \ref{ident2} are formulated when the test function $ \psi\colon \bR\times \bT\to \bR$ is supported on positive times in order to be coherent with the statement of the Theorem \ref{Integral_Ito} and \cite[Theorem 6.2]{MartinHairer2015}. However for any generic test function $ \psi$, we can apply the identification theorems to the sequence  $ \gp_N\psi$ given in \ref{indicator} and the explicit definition of the indicator operator $ \1_{[0,t]}$ to obtain that the same result holds without any restriction on the support of the test function.
\end{remark}
\begin{remark}
The approximating procedure we used to prove this result is very different compared to the proof of \cite[Theorem 6.2]{MartinHairer2015}. In that case, the result is obtained by studying an approximating sequence inspired from the proof of Theorem \ref{reconstruction_theorem}. That is $\1_{[0,t]}\hat{\cR}(\gP'(U)\Xi)$ is  the a.s. limit in the $\cC^{-3/2-\kappa}$ topology of the smooth random fields
\begin{equation}\label{hope}
\hat{\cR}^n(\gP'(U)\Xi)(z):= \1_{[0,t]}(t)\sum_{\bar{z} \in \gL ^n([0,t])}\hPi_{\bar{z}}(\gP'(U)\Xi(\bar{z}))(\gp^n_{\bar{z}})\gp^n_{\bar{z}}(z)\,,
\end{equation}
where $\gL^n([0,T])$ denotes the dyadic grid on $[0,T]\times \bT$ of order $n$ and the functions $\gp^n_{\bar{z}}(z)$ are obtained by rescaling of a specific compactly supported function $\gp\colon \bR\times \bT\to \bR$. When we study the sequence \eqref{hope} in the $L^2(\gO)$ topology, the behaviour of this sequence is completely determined by knowing only the terms $\hPi_z  (\tau\Xi))(\gp^n_{\bar{z}})$ for $\tau \in U$. Thus we can apply the identity \eqref{stochastic integral} and conclude. However, considering the same approximations for $\1_{[0,t]}\hat{\cR}(\gP''(U)(D_xU)^2)$, we do not have the same simplification. In particular, the splitting of the heat kernel $G$ as a sum $K+R$ as explained in the Lemma \ref{KandR1} make all the calculations very indirect and it does not allow to use directly the explicit structure of $P$. A general methodology to describe the stochastic properties of the reconstruction operator for the BPHZ model is still missing.
\end{remark}
\begin{remark}\label{periodic_writing}
From the formulae \eqref{ident1_equation} and \eqref{formula_ident}, we can easily write the periodic extension of the reconstruction defined above. Indeed for any smooth function $\psi\colon \bR^2\to \bR $ with supp$(\psi) \subset (0,+\infty) \times \bR$ we have the identities
\begin{equation}\label{periodic_ext1}
\widetilde{\1_{[0,t]}\widehat{\cR}(\gP'(U)\Xi))}(\psi)= \int_0^t  \int_{\bR} \psi (s,y)\gp'(\widetilde{u}_s(y)) d\widetilde{W}_{s,y}\,,
\end{equation}
\begin{equation}\label{periodic_ext2}
\begin{gathered}
\widetilde{\1_{[0,t]}\widehat{\cR}(\gP''(U)(D_xU)^2}(\psi)=-\frac{1}{2}\int_0^t\int_{\bR} \psi(s,y)\gp''(\widetilde{u}_s(y))C(s) dy\, ds\\
+\int_{[0,t]^2\times \bR^2}\left[\int_{\mathbf{s}_2\vee \mathbf{s}_1}^t \int_{\bR} \psi(s,y)\gp''(\widetilde{u}_s(y))\partial_xG_{s-\mathbf{s}_1}( y- \mathbf{y}_1)\partial_xG_{s-\mathbf{s}_2}( y- \mathbf{y}_2) dyds \right]d\widetilde{W^2}_{\mathbf{s},\mathbf{y}}\,.
\end{gathered}
\end{equation}
And the indicator operator on the right-hand side tell us that that these identities hold also for any smooth function $\psi$ (see Remark \ref{indicator_rk}). 
\end{remark}

\subsection{Identification of the integral formula}
We pass to the identification of the terms involving the convolution with $P$. In principle this operation is deterministic and it should be obtained by applying the previous results to the deterministic test function $\psi\colon \bR\times \bT\to \bR$ given by $\psi(s,y)= P_{t-s}( x-y)$ for some $(t,x)\in [0,T]\times \bT$. However the function $\psi$ is not smooth because $\psi$ has a singularity at $(t,x)$. In order to skip this obstacle we  recall an additional property of the function  $K\colon \bR^2\setminus \{0\}\to \bR$, introduced in the Lemma \ref{KandR1} and the Lemma \ref{KandR2}.
\begin{lemma}\label{KandR3}
There exists a sequence of smooth positive function $K_{n}\colon \bR^2\to \bR $, $n\geq 0$ satisfying $\text{\emph{supp}}(K_n)=\{z=(t,x)\in \bR^2\colon  \norm{z}\leq 2^{-n},\; t>0\}$ such that for any $z\in \mathbb{R}^2\setminus\{0\}$
\begin{equation}\label{smooth_decomposition}
K(z)= \sum_{n \geq 0}K_n(z)\,.
\end{equation}
Moreover for every distribution $u\in \cC^{\ga}$ with $-2<\ga<0$ non integer one has for any $z\in \bR^2$ 
\begin{equation}\label{smooth_decomposition2}
(K*u)(z)= \sum_{n\geq 0}(K_n*u)(z)\,.
\end{equation}
\end{lemma}
\begin{proof}
The Kernel $K$ satisfies automatically the property \eqref{smooth_decomposition} by construction, as expressed in \cite[Ass. 5.1]{Hairer2014}.  Moreover for all test functions $\psi $  and $N\geq 0$ we have the identity
\begin{equation}\label{really_final_hard}
((\sum_{n= 0}^N K_n)*u)(\psi)= \sum_{n= 0}^N (K_n*u)(\psi)\,.
\end{equation}
Following \cite[Lem. 5.19]{Hairer2014}, the right-hand side sequence of \eqref{really_final_hard} is a Cauchy sequence with respect to the topology of $\cC^{\ga+2}$. Thus by uniqueness of the limit we obtain the equality 
\begin{equation}\label{final_hard_equality}
(K*u)= \sum_{n\geq  0} (K_n*u)\,,
\end{equation}
as elements of $\cC^{\ga+2}$. Since $\ga+2>0$ \eqref{final_hard_equality} becomes an equality between functions, thereby obtaining the thesis.
\end{proof}

\begin{theorem}\label{last_theorem}
Let $\gp\in C^7_b(\bR)$. Then for any $(t,x)\in [0,T]\times \bT$ one has the a.s. identities
\begin{equation}\label{almost_last_integral}
P*(\1_{[0,t]}\hat{\cR}(\gP'(U)\Xi))(t,x)= \int_0^t  \int_{\bT} P_{t-s}(x-y)\gp'(u_s(y)) dW_{s,y}\,,
\end{equation}
\begin{equation}\label{last_integral}
\begin{gathered}
P*(\1_{[0,t]}\hat{\cR}(\gP''(U)(D_xU)^2)(t,x)=-\frac{1}{2}\int_0^t\int_{\bR} \psi(s,y)\gp''(\widetilde{u}_s(y))C(s) dy\, ds\\
+\int_{[0,t]^2\times \bR^2}\left[\int_{\mathbf{s}_2\vee \mathbf{s}_1}^t \int_{\bR} \psi(s,y)\gp''(\widetilde{u}_s(y))\partial_xP_{s-\mathbf{s}_1}( y- \mathbf{y}_1)\partial_xP_{s-\mathbf{s}_2}( y- \mathbf{y}_2) dyds \right]dW^2_{\mathbf{s},\mathbf{y}}\,.
\end{gathered}
\end{equation}
\end{theorem}
\begin{proof}
We will prove equivalently the identities \eqref{almost_last_integral} and \eqref{last_integral} on the periodic extension. Using Lemma \ref{KandR1},  for any periodic random distribution $v$ and $(t,x)\in [0,T]\times \bR$ we have the general identity  
\begin{equation}\label{last_lifting}
\widetilde{(P*\1_{[0,t]}v)}(t,x)=(G*\widetilde{\1_{[0,t]}v})(t,x)= (K*\widetilde{\1_{[0,t]}v})(t,x)+ (R*\widetilde{\1_{[0,t]}v})(t,x)\,.
\end{equation}
By choosing $v=\hat{\cR}(\gP''(U)(D_xU)^2) $, $\hat{\cR}(\gP'(U)\Xi)$, we can apply directly Theorem \ref{ident2} and the previous formulae \eqref{periodic_ext1}, \eqref{periodic_ext2} to the term involving the kernel $R$. Thus the theorem is proved as long as these formulae are also true for the kernel $K$. To calculate it, we apply Lemma \ref{KandR3} obtaining  for any periodic random distribution $v$ the identity
\[K*(\widetilde{\1_{[0,t]}v})(t,x)= \sum_{n\geq 0}K_n*(\widetilde{\1_{[0,t]}v})(t,x)= \lim_{N\to +\infty}\widetilde{\1_{[0,t]}v}(\gh_{N})\,,\]
where $\gh_{N}\colon \bR^2\to \bR$ is the sequence of compactly smooth functions 
\[\gh_{N}(s,y):= \sum_{n= 0}^N  K_n(t-s,x-y)\,,\]
and the convergence is a.s. By applying the Theorem \ref{ident2} to the sequence of function $\gh_{N}$, we study the convergence of the sequence $\widetilde{\1_{[0,t]}v}(\gh_{N})$ in the  $L^2(\gO)$ topology when $v$ is equal to $\hat{\cR}(\gP''(U)(D_xU)^2) $ or $\hat{\cR}(\gP'(U)\Xi)$. In case $v=\hat{\cR}(\gP'(U)\Xi)$ one has trivially
\[\widetilde{\1_{[0,t]}v}(\gh_{N})=\int_0^t  \int_{\bR}\gh_{N}(s,y)\gp'(\widetilde{u}(s,y))d\widetilde{W}_{s,y}\,.\]
Since $\gp' $ is bounded, there exists a constant $M>0$ such that for any $(s,y)\in [0,t]\times \bR$ and $N\geq 0$ one has
\[\vert \gh_{N}(s,y)\gp'(\widetilde{u}_s(y))\vert\leq M G_{t-s}(x-y)\,.\]
The function $(s,y)\to G_{t-s}(x-y) $ is $L^2$ integrable on $[0,t]\times \bR$  and $\gh_{N}(s,y)$ converges a.e. to $K$. By using the It\^o isometry and the dominated convergence theorem, we can straightforwardly prove
\[ \widetilde{\1_{[0,t]}v}(\gh_{N})\overset{L^2(\gO)}{\longrightarrow}  \int_0^t  \int_{\bR}K(t-s,x-y)\gp'(\widetilde{u}_{s}(y))d\widetilde{W}_{s,y}\,,\]
Thereby obtaining the identity \eqref{periodic_ext1} with the kernel $K$. Let us consider the case $v=\hat{\cR}(\gP''(U)(D_xU)^2)$. To shorten the notation we adopt the convention $g_{s-r}(x-y):=\partial_x G_{s-r}( x-y)$ and $O_t=[0,t]\times \bR$. Looking again at the equation \eqref{periodic_ext2} we have $\widetilde{\1_{[0,t]}v}(\gh_{N})=A_N^1+A^2_N$ where
\[
A^1_N=- \frac{1}{2}\int_{O_t} \gh_{N}(s,y)\gp''(\widetilde{u}_s(y))C(s) dy\, ds\,,
\]
\[\begin{split}
A^2_N&=\int_{O_t\times O_t}\left[\int_{\mathbf{s}_2\vee \mathbf{s}_1}^t \int_{\bR}\gh_{N}(s,y)\gp''(\widetilde{u}_s(y))g_{s-\mathbf{s}_1}( y- \mathbf{y}_1)g_{s-\mathbf{s}_2}( y- \mathbf{y}_2) dyds \right]d\widetilde{W}^{2}_{\mathbf{s},\mathbf{y}}\,.
\end{split}\]
From the definition of $\gh_N$, one has a.e. and a.s.
 \[\gh_{N}(s,y)\gp''(\widetilde{u}(s,y))C(s)\to K(t-s,x-y) \gp''(\widetilde{u}(s,y))C(s)\,.\]
Moreover there exists a constant $M>0$ such that for every $N\geq 0$
\[\abs{ A_{N}^1}\leq M\int_{O_t}G_{t-s}(x-y) C(s)dy\, ds= M\int_0^t C(s)\, ds<\infty\,\]
(the last equality comes by integrating on $\bR$ the density function of a Gaussian random variable). Therefore we obtain by dominated convergence
\[A^1_N\overset{L^2(\gO)}{\longrightarrow}-\frac{1}{2}\int_0^t\int_{\bR}K(t-s,x-y)\gp''(\widetilde{u}_s(y))C(s) dyds \,.\]
Let us pass to the convergence of the sequence $A^2_N$. Introducing the functions $\{\gP^N\}_{N\geq 0}$, $\gP_K$, $\{F^N\}_{N\geq 0}$ and $F_K$ defined by
\[\gP^{N}(s,y,\mathbf{s},\mathbf{y}):=\gh_{N}(s,y) \gp''(\widetilde{u}_{s}(y))g_{s-\mathbf{s}_1}( y- \mathbf{y}_1)g_{s-\mathbf{s}_2}(y- \mathbf{y}_2)\,,\]
\[\gP_K(s,y,\mathbf{s},\mathbf{y}):=K(t-s,x-y)\gp''(\widetilde{u}_{s}(y))g_{s-\mathbf{s}_1}( y- \mathbf{y}_1)g_{s-\mathbf{s}_2}(y- \mathbf{y}_2)\,,\]
\[F^{N}(\mathbf{s},\mathbf{y}):= \int_{\mathbf{s}_2\vee\mathbf{s}_1}^{t}\int_{\bR}\gP^{N}(s,y,\mathbf{s},\mathbf{y})ds dy\,, \quad F_K(\mathbf{s},\mathbf{y}):= \int_{\mathbf{s}_2\vee\mathbf{s}_1}^{t}\int_{\bR}\gP_K(s,y,\mathbf{s},\mathbf{y})ds dy,\]
(as usual $\mathbf{s}=(\mathbf{s}_1,\mathbf{s}_2)$ and $\mathbf{y}=(\mathbf{y}_1,\mathbf{y}_2)$ and $\mathbf{s}_1\vee \mathbf{s}_2\leq t$), we will prove the last convergence
\begin{equation}\label{last_convergence}
A^2_N\overset{L^2(\gO)}{\longrightarrow}\int_{O_t\times O_t} F_K(\mathbf{s},\mathbf{y})d\widetilde{W}^{2}_{\mathbf{s},\mathbf{y}}\,.
\end{equation}
The multiple Skorohod integral is a continuous map from $H= \mathbb{D}^{2,2}(L^2(O_t\times O_t))$ to $L^2(\gO)$ (see the definition of $\mathbb{D}^{2,2}(V)$ and the inequality \eqref{continuity_property} in the section \ref{elements}). Then
the convergence  \eqref{last_convergence} will follow by proving that $F^N$ and $F_K$ belong to $H$ and $F^N\to F_K$ in $H$. To prove these results, we calculate the first and second Malliavin derivative of $F^N$ and $F_K$ thanks to chain rule formula of the Malliavin derivative (see \cite[Proposition 1.2.3]{nualart1995malliavin}). In particular, we have 
\begin{equation}\label{mall_deriv1}
\begin{split}
\nabla_{\mathbf{t}_1\mathbf{z}_1} F^N(\mathbf{s},\mathbf{y})= \int_{\mathbf{s}_2\vee\mathbf{s}_1\vee \mathbf{t}_1}^{t}\int_{\bR}&\gh_{N}(s,y) \gp^{(3)}(\widetilde{u}_{s}(y))G_{s-\mathbf{t}_1}(y-\mathbf{z}_1)\\&\times g_{s-\mathbf{s}_1}( y- \mathbf{y}_1)g_{s-\mathbf{s}_2}(y- \mathbf{y}_2)dy ds\,,
\end{split}
\end{equation}
\begin{equation}\label{mall_deriv2}
\begin{split}
\nabla^2_{\mathbf{t}_1\mathbf{z}_1\mathbf{t}_2\mathbf{z}_2} F^N(\mathbf{s},\mathbf{y})=\int_{\mathbf{s}_2\vee\mathbf{s}_1\vee\mathbf{t}_1\vee\mathbf{t}_2}^{t}\int_{\bR}\gh_{N}(s,y) &\gp^{(4)}(\widetilde{u}_{s}(y))G_{s-\mathbf{t}_2}(y-\mathbf{z}_2)G_{s-\mathbf{t}_1}(y-\mathbf{z}_1)\\&\times g_{s-\mathbf{s}_1}( y- \mathbf{y}_1)g_{s-\mathbf{s}_2} (y-\mathbf{y}_2)dy ds\,.
\end{split}
\end{equation}
and similarly for $F_K$ by replacing $\gh_N$ with $K(t-s,x-y)$. Bounding uniformly $\gh_N$ and the kernel $K$ by $G_{t-s}(x-y)$, we have trivially $\norm{F^N}^2\leq \norm{F_G}^2$ for every $N\geq 0$ and $ \norm{F_K}^2\leq \norm{F_G}^2$ where 
\[F_G(\mathbf{s},\mathbf{y}):= \int_{\mathbf{s}_2\vee\mathbf{s}_1}^{t}\int_{\bR}\gP_G(s,y,\mathbf{s},\mathbf{y})ds dy\,,\]
\[\gP_G(s,y,\mathbf{s},\mathbf{y}):=G_{t-s}(x-y)\gp''(\widetilde{u}_{s}(y))g_{s-\mathbf{s}_1}( y- \mathbf{y}_1)g_{s-\mathbf{s}_2}(y- \mathbf{y}_2).\]
Since the Malliavin derivatives of $F_G$ are given by \eqref{mall_deriv1} and \eqref{mall_deriv2} where $\gh_N$ is replaced by $G(t-\cdot, x-\cdot)$, it is sufficient to prove the convergence that the random variables
\[\ga_1:=\int_{(O_t)^2}\left(F_G(\mathbf{s},\mathbf{y})\right)^2d\mathbf{s}\,d\mathbf{y}\,,\quad\ga_2:=\int_{(O_t)^3}(\nabla_{\mathbf{t}_1\mathbf{z}_1} F_G(\mathbf{s},\mathbf{y}))^2d\mathbf{s}\,d\mathbf{y} d\mathbf{t}_1d\mathbf{z}_1\,,\]
\[ \ga_3:=\int_{(O_t)^4}(\nabla^2_{\mathbf{t}_1\mathbf{z}_1\mathbf{t}_2\mathbf{z}_2} F_G(\mathbf{s},\mathbf{y}))^2d\mathbf{s}\,d\mathbf{y} d\mathbf{t}\,d\mathbf{z}\,,\qquad \mathbf{t}=(\mathbf{t}_1,\mathbf{t}_2)\,,\quad\mathbf{z}=(\mathbf{z}_1,\mathbf{z}_2)\,,\]
are uniformly bounded. Let us analyse them separately. Looking at the expression of $\ga_1$, it is possible to express the term $g_{s-\mathbf{s}_1}( y- \mathbf{y}_1)g_{s-\mathbf{s}_2} (y-\mathbf{y}_2)$ appearing in $F^2_G$ in the same way as in the equation \eqref{simplification}, by replathe kernels $p$ and $P$  by $g$ and $G$, thereby obtaining
\[\begin{split}
\ga_1=\int_{(O_t)^2}G_{t-r}&(x-z) G_{t-s}(x-y) \gp''(\widetilde{u}_{s}(y))\gp''(\widetilde{u}_{r}(z))\\&\times(G_{\vert s- r\vert}(y-z)- G_{s+ r}(y-z))^2 drds dydz\,.\end{split}\]
By hypothesis on $\gp$ and bounding roughly the difference of a square there exists a constant $M>0$ such that for every $N\geq 0$
\begin{equation}\label{det_integral}
\ga_1\leq M \int_{O_t\times O_t} G_{t-s}( x-y)G_{t-r}(x-z)(G_{\vert s- r\vert}(z-y))^2+ (G_{s+ r}(z-y))^2 dr ds dy dz\,.
\end{equation}
Let us show that the deterministic integral in the right-hand side of \eqref{det_integral} is finite. By definition of $G$ one has
\begin{equation}\label{square_heat}
G_{\vert s- r\vert}(z-y)^2 + G_{s+ r}(z-y)^2=  \frac{G_{\vert s- r\vert/2}(z-y)}{\sqrt{8\pi\vert s- r\vert}}+ \frac{G_{(s+ r)/2}(z-y)}{\sqrt{8\pi (s+ r)}}\,.
\end{equation}
Plugging this identity in the deterministic integral in the right-hand side of \eqref{det_integral}, we can apply the semigroup property of $G$ and  a rough estimate on the Heat kernel to show that there exist some constants  $C,\, C'>0$ such that
\[\begin{split}
&\int_{O_t\times O_t}  G_{t-s}( x-y)G_{t-r}(x-z)\left(\frac{G_{\vert s- r\vert/2}(z-y)}{\sqrt{8\pi \vert s- r\vert}}+ \frac{G_{(s+ r)/2}(z-y)}{\sqrt{8\pi (s+ r)}}\right)dr \,ds \,dy \,dz\\&=\int_{0}^t\int_{O_t}  G_{t-s}( x-y)\left(\frac{G_{t-r+\vert s- r\vert/2}(x-y)}{\sqrt{8\pi \vert s- r\vert}} +  \frac{G_{t+(s-r)/2}(x-y)}{\sqrt{8\pi (s+ r)}}\right)ds\,dr\,dy\\&\leq C \int_{0}^t\int_{O_t}  \frac{1}{\sqrt{t-s}}\left(\frac{G_{t-r+\vert s- r\vert/2}(x-y)}{\sqrt{8\pi \vert s- r\vert}} +  \frac{G_{t+(s-r)/2}(x-y)}{\sqrt{8\pi (s+ r)}}\right)ds\,dr\,dy\\&\leq C'\int_0^t \int_0^t \left( \frac{1}{\sqrt{t-s}}\frac{1}{\sqrt{ \vert s- r\vert}}  +  \frac{1}{\sqrt{t-s}}\frac{1}{\sqrt{ s+ r}}\right) ds\,dr<+\infty\,.\end{split}\]
(we are again integrating on $\bR$  the density function of a Gaussian random variable). Passing to $\ga_2$, we rewrite $(\nabla_{\mathbf{t}_1\mathbf{z}_1} F(\mathbf{s},\mathbf{y}))^2$ as a double integral and applying again the semigroup property of $G$ we have
\[\begin{split}
&\ga_2= 2\int_{(\gD_{2,t}\times \bR^2)\times O_t}(\nabla_{\mathbf{t}_1\mathbf{z}_1} F_G(\mathbf{s},\mathbf{y}))^2d\mathbf{s}\,d\mathbf{y} d\mathbf{t}_1d\mathbf{z}_1\\&=2\int_{(O_t)^2}\,G_{t-s}( x-y)G_{t-r}(x-z)\gp^{(3)}(\widetilde{u}_{s}(y))\gp^{(3)}(\widetilde{u}_{r}(z))\gG_{s,r}^3(z,y) drds dydz\,,\end{split}\]
where the function $\gG_{s,r}^3(z,y)$ is defined through the identities
\[\begin{split}
\gG_{s,r}^3(z,y)&:= \int_0^{s\wedge r}\int_0^{\mathbf{t}_1}\int_{0}^{\mathbf{s}_2}\gG_{s,r,\mathbf{s},\mathbf{t}_1}^3(z,y) d\mathbf{s}_1d\mathbf{s}_2d\mathbf{t}_1\\&+\int_0^{s\wedge r}\int_0^{\mathbf{s}_2}\int_{0}^{\mathbf{s}_1}\gG_{s,r,\mathbf{s},\mathbf{t}_1}^3(z,y) d\mathbf{t}_1d\mathbf{s}_1d\mathbf{s}_2\\& + \int_0^{s\wedge r}\int_0^{\mathbf{s}_2}\int_{0}^{\mathbf{t}_1}\gG_{s,r,\mathbf{s},\mathbf{t}_1}^3(z,y) d\mathbf{s}_1d\mathbf{t}_1d\mathbf{s}_2\,,\end{split}\]
\[ \gG_{s,r,\mathbf{s},\mathbf{t}_1}^3(z,y):=G_{s+r-2\mathbf{t}_1}(y-z) g_{s+r-2\mathbf{s}_1}( y- z)g_{s+r-2\mathbf{s}_2}(y- z)\,.\]
Let us consider the term $\gG_{s,r}^3(z,y)$. Using the elementary estimates
\begin{equation}\label{elem_estimates}
\abs{ g_{t}(y)}\leq \sup_{u\in\bR}(\frac{u}{\sqrt{4\pi}} \exp^{-u^2})\frac{1}{t}\,,\qquad \abs{ G_{t}(y)}\leq \frac{1}{\sqrt{t}}\ \,,
\end{equation} 
we can bound each term in the sum defining $\gG_{s,r}^3(z,y)$ by some integrable functions depending only on $r$ and $s$. For example, in case of the first term in the sum defining $\gG_{s,r}^3(z,y)$ there exists a constant $C>0$ such that
\[\begin{split}
&\bigg\vert\int_0^{s\wedge r}\int_0^{\mathbf{t}_1}\int_{0}^{\mathbf{s}_2}\gG_{s,r,\mathbf{s},\mathbf{t}_1}^3(z,y) d\mathbf{s}_1d\mathbf{s}_2d\mathbf{t}_1\bigg\vert \leq \\&\leq  C \int_0^{s\wedge r}\bigg[\int_0^{\mathbf{t}_1}\bigg[\int_{0}^{\mathbf{s}_2} \frac{1}{ s+r- 2\mathbf{s}_1}d\mathbf{s}_1\bigg]\frac{1}{ s+r- 2\mathbf{s}_2}d\mathbf{s}_2 \bigg]\frac{1}{\sqrt{s+r- 2\mathbf{t}_1}}\,d\mathbf{t}_1\,.\end{split}\]
Writing explicitly the integral on the right-hand side, there exists a constant $C'>0$ such that this integral is bounded by
\[\begin{split}
&C'\left[\ln(s+r)^2(\sqrt{s+r}- \sqrt{\abs{s-r}})+\ln(s+r)\int^{s+r}_{\abs{s-r}} \frac{\abs{\ln(y) }}{\sqrt{y}} dy+\int^{s+r}_{\abs{s-r}} \frac{\ln(y)^2}{\sqrt{y}} dy\right]\\&\leq C'((\ln(2T)^2\sqrt{2T})\vee 1)\left[1+\int^{s+r}_{\abs{s-r}} \frac{\abs{\ln(y) }}{\sqrt{y}} dy+\int^{s+r}_{\abs{s-r}} \frac{\ln(y)^2}{\sqrt{y}} dy\right].
\end{split}\]
Working exactly in the same way on the other integrals, it is possible to show  that there exists a constant $D_T>0$ depending on $T$ such that
\begin{equation}\label{integrable_function}
\begin{split}
&\abs{\gG_{s,r}^3(z,y)} \leq D_T \bigg(1 +\int^{s+r}_{\abs{s-r}} \frac{\ln(y)^2}{\sqrt{y}\vee 1} dy+\int^{s+r}_{\abs{s-r}} \frac{\abs{\ln(y)}}{\sqrt{y}\vee 1} dy\bigg)\,.\end{split}
\end{equation}
Let us denote the right-hand side of \eqref{integrable_function} by $ C_T(s,r)$. This function is integrable on $[0,t]^2$. By integrating on the remaining components and bounding the derivatives with some uniform constant, there exists a constant $M>0$ such that 
\[\begin{split}\ga_2&\leq M\int_{(O_t)^2} G_{t-s}( x-y)G_{t-r}(x-z)C_T(s,r)drds dydz\\&= M\int_{[0,t]^2}C_T(s,r)drds <+\infty\,.\end{split}\] 
We consider the last term $\ga_3 $. By writing $(\nabla^2_{\mathbf{t},\mathbf{z}} F(\mathbf{s},\mathbf{y}))^2$ with the same technique to express $\ga_2$, we obtain
\[\begin{split}
&\ga_3= 8\int_{(\gD_{2,t}\times \bR^2)^2}(\nabla^2_{\mathbf{t},\mathbf{z}} F_G(\mathbf{s},\mathbf{y}))^2d\mathbf{s}\,d\mathbf{y} d\mathbf{t}d\mathbf{z}\\&=8\int_{(O_t)^2}\,G_{t-s}( x-y)G_{t-r}(x-z)\gp^{(4)}(\widetilde{u}_{s}(y))\gp^{(4)}(\widetilde{u}_{r}(z))\gG_{s,r}^4(z,y)drds dydz\,.\end{split}\]
(the factor $8$ comes out because the function $(\nabla^2_{\mathbf{t},\mathbf{z}} F^N(\mathbf{s},\mathbf{y}))^2$ is symmetric under the change of coordinates $\mathbf{s}_1\to \mathbf{s}_2 $, $\mathbf{t}_1\to \mathbf{t}_2 $ and $\mathbf{s}\to \mathbf{t}$). The function $\gG_{s,r}^4(z,y)$ is defined through the new  identities
\[\begin{split}
\gG_{s,r}^4(z,y)&:= \int_0^{s\wedge r}\int_0^{\mathbf{t}_2}\int_{0}^{\mathbf{t}_1}\int_{0}^{\mathbf{s}_2}\gG_{s,r,\mathbf{s},\mathbf{t}}^4(z,y) d\mathbf{s}d\mathbf{t}\\&+ \int_0^{s\wedge r}\int_{0}^{\mathbf{t}_2}\int_0^{\mathbf{s}_2}\int_{0}^{\mathbf{s}_1}\gG_{s,r,\mathbf{s},\mathbf{t}}^4(z,y) d\mathbf{t}_1d\mathbf{s}_1d\mathbf{s}_2d\mathbf{t}_2\\&+\int_0^{s\wedge r}\int_{0}^{\mathbf{t}_2}\int_0^{\mathbf{s}_2}\int_{0}^{\mathbf{t}_1}\gG_{s,r,\mathbf{s},\mathbf{t}}^4(z,y) d\mathbf{s}_1d\mathbf{t}_1d\mathbf{s}_2d\mathbf{t}_2 \;,\end{split}\]
\[\gG_{s,r,\mathbf{s},\mathbf{t}}^4(z,y):=G_{s+r-2\mathbf{t}_1}(y-z) G_{s+r-2\mathbf{t}_2}(y-z)g_{s+r-2\mathbf{s}_1}( y- z)g_{s+r-2\mathbf{s}_2}(y- z)\,.\]
Recalling the elementary estimates in \eqref{elem_estimates}, we can similarly bound every single integral appearing in $\gG_{s,r}^4(z,y)$ in the same way implying there exists an integrable function $B_T(r,s)$ such that $\abs{\gG_{s,r}^4(z,y)} \leq B_T(s,t) $. Bounding $\gp^{(4)}$ we conclude there exists a constant $M>0$ such that
\[\begin{split}\ga_3&\leq M\int_{(O_t)^2} G_{t-s}( x-y)G_{t-r}(x-z)B_T(s,r)drds dydz\\&= M\int_{[0,t]^2}B_T(s,r)drds <+\infty\,.\end{split}\]
Thus we conclude that the random variables $\ga_1$, $\ga_2$ and $\ga_3$ are uniformly bounded and $F^N,F_K\in H$. As a matter of fact, the previous estimates have a stronger consequence because they imply that the functions $\gP^{N}(s,y,\mathbf{s},\mathbf{y}) $ and $\gP_K(s,y,\mathbf{s},\mathbf{y}) $  defined above are a.e. on $s$, $y$, $\mathbf{s}$, $\mathbf{y}$ and a.s. dominated by some integrable functions. Rewriting the norm on $H$ as follows
\begin{equation}\label{first_convergence}
\begin{split}
\norm{F^N - F_K}_H^2 &=\bE \int_{(O_t)^2}\left(F^N(\mathbf{s},\mathbf{y})-F_K(\mathbf{s},\mathbf{y})\right)^2d\mathbf{s}d\mathbf{y}\\&+\bE\int_{(O_t)^3}(\nabla_{\mathbf{t}_1\mathbf{z}_1} F^N(\mathbf{s},\mathbf{y})-\nabla_{\mathbf{t}_1\mathbf{z}_1} F_K(\mathbf{s},\mathbf{y}))^2d\mathbf{s}\,d\mathbf{y} d\mathbf{t}_1d\mathbf{z}_1\\&+\bE\int_{(O_t)^4}(\nabla^2_{\mathbf{t},\mathbf{z}} F^N(\mathbf{s},\mathbf{y})-\nabla^2_{\mathbf{t},\mathbf{z}} F_K(\mathbf{s},\mathbf{y}))^2d\mathbf{s}\,d\mathbf{y} d\mathbf{t}d\mathbf{z} \,,
 \end{split}
 \end{equation}
we obtain $\norm{F^N - F_K}_H^2\to 0 $  by dominated convergence because we have trivially the a.e. a.s. the convergence of the functions
\[\gP^{N}(s,y,\mathbf{s},\mathbf{y}) \to \gP_K(s,y,\mathbf{s},\mathbf{y})\,, \quad \nabla_{\mathbf{t}_1\mathbf{z}_1} F^N(\mathbf{s},\mathbf{y})\to \nabla_{\mathbf{t}_1\mathbf{z}_1} F_K(\mathbf{s},\mathbf{y})\,,\]
\[\nabla^2_{\mathbf{t},\mathbf{z}} F^N(\mathbf{s},\mathbf{y})\to\nabla^2_{\mathbf{t},\mathbf{z}} F_K(\mathbf{s},\mathbf{y})\,,\]
Thereby proving the theorem.
\end{proof}
\begin{proof}[Proof of the Theorem \ref{Integral_Ito}]
For any $\gp\in C^7_b(\bR)$ the differential and the integral formula are obtained applying straightforwardly the previous results. Looking at their proofs, we realise that the Skorohod and the Wiener integrals and their convolution with $P$, differently from the reconstructions, are well defined if the derivatives of $\gp$ are bounded up to the order $4$. Thus for any fixed $ \gp\in C^{4}_b(\bR)$  we can write the  differential and the integral It\^o formula on  $\gp_{\delta}$, a  sequence $\{\gp_{\delta}\}_{\delta>0}$ of smooth functions with all bounded derivatives converging to $\gp$.  Using the same calculations of Theorem \ref{last_theorem}, we can prove that the terms involving $\gp_{\delta}$ converges in $L^2(\gO)$ to the same terms involving $\gp$. Thereby obtaining the proof.
\end{proof}
\begin{remark}
Using the integrability of the random field $u$ in \eqref{integrability_u} and looking carefully at the proof of the identity \eqref{last_integral}, we could lower down slightly the hypothesis on $\gp$ in the Theorem \ref{Integral_Ito}, supposing that $\gp$ has only the second, the third and the fourth derivative bounded. Indeed the function $\gp'(u)$ will have linear growth and the right-hand side of \eqref{almost_last_integral} will be always well defined. In this way, the same argument given in the proof above should provide the Theorem \ref{Integral_Ito} even in this case. These slight modifications should allow us to obtain a differential and an integral formula even for the random field $u^2$, giving an interesting decomposition of this random field. 
\end{remark}
\section{Renormalisation constants}
We calculate the asymptotic behaviour of the renormalisations constants defined in \eqref{C_1}, \eqref{C_2}. A preliminary result to analyse them lies on a remarkable identity on $G$, the heat kernel on $\bR$, interpreted as a function $G\colon \bR^2\setminus \{0\}\to \bR$.
\begin{lemma}
For any $z\in \bR^2\setminus \{0\}$ one has
\begin{equation}\label{heat}
2\int_{\bR^2} G_x(z-\z)G_x(-\z)d\z=G(z)+ G(-z)
\end{equation}
\end{lemma}
\begin{proof}
We verify this identity by calculating the space-time Fourier transform
\[ f\to \widehat{f}(\xi)=  \int_{\bR^2}e^{-2\pi i (\xi\cdot z)}f(z)dz\]
of both sides. In order to do that, we recall the elementary identity
\[\widehat{G}(\xi)= \frac{1}{2\pi i\xi_1+ 4\pi^2 \xi_2^2}\,.\]
Using the notation $\overline{u}(z)= u(-z)$, for any function $u\colon \bR^2\to \bR$, we rewrite the left-hand side of (\ref{heat}) as $2G_x*\overline{G}_x(z)$. Applying the Fourier transform, we then obtain
\[\begin{split}2\widehat{G_x*\overline{G}_x}(\xi)&= 2\widehat{G_x}(\xi)\overline{\widehat{G_x}}(\xi)=2(2\pi i\xi_2\widehat{G}(\xi))(-2\pi i \xi_2\widehat{G}(-\xi))= \frac{8\pi^2\xi_2^2}{4\pi^2 \xi_1^2+ (4\pi^2 \xi_2^2)^2}\,. \end{split}\]
On the other hand, the same operation on the right-hand side of (\ref{heat}) implies
\[\widehat{G}(\xi)+ \widehat{G}(-\xi)= \frac{8\pi^2 \xi_2^2}{(2\pi i\xi_1+ 4\pi^2 \xi_2^2)(-2\pi i\xi_1+ 4\pi^2 \xi_2^2)}=\frac{8\pi^2\xi_2^2}{4\pi^2 \xi_1^2+ (4\pi^2 \xi_2^2)^2}\,.\]
By uniqueness of the Fourier transform, we conclude.
\end{proof}
\begin{theorem}\label{constants}
Let $C_{\gep}^{1} $, $C^{2}_{\gep}$ be the constants introduced in (\ref{C_1}), (\ref{C_2}). Using the convention $\gr^{*2}= \gr*\gr$, one has following estimates as $\gep \to 0^+$
\begin{align}
C_{\gep}^{1}&= \frac{1}{\gep} \int_{\bR^2}G(s,y) \gr^{*2}(s,y)ds dy+ o(1)\,,\\ C^{2}_{\gep}&= \frac{1}{\gep} \int_{\bR^2} (G_x*\gr)^2(s,y)ds dy+ o(1)\,,\\ C_{\gep}^{1}&= C^{2}_{\gep}+ o(1)\,.
\end{align}
\end{theorem} 
\begin{proof}
All the integrals we consider in the proof will be taken on the whole space $\bR^2$. We will not write it explicitly to simplify the notation. For any function $F\colon \bR^2\setminus \{0\} \to \bR$, any integer $m$ and $\gep>0$, we introduce the function $S^m_{\gep}(F)\colon \bR^2\setminus \{0\} \to \bR$ given by
\[S^m_{\gep}(F)(t,x):= \gep^m F(\gep^2 t,\gep x)\,.\]
Using the definition of $C_{\gep}^{1}$, together with the hypothesis $\gr(-z)=\gr(z)$ one has
\[C_{\gep}^{1}=\int\int K(w)\gr_{\gep}(z) \gr_{\gep}(z-w)dw dz \]
\[=\int K(w)\int\gr_{\gep}(z) \gr_{\gep}(w-z)dz dw =\int K(w)(\gr_{\gep})^{*2}(w)dw. \]
A simple change of variable tells us that $(\gr_{\gep})^{*2}(w)= (\gr^{*2})_{\gep}(w)$. Therefore we deduce that
\[C_{\gep}^{1}=\int K(t,x) \gep^{-3}\gr^{*2}\left(\frac{t}{\gep^2},\frac{x}{\gep}\right)dt dx = \int K(\gep^2 t,\gep x) \gr^{*2}(t,x)dt dx\]
\[= \frac{1}{\gep} \int  S^1_{\gep}(K)(t,x) \gr^{*2}(t,x)dt dx\,.\]
Since $S^1_{\gep}(K) $ is equal to $S^1_{\gep}(G)$ as $\gep \to 0^+$ and $G$ satisfies $S^1_{\gep}(G)=G $, one has 
\[S^1_{\gep}(K)(t,x) \gr^{*2}(t,x)\rightarrow G(t,x) \gr^{*2}(t,x) \quad \text{a.e.}\]
Moreover, it is straightforward to show that the function $G\gr^{*2}$ is integrable and it dominates $S^1_{\gep}(K)\gr^{*2}$, therefore we obtain 
\[\int S^1_{\gep}(K)(t,x)  \gr^{*2}(t,x)dt dx\rightarrow \int G(t,x) \gr^{*2}(t,x)dt dx\,,\]
by dominated convergence. We recover the identity (A.2), by using the decomposition $G=K+R$, as explained in the Lemma \ref{KandR1}. Writing again $S^1_{\gep}(G)=G $, we obtain
\[\frac{1}{\gep}\int \left[G(t,x)- S^1_{\gep}(K)(t,x) \right]\gr^{*2}(t,x)dt dx=\int S^0_{\gep}(R)(t,x)\gr^{*2}(t,x)dtdx\,. \]
The function $\gr^{*2}$ is compactly supported and $R$ is bounded. Thus  we can apply the dominated convergence theorem to obtain
\[\int S^0_{\gep}(R)(t,x)\gr^{*2}(t,x)dtdx\rightarrow R(0,0) \int \gr^{*2}(t,x)dtdx =0\,.\]
The last equality holds because the function $R$ satisfies trivially $R(0,0)=0$. Passing to the identity (A.3), we rewrite $C^2_{\gep}$ as
\[C^2_{\gep}=\int_{\bR^2} (K_x*\gr_{\gep})^2(z)dz = \frac{1}{\gep} \int \gep(K_x*\gr_{\gep})^2(z)dz\,. \]
Let express $\int\gep(K_x*\gr_{\gep})^2(z)dz $ in terms of  $S^2_{\gep}(K_x)$ and of $\gr$. By applying a standard change of variable in the integrals we get
\[(K_x*\gr_{\gep})(\gep^2 t,\gep x)= \int K_x( \gep^2 t-\gep^2s,\gep x-\gep y)\gr(s,y)ds dy = (S_{\gep}^0 (K_x)*\gr)( t, x)\,. \]
Therefore for any $z=(t,x)$ we write
\[(K_x*\gr_{\gep})(t,x)= (S^0_{\gep}( K_x)*\gr)\left(\frac{t}{\gep^2}, \frac{x}{\gep}\right)\,.\]
Integrating this identity on both sides and multiplying it with $\gep$, we obtain
\[\int \gep(K_x*\gr_{\gep})^2(z)dz= \gep^4\int (S_{\gep}^0 (K_x)*\gr)^2(t, x)dt dx= \int (S^2_{\gep} (K_x)*\gr)^2(t,x)dt dx \,.\]
By sending $\gep\to 0^+$ the function $S^2_{\gep} (K_x)$ becomes equal to $S^2_{\gep}(G_x)$ and, using the scaling relation $S^2_{\gep}(G_x)= G_x$, for a.e. couple of points $(t,x)$, $(s,y)$ one has 
\[S^2_{\gep} (K_x)(t-s,x-y)\gr(s,y)\rightarrow G_x(t-s,x-y)\gr(s,y)\,.\]
The function $G_x(t-s,x-y)\gr(s,y)$ is clearly integrable in both variables $(t,x)$ $(s,y)$. Therefore as a consequence of Fubini's theorem we get the convergence
\[(S^2_{\gep} (K_x)*\gr)(t,x)\rightarrow (G_x*\gr)(t,x)\quad \text{a.e.}\]
which implies trivially
\[(S^2_{\gep} (K_x)*\gr)^2(t,x)\rightarrow (G_x*\gr)^2(t,x)\quad \text{a.e.}\]
Since the function  $(G_x*\gr)^2$ is also integrable and it dominates $(S^2_{\gep} (K_x)*\gr)^2 $, we have by dominated convergence
\[ \int (S^2_{\gep} (K_x)*\gr)^2(t,x)dt dx \rightarrow \int (G_x*\gr)^2(z)dz\, .\]
Let us prove the infinitesimal behaviour of the remainder. Writing again the decomposition $G= K+R$ as explained in the Lemma \ref{KandR1}, we deduce from the Cauchy-Schwarz inequality the estimate
\[
\begin{split}
&\frac{1}{\gep}\int (G_x*\gr)^2(z)dz- \frac{1}{\gep}\int \gep(K_x*\gr_{\gep})^2(z)dz=\\&=\frac{1}{\gep}\int (S^2_{\gep}(G_x)*\gr)^2(z)-(S^2_{\gep} (K_x)*\gr)^2(z)dz\\&= \frac{1}{\gep}\int 2(S^2_{\gep}(K_x)*\gr)(z)(S^2_{\gep}(R_x)*\gr)(z)dz + \frac{1}{\gep}\int (S^2_{\gep}(R_x)*\gr)^2(z)dz \\&=  \int 2(S^2_{\gep}(K_x)*\gr)(z)(S^1_{\gep}(R_x)*\gr)(z)dz+ \gep\int (S^1_{\gep}(R_x)*\gr)^2(z)dz\\& \leq2\left(\int (S^2_{\gep}(K_x)*\gr)^2(z)dz \right)^{1/2}\left(\int (S^1_{\gep}(R_x)*\gr)^2(z)dz\right)^{1/2}+ \gep\int (S^1_{\gep}(R_x)*\gr)^2(z)dz\,.
\end{split}
\]
Using the identity $S^1_{\gep}(R_x)*\gr= S^0_{\gep}(R)*\gr_x$ and the properties of the function $R$ stated above, it is straightforward to show that
\[(S^1_{\gep}(R_x)*\gr)(z)\to 0 \quad \text{a.e.}\]
Moreover we can bound $(S^0_{\gep}(R)*\gr_x)^2$ by an integrable function. Thus we obtain
\[\int (S^1_{\gep}(R_x)*\gr)^2(z)dz\to 0\,,\]
and the identity (A.3) follows. To finally prove the identity (A.4), it is sufficient to show 
\[\int G(s,y) \gr^{*2}(s,y)ds dy=\int (G_x*\gr)^2(s,y)ds dy.\]
Starting from the identity (\ref{heat}), we convolve both sides of the equation with the function $\gr^{*2}= \gr*\gr$. Therefore for any $u\in \bR^2$ the left-hand side of (\ref{heat}) becomes
\[2(G_x*\overline{G_x})*\gr^{*2}(u)= \int \int \int 2G_x(u-v-w)G_x(-w)\gr(v-x)\gr(x) dx dvdw\,.\]
Choosing the following change of variable
\[\begin{cases}x'= x\\v'= v-x \\ w'= w+x\end{cases}\quad \begin{cases} x= x'\\v=v' +x'\\ w= w'-x'\end{cases}\quad  dx dv dw= dx' dv'dw' \,,\]
the integral becomes
\[\int \int \int 2G_x(u-v'-w')G_x(-w'+x')\gr(v')\gr(x') dx' dv'dw'\,.\]
Using the identity $\gr(x')= \gr(-x')$, the above integral equals
\[2\int \int \int G_x(u-v'-w')G_x(-w'-x')\gr(v')\gr(x') dx' dv'dw'=\]
\[=2\int(G_x*\gr)(u-w')(G_x*\gr)(-w')dw'.\]
On the other hand, the right-hand side of (\ref{heat}) convolved with $\gr^{*2} $ gives
\[\int G(u-w) \gr^{*2}(w)dw+ \int G(w-u) \gr^{*2}(w)dw= \]
\[\int G(u-w) \gr^{*2}(w)dw+ \int G(-u-w) \gr^{*2}(w)dw=(G*\gr^{*2})(u)+ (G*\gr^{*2})(-u)\,.\]
Evaluating both sides in $u=0$, we finally conclude.
\end{proof}

\bibliographystyle{abbrv}
\bibliography{bibliography}

\end{document}